\documentclass[11pt]{article}
\usepackage{verbatim}
\usepackage{fancyhdr}
\usepackage{graphicx}
\usepackage{amsmath,amsfonts,amssymb,amsthm}
\usepackage{subfigure}
\usepackage{color}
\usepackage{bm}
\usepackage{comment}
\usepackage{multirow}
\usepackage{enumerate}
\usepackage{epstopdf}

\usepackage{setspace}

\newtheorem{lem}{Lemma}[section]
\newtheorem{Thm}{Theorem}[section]
\newtheorem{Pro}{Proposition}[section]
\newtheorem{rem}{Remark}[section]
\newtheorem{coro}{Corollary}[section]
\newtheorem{exmp}{Example}[section]
\newtheorem{Def}{Definition}[section]
\newtheorem{alg}{Algorithm}[section]
\numberwithin{figure}{section}

\newcommand{\R}{\mathbb{R}}

\newcommand{\cM}{\mathcal{M}}

\newcommand{\dist}{{\rm dist}}
\newcommand{\Ga}{\Gamma}
\newcommand{\Om}{\Omega}
\newcommand{\na}{\nabla}

\newcommand{\de}{\delta}
\newcommand{\lam}{\lambda}
\newcommand{\supp}{{\rm supp}}
\newcommand{\cT}{\mathcal{T}}
\newcommand{\cE}{\mathcal{E}}
\newcommand{\bfx}{\mathbf{x}}

\newcommand{\bfy}{\mathbf{Y}}
\newcommand{\bfr}{\mathbf{R}}
\newcommand{\bfd}{\mathbb{D}}
\newcommand{\bfu}{\mathbf{U}}

\newcommand{\bfQ}{\mathbf{Q}}
\newcommand{\bfm}{\mathbb{M}_{\rho,c}}
\newcommand{\bff}{\mathbf{F}}

\newcommand{\bfq}{\mathbf{q}}
\newcommand{\bms}{\bm{\sigma}}
\newcommand{\bmp}{\bm{\psi}}
\newcommand{\E}{\mathcal{E}}

\newcommand{\cat}{\mathcal{T}}
\newcommand{\fespace}{X_p^0(\mathcal{M})}
\newcommand{\fespaceq}{\mathbf{W}_p(\mathcal{M})}
\newcommand{\cam}{\mathcal{M}}
\newcommand{\bM}{\mathbb{M}}
\newcommand{\ba}{\mathbf{a}}
\newcommand{\GaK}{\Gamma_{\!\!K}}
\renewcommand{\div}{{\rm div\,}}
\newcommand{\curl}{{\bf curl\,}}

\newcommand{\bmu}{\bm{u}}
\newcommand{\bmP}{\bm{P}}
\newcommand{\cS}{\mathcal{S}}
\newcommand{\cP}{\mathcal{P}}

\newcommand{\rev}{}

\allowdisplaybreaks[4]

\newcommand{\lj}{[{\hskip -1.5pt} [}
\newcommand{\rj}{]{\hskip -1.5pt} ]}

\newcommand{\la}{\langle}
\newcommand{\ra}{\rangle}

\newcommand{\pa}{\partial}

\newcommand{\be}{\begin{eqnarray}}
\newcommand{\ee}{\end{eqnarray}}
\newcommand{\ben}{\begin{eqnarray*}}
\newcommand{\een}{\end{eqnarray*}}
\newcommand{\nn}{\nonumber}

\title{A high order explicit time finite element method for the acoustic wave equation with discontinuous coefficients\footnotemark[1]}
\author{
Zhiming Chen\footnotemark[2]
\and
Yong Liu\footnotemark[3]
\and
Xueshuang Xiang\footnotemark[4]
}

\date{}

\begin{document}
\maketitle

\renewcommand{\thefootnote}{\fnsymbol{footnote}}
\footnotetext[1]{The work was supported in part by China National Key Technologies R\&D Program under the grant 2019YFA0709600, China NSF under the grant 118311061, 12288201, 12201621, and the fellowship of China Postdoctoral Science Foundation No. 2020TQ0343.}
\footnotetext[2]{LSEC, Institute of Computational Mathematics,
Academy of Mathematics and System Sciences and School of Mathematical Science, University of
Chinese Academy of Sciences, Chinese Academy of Sciences,
Beijing 100190, China. E-mail: zmchen@lsec.cc.ac.cn}
\footnotetext[3]{LSEC, Institute of Computational Mathematics, Academy of Mathematics and Systems Science, Chinese Academy of Sciences,
Beijing 100190, P.R. China. E-mail: yongliu@lsec.cc.ac.cn. Corresponding author.}
\footnotetext[4]{Qian Xuesen Laboratory of Space Technology, China Academy of Space Technology, Beijing 100194, P.R. China. E-mail: xiangxueshuang@qxslab.cn}

\begin{center}
\small
\begin{minipage}{0.9\textwidth}
\textbf{Abstract.}
In this paper, we propose a novel high order unfitted finite element method on Cartesian meshes for solving the acoustic wave equation with discontinuous coefficients having complex interface geometry. The unfitted finite element method does not require any penalty to achieve optimal convergence. We also introduce a new explicit time discretization method for the ODE system resulting from the spatial discretization of the wave equation. The strong stability and optimal $hp$-version error estimates both in time and space are established. Numerical examples confirm our theoretical results.

\medskip
\textbf{Key words.}
Explicit time discretization; strong stability; unfitted finite element; $hp$ error estimates.
\medskip

\textbf{AMS classification}.
65M12, 65M60
\end{minipage}
\end{center}
\setlength{\parindent}{2em}

\section{Introduction} \label{intro}

The wave equation is a fundamental equation in mathematical physics describing the phenomena of wave propagation. It finds diverse applications in science and engineering, including geoscience, petroleum engineering, and telecommunication (see \cite{KAMPANIS2008, Tromp} and the references therein). Let {\rev{$\Omega
\subset \mathbb{R}^2$}} be a bounded Lipschitz domain and $T>0$ be the length of the time interval.
We consider in this paper the acoustic wave equation
\begin{align}\label{modelproblem}
\left\{\begin{aligned}
&\frac{1}{\rho c^2} \pa_tu= \mbox{div}\, \mathbf{q}+f,\ \  \rho\pa_t\mathbf{q}=\nabla u \ \ \ \ \mbox{in }\Omega\times (0,T),\\
&[\![u]\!]=0,\quad [\![\mathbf{q}\cdot \mathbf{n}]\!]=0\ \ \ \ \text{ on } \Gamma\times (0,T),\\
&u=0\quad \text{ on }\partial\Omega\times(0,T),\\
&u(\mathbf{x},0)=u_0(\mathbf{x}),\ \  \mathbf{q}(\mathbf{x},0)=\mathbf{q}_0(\mathbf{x})\ \ \ \ \mbox{in }\Omega,
\end{aligned}\right.
\end{align}
where $u$ is the pressure, $\mathbf{q}$ is the speed of the displacement in the medium, and $f$ is the source. The domain $\Omega$ is assumed to be divided by a $C^2$-smooth interface $\Gamma$ into two nonintersecting subdomains such that
$\Om=\Om_1\cup\Ga\cup\Om_2$ and $\Om_1\subset\bar\Om_1\subset\Om$.
For simplicity, we assume that the density of the medium $\rho$ and the speed of the propagation of the wave $c$ are piecewise constants, namely,
\begin{align*}
&\rho=\rho_1\chi_{\Omega_1}+\rho_2\chi_{\Omega_2}, \ \ c=c_1\chi_{\Omega_1}+c_2\chi_{\Omega_2},
\end{align*}
where for $i=1,2$, $\rho_i,c_i$ are positive constants and $\chi_{\Omega_i}$ denotes the characteristic function of $\Omega_i$. Here $\mathbf{n}$ is the unit outer normal to $\Omega_1$, and $[\![v]\!]|_{\Gamma}:=v|_{\Omega_1}-v|_{\Omega_2}$ denotes the jump of a function $v$ across the interface $\Gamma$. 

There exists a large literature on numerical methods for solving the wave equation on conforming quadrilateral/hexahedral or triangular/tetrahedral meshes for which we refer to the monograph \cite{Cohen2002higher} and \cite{French1996CTG, Wu2021SINUM} for the construction of the algorithms and the finite element error analysis. Local discontinuous Galerkin (DG) methods for the wave equation are studied in \cite{csxjcp2014, SX2021mc}. Optimal error estimates for sufficiently smooth solutions are proved in \cite{csxjcp2014} on Cartesian meshes without using the penalty and in \cite{SX2021mc} on unstructured meshes by adding appropriate penalty terms, which leads to however a dissipative method. In \cite{Monk2014JSC}, both dissipative and non-dissipative variants of the hybridizable DG methods are proposed, where it is shown that the dissipative method has the optimal error estimate and the non-dissipative method whose numerical flux includes the time derivative of the pressure is sub-optimal.

In order to deal with an arbitrarily shaped interface where the coefficients of the partial differential equations are discontinuous, immersed or unfitted mesh methods are developed to avoid expensive work of mesh generation using body-fitted methods in e.g.,  \cite{Babuska1983, Chen1998}. For acoustic wave equations with discontinuous coefficients, a second order immersed interface method on Cartesian meshes with suitable modification of the finite difference stencil near the interface is developed in \cite{Jeong2021}. In \cite{Adjerid2019}, second and third order immersed DG methods are proposed which design polynomial shape functions to approximately satisfy the interface conditions. In \cite{Sticko2019, Schoeder2020} high order cut finite elements for solving the wave equation are studied. The small cut cell problem, that is, the small intersection of the interface and the elements of the mesh can always occur, is treated by adding penalty terms of jumps of high order derivatives over interior sides of cut elements in \cite{Sticko2019} and by the approach of cell merging in \cite{Schoeder2020} following an idea in \cite{Johansson2013} for elliptic equations. We remark that appropriate penalties are crucial in \cite{Adjerid2019, Sticko2019, Schoeder2020} in designing DG methods for solving the wave equation.

The first objective of this paper is to propose an arbitrarily high order unfitted finite element method for solving \eqref{modelproblem} without adding any penalty terms. Our method is defined on an induced mesh from a Cartesian mesh with hanging nodes by merging small interface elements with their neighboring elements so that the elements in the induced mesh are large with respect to both subdomains $\Om_1,\Om_2$. A reliable algorithm to generate the induced mesh from the Cartesian mesh with hanging nodes is constructed in \cite{ChenLiu2022} for any $C^2$-smooth interface. We will show in this paper that the same induced mesh also allows us to define a new unfitted finite element space which is conforming in each subdomain $\Om_1,\Om_2$. This new finite element space, together with a new observation of DG methods (see \eqref{keyp} below) and the lifted regular decomposition theorem of vector fields, leads to optimal energy error estimates of our semi-discrete unfitted finite element method without resorting to the penalties. This new piecewise conforming unfitted finite element space, which is less expensive than the standard unfitted finite element space first introduced in the seminal work \cite{Hansbo}, is of independent interest. We refer to \cite{CLX2020, ChenLiu2022} for more references on the development of unfitted finite element methods in the setting of elliptic equations. We remark that the theoretical results in this paper can be extended to the three-dimensional case, but the reliable algorithm for constructing cubic macro-elements is more challenging. We leave the extension to solve the three-dimensional wave problems in a future work.

After spatial discretization, we obtain a linear ODE system of the form
\begin{align}\label{ODE_model}
\frac{d}{dt}\mathbf{Y}=\mathbb{D}\mathbf{Y}+\bfr,
\end{align}
where $\mathbf{Y},\bfr\in \mathbb{R}^M$, and $\mathbb{D}\in\R^{M\times M}$ is a constant matrix. Here $M$ is the number of degrees of freedom for the spatial discretization.
Explicit Runge-Kutta (RK) methods have been successfully used for time integration for hyperbolic conservation laws when coupled with the  DG scheme in space \cite{cockburn1989tvb}.
 In \cite{Sun2019ssp}, the strong stability of explicit RK methods is studied for semi-negative autonomous linear systems, that is, $\mathbb{D}^T\mathbb{H}+\mathbb{H}\mathbb{D}$ is semi-negative definite for some symmetric positive definite matrix $\mathbb{H}\in\R^{M\times M}$. It is proved in \cite{Sun2019ssp} that for $r\ge 1$, the standard $r$ stage $r$ order RK methods are strongly stable when $r=3\, ({\rm mod}\,4)$, not strongly stable when $r=1,2\,({\rm mod}\,4)$. When $r=0\, ({\rm mod }\,4)$, the $r$ stage $r$ order RK method is strongly stable under the condition $\mathbb{D}^T\mathbb{H}+\mathbb{H}\mathbb{D}=\mathbb{O}$, where $\mathbb{O}\in \R^{M\times M}$ is the zero matrix.

The second objective of this paper is to propose a strongly stable and arbitrarily high order explicit time discretization for \eqref{ODE_model} by using the property $\mathbb{D}+\mathbb{D}^T=\mathbb{O}$ which results from the duality of the DG spatial discretization of gradient and divergence operators. The scheme is formulated in the finite element framework, that is, we find a continuous piecewise polynomial function of time to discretize \eqref{ODE_model}. This allows us to prove the stability and $hp$-version error estimates under explicit CFL bounds in the whole time interval instead of only at time discretization nodes. We also introduce an efficient finite difference implementation of our finite element time scheme based on Legendre polynomial basis functions.


The layout of this paper is as follows. In section \ref{secfem} we introduce the semi-discrete unfitted finite element method and prove the energy conservation property and $hp$ optimal error estimates. In section \ref{sectimedisc} we introduce the explicit time discretization for \eqref{ODE_model} and prove the {\rev{strong stability property}} and the error estimates under suitable CFL conditions. In section \ref{secfull} we consider the {\rev{full discretization scheme}} for \eqref{modelproblem} and prove $hp$-version error estimates. 
In section \ref{secnum} we provide some numerical examples to verify our theoretical results. In section 6 we show the compatibility property of the induced mesh by the merging algorithm in \cite{ChenLiu2022}.

\setcounter{equation}{0}
\section{The semi-discrete unfitted finite element method}\label{secfem}

In this section we first recall some elements of the unfitted finite element method in the framework of Chen et al \cite{CLX2020} in subsection 2.1. In subsection 2.2, we introduce a new unfitted finite element space which is conforming in each subdomain $\Om_1,\Om_2$.   We propose the semi-discrete unfitted finite element method for solving \eqref{modelproblem} and prove the optimal energy error estimates in subsection 2.3.

\subsection{The induced mesh}

Let $\Om$ be the union of rectangles so that it can be covered by a Cartesian mesh $\cT_0$. The extension to the case when $\Om$ is a smooth domain can be done in a straightforward way by using the ideas in this paper. The general case when $\Om$ is a Lipschitz domain with piecewise smooth boundary can be studied by combining the ideas in \cite{CLX2020} on the large element and interface deviation for piecewise smooth interfaces with the extension of the merging algorithm in Chen and Liu \cite{ChenLiu2022} for $C^2$ smooth interfaces to deal with the piecewise smooth boundaries, for which we will pursue in a future work.
We remark that the special case when $\Om$ is a rectangle is of particular interests when solving the acoustic scattering problems by the method of perfectly matched layer (PML) to truncate the unbounded domains \cite{ChenWu2012}.

Let $\mathcal{T}$ be a Cartesian partition of the domain $\Omega$ obtained by quad refinements of $\cT_0$ with possible local refinements and hanging nodes. We assume each element $K\in\cat$ is intersected by the interface $\Gamma$ at most twice at different (open) sides. From $\mathcal{T}$ we want to construct an induced mesh $\cM$ which avoids possible small intersections of the interface and the elements of the mesh. We start by defining the concept of large element.

\begin{Def}\label{def:2.1}
For $i=1,2$, an element $K\in \cat$, is called a large element with respect to $\Omega_i$ if $K \subset \Omega_i$; or $K \in \cat^\Gamma:=\{K\in\cat:K\cap \Gamma \not= \emptyset\}$ for which there exists a $\delta_0\in (0,1/2)$ such $|e\cap\Omega_i|\geq \delta_0 |e|$ for each side $e$ of $K$ having nonempty intersection with $\Omega_i$.
\end{Def}

When the element $K\in\cat^\Ga$ is not large with respect to both $\Omega_i$, $i=1,2$, we make the following assumption as in \cite{CLX2020}.

\medskip\noindent
{\bf{Assumption (H1)}}: For each $K\in \cat^\Gamma$, there exists a rectangular macro-element $M(K)$ which is a union of $K$ and its surrounding element (or elements) such that $M(K)$ is large with respect to both $\Omega_1$, $\Omega_2$. We assume $h_{M(K)}\leq C_0 h_K$ for some fixed constant $C_0$.
\medskip

This assumption can always be satisfied by using the idea of cell merging originated in Johansson and Larson \cite{Johansson2013}. We refer to \cite{ChenLiu2022} for a reliable algorithm to satisfy this assumption when the interface is $C^2$ smooth. Set $M(K)=K$ if $K\in\cat^\Gamma$ and $K$ is large with respect to both $\Omega_1,\Omega_2$. Then the induced mesh
$\cam=\{M(K):K\in\cat^\Gamma\}\cup\{K\in\cat:K\subset\Omega_i,i=1,2,K\not\subset M(K')\mbox{ for some }K'\in\cat^\Gamma\}$
satisfies the desired property that the elements in $\cam$ are large with respect to both domains $\Omega_1$, $\Omega_2$ and the interface $\Ga$ intersects the boundary of element $K\in \cam$ also twice at different sides. We denote $\cM={\rm Induced}(\cT)$. We require the following compatibility assumption on the induced mesh.

\medskip\noindent
{\bf Assumption (H2)}: For any $e=\pa K\cap\pa K'$, $K,K'\in\cM$, let $f,f'$ be respectively the sides of $K, K'$ including $e$, then either (1) $f\subset f'$ or $f'\subset f$; or (2) $e\cap\Ga\not=\emptyset$.
\medskip

The first condition in the assumption (H2) is standard in the literature in order to define conforming finite element methods on meshes with possible hanging nodes (see, e.g., Bonito et al \cite[\S 4.1]{Bo16}). The second condition when the interface is present is new, which is important for us to define a new unfitted finite element space which is conforming in each domain $\Om_1,\Om_2$. In the appendix of this paper we will show that the induced mesh obtained by the merging algorithm in \cite[Algorithm 6]{CLX2020} will satisfy the assumption (H2).

For any $K\in\cM$, we denote $h_K$ the diameter of $K$. For $K\in\cM^\Ga:=\{K\in\cM:K\cap\Ga\not=\emptyset\}$, denote $\Ga_K=K\cap\Ga$ and $\Gamma_K^h$ the (open) straight segment connecting the two intersection points of $\Gamma$ and $\partial K$. For $i=1,2$, let $K_i=K\cap\Om_i$, $K_i^h$ the polygon whose vertices are the vertex (vertices) of $K$ inside $\Om_i$ and the endpoints of $\Ga_K^h$, and $A_K^i$ the vertex of $K$ {\rev{within $\Omega_i$}} which has the maximum distance to $\Gamma_K^h$.
The following lemma shows that $K_i^h$ is the union of shape regular triangles as the consequence of $K$ being a large element. The lemma can be easily proved and we omit the details.

\begin{lem}\label{lem:2.1new}
Let $K\in\cam^\Ga$. Then for $i=1,2$, $K^h_i$ is the union of triangles $K^h_{ij}$, $j=1,\cdots,J_i^K$, $1\le J_i^K\le 3$, such that $K_{ij}^h$ has one vertex at $A_K^i$ and the other two vertices being the endpoints of $\Ga_K^h$ or the vertex of $K$ in $\Om_i$.  Moreover, $K_{ij}^h$, $j=1,\cdots,J_i^K$, $i=1,2$, is shape regular in the sense that the radius of the inscribed circle of $K_{ij}^h$ is bounded below by $c_0h_K$ for some constant $c_0>0$ depending only on $\de_0$ in Definition \ref{def:2.1}. We always set $K_{i1}^h$ the triangle with $\Ga_K^h$ as one of its sides, see Fig.\ref{fig:2.1}.
\end{lem}

\begin{figure}\centering
\includegraphics[width=0.6\textwidth]{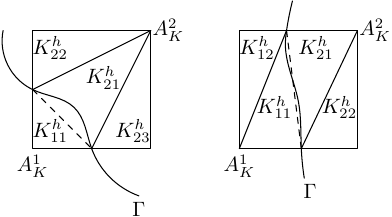}\\
\caption{Illustration of $K_{ij}^h$, $\Gamma_K^{h}$ is denoted by the dashed line.}\label{fig:2.1}
\end{figure}

We now recall an important concept of $K$-mesh in Babu\v{s}ka and Miller \cite{Babu87} for the Cartesian mesh $\cT$ having hanging nodes. Let $\mathcal{N}^0$ be the set of conforming nodes of $\cT$ which are the vertices of the elements either locating on the boundary $\pa\Om$ or shared by the four elements to which they belong. For each conforming node $P$, we denote $\psi_P\in H^1(\Om)$ which is bilinear in each element and satisfies $\psi_P(Q)=\delta_{PQ}$ for any $Q\in\mathcal{N}^0$. Here $\de_{PQ}$ is the Kronecker delta. We impose the following $K$-mesh condition on the mesh $\cT$.

\medskip\noindent
{\bf{Assumption (H3)}}: There exists a constant $C>0$ uniform on the level of discretizations of $\cT$ such that for any conforming node $P\in\mathcal{N}^0$, ${\rm diam}(\supp(\psi_P))\le C\min_{K\in\cT_P}h_K$, where $\cT_P=\{K\in\cT:K\subset\supp(\psi_P)\}$.
\medskip

Further properties of $K$-meshes can be found in \cite{Babu87}. A refinement algorithm to enforce the assumption (H3) can be found in Bonito and Nochetto \cite[\S 6]{Bonito}.

Let $\E=\E^{\rm side}\cup\E^{\Gamma}\cup\E^{\rm bdy}$, where $\E^{\rm side}:=\left\{e=\partial K \cap \partial K': K,K' \in \cam \right\}$, $\E^\Gamma:=\{\GaK=\Gamma \cap K: K\in \cam \}$, and $\E^{\rm bdy}:=\left\{e=\partial K \cap \partial \Omega: K \in \cam \right\}$. Since hanging nodes are allowed, $e\in \E^{\rm side}$ can be part of a side of an adjacent element. For any subset $\widehat{\cam}\subset\cam$ and $\widehat{\E}\subset\E$, we use the notation
\begin{align*}
(u,v)_{\widehat{\cam}}=\sum_{K \in \widehat{\cam}}(u,v)_{K}, \ \ \langle u,v\rangle_{\widehat{\E}}=\sum_{e \in \widehat{\E}}\langle u,v\rangle_{e},
\end{align*}
where $(\cdot,\cdot)_K$ and $\la\cdot,\cdot\ra_e$ denote the inner product of $L^2(K)$ and $L^2(e)$, respectively.

For any $e\in \E$, we fix a unit normal vector $\mathbf{n}_e$ of $e$ with the convention that $\mathbf{n}_e$ is the unit outer normal to $\pa\Om_1$ if $e\in\E^\Ga$ and to $\partial \Omega$ if $e \in \E^{\rm bdy}$. Define the normal function $\mathbf{n}|_e=\mathbf{n}_e\ \forall e\in \E$. For any $v\in H^1(\cM):=\{v_1\chi_{K_1}+v_2\chi_{K_2}:v_1,v_2\in H^1(K), K\in\cM\}$, we define the jump operator of $v$ across $e$:
\begin{align*}
[\![v]\!]|_e:=v^{-} -v^{+}\ \  \forall e \in \E^{\rm side}\cup\E^{\Gamma},\ \ \ \
[\![v]\!]|_e:=v^{-}\ \ \forall e \in \E^{\rm bdy},
\end{align*}
where $v^{\pm}(\mathbf{x}):=\lim_{\varepsilon\rightarrow 0^+} v(\mathbf{x}\pm\varepsilon \mathbf{n}_e)$ for any $\mathbf{x}\in e$. The mesh function $h|_e=(h_K+h_{K'})/2$ if $e=\pa K\cap\pa K'\in\cE^{\rm side}$ and $h|_e=h_K$ if $e=K\cap\Ga\in\cE^\Ga$ or $e=\pa K\cap\pa\Om\in\cE^{\rm bdy}$ for some $K\in\cM$.

For any $p\ge 1$ and any Lipschitz domain $D\subset\R^d$, $d\ge 1$, we denote $Q_p(D)$ the set of polynomials of degree at most $p$ in each variable. The following lemma on the local smoothing operator on $K$-meshes is proved in \cite[Lemma 3.2]{CLX2020}.

\begin{lem}\label{lem:local_smoothing}
There exists an interpolation operator $\pi_h$: $\mathbb{V}_p(\cT)\rightarrow \mathbb{V}_p(\cT)\cap H^1(\Omega)$ such that for any $v\in \mathbb{V}_p(\cT)$,
\ben
& &\|v-\pi_h v\|_{L^2(K)} \leq C \|p^{-1}h^{1/2}\lj v\rj\|_{L^2(\sigma(K))},\\
& &\|\nabla (v-\pi_h v)\|_{L^2(K)} \leq C \| ph^{-1/2}\lj v\rj\|_{L^2(\sigma(K))},
\een
where $\mathbb{V}_p(\cT)=\prod_{K\in \cT}Q_p(K)$ and $\sigma(K)=\{e\in \E^{side}: e \subset \tilde{\omega}(K)\}$, $\tilde{\omega}(K)$ is a set of elements including $K$ such that  $\text{diam}(\tilde{\omega}(K))\leq C h_K$. The constant $C$ is independent of $h_K$, $p$. Moreover, $\pi_h v \in H_0^1(\Omega)$ if $v=0$ on $\partial \Omega$.
\end{lem}

Since the induced mesh $\cM={\rm Induced}(\cT)$ is obtained by merging some of the elements of $\cT$, $\mathbb{V}_p(\cam)\subset\mathbb{V}_p(\cT)$, Lemma \ref{lem:local_smoothing} is valid for any functions in $v\in\mathbb{V}_p(\cam):=\Pi_{K\in\cam}Q_p(K)$.

\subsection{Unfitted finite element spaces}

In this subsection we introduce the scalar and vector unfitted finite element spaces on the induced mesh $\cam$ which are motivated by the idea of ``doubling of unknowns'' in Hansbo and Hansbo \cite{Hansbo}. For any $p,q\ge 1$ and any Lipschitz domain $D\subset\R^d$, $d\ge 1$, the space $P_p(D)$ denotes the space of polynomials of degree at most $p$ in $D$, and $Q_{p,q}(D)$ denotes the space of polynomials of degree at most $p$ for the first variable and $q$ for the second variable in $D$. For any $K\in\cM^\Ga$, in the notation in Lemma \ref{lem:2.1new}, we know that
\ben
K=K_1^h\cup\Ga_K^h\cup K_2^h,\ \ \bar K_i^h=\cup^{J_i^K}_{j=1}\bar K_{ij}^h,\ \ i=1,2.
\een
From $K_{ij}^h$ we define the curved element $\widetilde K_{ij}^h$ by
\ben
\widetilde K_{i1}^h=(K_i\cap  K_{i1}^h)\cup(K_i\backslash \bar K_{i}^h),\ \ \widetilde K_{ij}^h=K_i\cap K_{ij}^h,\ \ j=2,\cdots,J_i^K.
\een
Then we have
\ben
K=K_1\cup\Ga_K\cup K_2,\ \ K_i=\mbox{\rm the interior of }\overline{\cup^{J_i^K}_{j=1}{\widetilde K}_{ij}^h},\ \ i=1,2.
\een
For $i=1,2$, let $\cM_i$ be the union of elements of $\cM$ which is inside $\Om_i$ and all curved triangles $\widetilde K_{ij}^h$, $j=1,\cdots,J_i^K$, for all $K\in\cM^\Ga$. Then $\cM_i$ is a mixed rectangular and curved triangular mesh of $\Om_i$. We have the following compatibility property of the mesh.

\begin{lem}\label{lem:H2}
For $i=1,2$, let $e=\pa K\cap\pa K'$, $K,K'\in\cM_i$, and $f,f'$ be respectively the side of $K,K'$ including $e$. Then either $f\subset f'$ or $f'\subset f$.
\end{lem}

\begin{proof}
The lemma is obvious if $e$ is the common side of two triangles $\widetilde K_{ij}^h$, $j=1,\cdots,J_i^K$, inside some element $K\in\cM^\Ga$. Thus we only need to consider the case when $e$ is part of the common boundary $\tilde e$ of two elements $M,M'\in\cM$ so that $K\subset M,K'\subset M'$. If $\tilde e\cap\Ga=\emptyset$, then the lemma follows from the Assumption (H2). If $\tilde e\cap\Ga\not=\emptyset$, then $M,M'\in\cM^\Ga$, $M$ consists of triangular elements $\widetilde M_{ij}^h$, $j=1,\cdots,J_i^M$, and $M'$ consists of triangular elements $\widetilde M^{'h}_{ij}$, $j=1,\cdots,J_i^{M'}$, see Fig.\ref{fig:2.2} (left). It is clear that $K,K'$ are one of the elements $\{\widetilde M_{ij}^h\}^{J_i^M}_{j=1}$, $\{\widetilde M^{'h}_{ij}\}^{J_i^{M'}}_{j=1}$, respectively. The lemma now follows easily.
\end{proof}

\begin{figure}\centering
\includegraphics[width=0.7\textwidth]{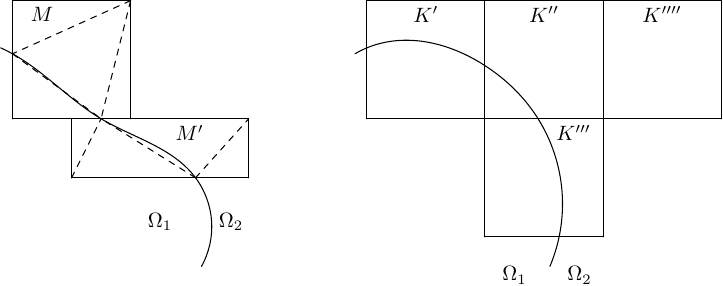}
\caption{The interface intersects two macro-elements $M,M'$ (left) and three elements $K',K'',K'''$ (right).}\label{fig:2.2}
\end{figure}

For any $K\in\cM^\Ga$, we define the interface finite element spaces
\ben
W_p(K)=\{\varphi: \varphi|_{\widetilde K_{ij}^h}\in P_p(\widetilde K_{ij}^h),\ i=1,2,\ j=1,\cdots,J_i^K\},
\een
and $X_p(K)=W_p(K)\cap H^1(K_1\cup K_2)$.
Notice that functions in $W_p(K)$ are piecewise polynomials which are discontinuous in $K$. The functions in $X_p(K)$ are, on the other hand, conforming in each $K_i, i=1,2$.

Now we define the following unfitted finite element spaces
\ben
& &X_p(\mathcal{M}):=\{v\in H^1(\Om_1\cup\Om_2): v|_K\in X_p(K)\ \ \forall K\in\cam^\Ga,\\
& &\hskip4.8cm v|_K\in Q_p(K)\ \ \forall K\in\cam\backslash\cam^\Ga\},\\
& &\mathbf{W}_p(\mathcal{M}):=\{\bmp\in [L^2(\Om)]^2: \bmp|_K \in [W_p(K)]^2 \ \ \forall K\in \mathcal{M}^{\Gamma}, \\
& &\hskip4.3cm \bmp|_K\in Q_{p-1,p}(K)\times Q_{p,p-1}(K)\ \ \forall K\in\cam\backslash\cam^\Ga\}.
\een
Let $X_p^0(\cM)=X_p(\cM)\cap H^1_0(\Om_1\cup\Om_2)$, where $H^1_0(\Om_1\cup\Om_2)=\{v\in H^1(\Om_1\cup\Om_2): v=0 \mbox{ on } \pa\Om\}$. Recall our convention that $\Om_1\subset\bar\Om_1\subset\Om$ so that $\pa\Om=\pa\Om_2\backslash\bar\Ga$.

Our finite element space $X_p(K)$ for the interface elements is different from the one in \cite{Hansbo} and also used in \cite{CLX2020, ChenLiu2022} in which the finite element functions are piecewise in $Q_p(K)$. The piecewise $Q_p$ unfitted finite element functions may always be discontinuous in each domain $\Om_1,\Om_2$. For example, when $p=1$, the $Q_1$ functions in the curved pentagon $K''_2=K''\cap\Om_2$ have only $4$ degrees of freedom, which cannot  be conforming with all $Q_1$ functions in $K'\cap\Om_2$, $K'''\cap\Om_2$, and $K''''$, see Fig.\ref{fig:2.2} (right).

The space $X_p(K)$ makes it possible to have the functions in $X_p(\cam)$ conforming in each subdomain $\Om_i$, $i=1,2$, which is crucial for us to prove the optimal energy error estimates in the next subsection. We also note that the space $\mathbf{W}_p(\mathcal{M})$ is chosen such that $\nabla_h v \in \mathbf{W}_p(\mathcal{M})$ for any $v \in X_p(\mathcal{M})$, where for any $v\in X_p(\cM)\subset H^1(\Om_1\cup\Om_2)$, we have $v=v_1\chi_{\Om_1}+v_2\chi_{\Om_2}$ with $v_1\in H^1(\Om_1), v_2\in H^1(\Om_2)$, we denote $\na_h v=\na v_1\chi_{\Om_1}+\na v_2\chi_{\Om_2}$.

To proceed, we recall the concept of interface deviation introduced in \cite{CLX2020} in order to quantify how the mesh resolves the geometry of the interface.

\begin{Def}\label{def:2.2}
For any $K\in \cam^\Gamma$, the interface deviation $\eta_K$ is defined as
\begin{align*}
\eta_K=\max_{i=1,2}\frac{\mbox{\rm dist}_{\rm H}(\GaK,\GaK^h)}{\mbox{\rm dist}(A_K^i,\GaK^h)},
\end{align*}
where $\displaystyle \mbox{\rm dist}_{\rm H}(\GaK,\GaK^h)=\max_{x\in\GaK}(\min_{y\in\GaK^h}|x-y|)$ and $\displaystyle \mbox{\rm dist}(A_K^i,\GaK^h)={\rev{\min_{y\in\GaK^h}|A_K^i-y|}}$.
\end{Def}

It is easy to show that if the interface $\Gamma$ is $C^2$-smooth, there exists a constant $C$ which is independent of $h_K$ such that $\eta_K\leq C h_K$. Thus the following assumption is not very restrictive in practical applications.

\medskip\noindent
{\bf{Assumption (H4)}}: For any $K \in \mathcal{M}^\Gamma$, $\eta_K\leq 1/2$.
\medskip

The interface deviation is crucial for us to show the inverse estimates on curved domains which play an important role in our study of unfitted finite element methods. We start
from the following one dimensional domain inverse estimate proved in \cite[Lemma 2.3]{CLX2020}
\be\label{1D}
\|g\|_{L^2(I_\lam\backslash\bar I)}^2\le \frac 12\left[(\lam+\sqrt{\lam^2-1})^{2p+1}-1\right]\|g\|_{L^2(I)}^2\ \ \forall g\in Q_p(I).
\ee
where $I=(-1,1)$, $I_\lam=(-\lam,\lam)$, $\lam>1$. We remark that the growing factor $(\lam+\sqrt{\lam^2-1})^{2p+1}$ in above bound is sharp which is attained by the Chebyshev polynomials $C_n(t)$, $n\ge 0$. It is well-known (e.g., DeVore and Lorentz \cite[P.76]{DeVore}) that $C_n(t)=\frac 12[(t+\sqrt{t^2-1})^n+(t-\sqrt{t^2-1})^n]$, $n\ge 0$.

By \eqref{1D}, one can prove the following two dimensional domain inverse estimate by the same argument as that in \cite[Lemma 2.4]{CLX2020}, where domain inverse estimates for $Q_p(\Delta)$ functions are proved. Here we omit the details.

\begin{lem}\label{lem:2.2new}
Let $\Delta$ be a triangle with vertices $A=(a_1,a_2)^T$, $B=(0,0)^T$, $C=(c_1,0)^T$, where $a_2,c_1>0$. Let $\de\in (0,a_2)$ and $\Delta_\de=\{x\in\Delta:\dist(x,BC)>\delta\}$, where $\dist(x,BC)=\min_{y\in BC}|x-y|$. Then, we have
\ben
\|v\|_{L^2(\Delta)}\le\mathsf{T}\left(\frac{1+\de a_2^{-1}}{1-\de a_2^{-1}}\right)^{p+3/2}\|v\|_{L^2(\Delta_\de)}\ \ \forall v\in P_p(\Delta),
\een
where $\mathsf{T}=t+\sqrt{t^2-1}\ \ \forall t\ge 1$.
\end{lem}


Set
\begin{align}\label{thetak}
\Theta_K=\left\{\begin{array}{cc}
\mathsf{T}(\frac{1+3\eta_K}{1-\eta_K})^{2p+3} & \text{ if } K \in \cam^\Gamma,\\
\\
1 & \text{ otherwise },
\end{array}\right.
\end{align}
We have the following inverse estimates on curved domains which is an adaption of \cite[Lemma 2.8]{CLX2020} to the new unfitted finite element spaces in this paper.

\begin{lem}\label{lem:2.1}
Let $K\in \cam^\Gamma$. Then there exists a constant $C$ independent of $h_K,p$, and $\eta_K$ such that for $i=1,2$,
\begin{align*}
&\|\nabla v\|_{L^2(K_i)}\leq C p^2h_K^{-1}\Theta_K^{1/2}\|v\|_{L^2(K_i)} \quad \forall v \in X_p(K),\\
&\|v\|_{L^2(\partial K_i)}\leq C ph_K^{-1/2}\Theta_K^{1/2}\|v\|_{L^2(K_i)} \quad \forall v \in X_p(K).
\end{align*}
\end{lem}

We also refer to Massjung \cite{Massjung}, Wu and Xiao \cite{Wu}, and Cangiani et al \cite{Cangiani} for $hp$ inverse trace inequalities on star shaped elements with curved boundary. Our inverse trace inequality in Lemma \ref{lem:2.1} is independent of the local shape of the interface.

\begin{proof} The argument is similar to that in \cite[Lemma 2.8]{CLX2020}. Here we sketch a proof for the inverse estimate $\|\na v\|_{L^2(\widetilde K_{i1}^h)}$. Let $\de={\rm dist}_{\rm H}(\Ga_K,\Ga_K^h)$ and $d_i={\rm dist}(A_K^i,\Ga_K^h)$. Then $\de/d_i\le\eta_K$ by Definition \ref{def:2.2}. Let $K_{i1}^h$ be the triangle with vertices $A_K^i, B,C$. Let $\Ga_K^{h\pm\de}$ the line parallel to $\Ga_K^h=BC$ with the distance of $A_K^i$ to $\Ga_K^{h\pm\de}$ being $d_i\pm\de$.
Let $\Ga_K^{h-\de}$ and $\Ga_K^{h+\de}$ intersect the extended line $A_K^iB$ at $B',B''$ and  $A_K^iC$ at $C',C''$, respectively, see Fig.\ref{fig:lem:2.1}. Then by Lemma \ref{lem:2.2new}, we have
\ben
\|\na v\|_{L^2(\widetilde K_{i1}^h)}\le\|\na v\|_{L^2(\Delta A_K^iB''C'')}&\le&\mathsf{T}\left(\frac{1+2\de(d_i+\de)^{-1}}{1-2\de(d_i+\de)^{-1}}\right)^{p+3/2}\|v\|_{L^2(\Delta A_K^iB'C')}\\
&\le&\mathsf{T}\left(\frac{1+3\eta_K}{1-\eta_K}\right)^{p+3/2}\|v\|_{L^2(\widetilde K_{i1}^h)}.
\een
This completes the proof.
\end{proof}
\begin{figure}[h]
\centering
\includegraphics[width=0.35\textwidth]{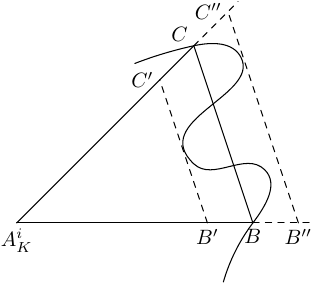}
\caption{The figure used in the proof of Lemma \ref{lem:2.1}.}\label{fig:lem:2.1}
\end{figure}

Denote
\ben
U_p(K)=\{\varphi:\varphi|_{K_{ij}^h}\in P_p(K_{ij}^h),i=1,2,j=1,\cdots,J_i^K\}.
\een
We define a mapping $\Lambda_K:U_p(K)\to W_p(K)$ such that for any $\varphi\in U_p(K)$, $i=1,2$, $j=1,\cdots,J_i^K$,
\ben
\Lambda_K(\varphi)|_{\widetilde K_{ij}^h\cap K_{ij}^h}=\varphi|_{K_{ij}^h}, \ \Lambda_K(\varphi)|_{\widetilde K_{i1}^h\backslash\bar K_{i1}^h}\mbox{ is the extension of }\varphi|_{K_{i1}^h}.
\een
One can construct finite element functions in $W_p(K)$ by using the degrees of freedom of the finite element functions in $P_p(K_{ij}^h)$, $i=1,2,j=1,\cdots,J_i^K$. It is obvious that $\Lambda_K: U_p(K)\cap H^1(K_1^h\cup K_2^h)\to X_p(K)$ and $\Lambda_K(\varphi)=\varphi$ if $\varphi\in P_p(K)$.
By Lemma \ref{lem:2.2new}, for $i=1,2$,
\be\label{z1}
\|\Lambda_K(\varphi)\|_{L^2(K_i)}\le \Theta_K^{1/2}\|\varphi\|_{L^2(K_i^h)}.
\ee
In fact, it is obvious that for $j=2,\cdots, J_i^K$, $\|\Lambda_K(\varphi)\|_{L^2(\widetilde K_{ij}^h)}\le\|\varphi\|_{L^2(K_{ij}^h)}$. For $j=1$, it follows from Lemma \ref{lem:2.2new} that $\|\Lambda_K(\varphi)\|_{L^2(\widetilde K_{i1}^h)}\le\mathsf{T}(1+2\eta_K)^{p+3/2}\|\varphi\|_{L^2(K_{i1}^h)}\le \Theta_K^{1/2}\|\varphi\|_{L^2(K_{i1}^h)}$.

Our next goal is to introduce an interpolation operator for the finite element space $X_p(\cM)$. Notice that for any $v\in X_p(\cM)$, $v=v_1\chi_{\Om_1}+v_2\chi_{\Om_2}$, where for $i=1,2$, $v_i$ is defined on the mixed mesh $\cM_i$ which includes rectangular elements $K\in\cM$ inside $\Om_i$ and curved elements $\widetilde K_{ij}^h$, $i=1,2, j=1,\cdots, J_i^K$, $K\in\cM^\Ga$.

We first recall classical results on $hp$ interpolation operators.

\begin{lem}\label{Babuska}
Let $s>1$. There exist $hp$-interpolation operators $\pi_{K}^{hp}: H^s(K)\to Q_p(K)$ and $\gamma_K^{hp}:H^s(K)\to P_p(K)$ such that for $0\leq t\leq s$, $v\in H^s(K)$,
\be
&\hskip1cm\|v-\pi_{K}^{hp}(v)\|_{H^t(K)}+\|v-\gamma_K^{hp}(v)\|_{H^t(K)}\leq C \frac{h_K^{\nu-t}}{p^{s-t}}\|v\|_{H^s(K)}, \label{ba1}\\
&\hskip1cm\|v-\pi_{K}^{hp} (v)\|_{L^{\infty}(K)}+\|v-\gamma_{K}^{hp}(v)\|_{L^{\infty}(K)}\leq C \frac{h_K^{\nu-1}}{p^{s-1}}\|v\|_{H^s(K)},\label{ba2}
\ee
where $\nu=\min(p+1,s)$. Moreover, if $s>3/2$, for $t=0,1$, we have
\be
& &{\hskip0.5cm\|D^t(v-\pi_K^{hp}(v))\|_{L^2(\partial K)}+\|D^t(v-\gamma_K^{hp}(v))\|_{L^2(\pa K)}\leq C \frac{h_K^{\nu-t-1/2}}{p^{s-t-1/2}}\|v\|_{H^s(K)}.}\label{ba3}
\ee
Here the constant $C$ is independent of $p$ and $h_K$.
\end{lem}

\begin{proof} We first remark that \eqref{ba3} follows from \eqref{ba1}-\eqref{ba2} by the well-known multiplicative trace inequality
\be\label{trace}
\|v\|_{L^2(\pa K)}\le Ch_K^{-1/2}\|v\|_{L^2(K)}+C\|v\|^{1/2}_{L^2(K)}\|v\|_{H^1(K)}^{1/2}\ \ \ \ \forall v\in H^1(K),
\ee
The estimates for $\pi_K^{hp}$ can be found in Babu\v{s}ka and Suri \cite[Lemma 4.5]{Babuska1987}. We consider the operator $\gamma_K^{hp}:H^s(K)\to P_p(K)$. Let $\hat K=(0,1)\times (0,1)$ be the reference element which is included in the triangle $\hat T=\{(x_1,x_2)^T:x_1,x_2>0, x_1+x_2<2\}$. Let $F_K:\hat K\to K$ be the one-to-one affine mapping and $\hat v=v\circ F_K\in H^s(\hat K)$. Let $\hat u\in H^s(\R^2)$ be the Stein extension in Adams and Fournier \cite[Theorem 5.14]{Adams}
of $\hat v$ such that $\|\hat u\|_{H^s(\R^2)}\le C\|\hat v\|_{L^2(\hat K)}$, where the constant $C$ may depend on $s$ but is independent of $\hat v$. Let $\pi_p:H^s(\hat T)\to P_p(\hat T)$ be the $hp$-interpolation operator in Melenk and Sauter \cite[Lemma B.3]{Melenk2010MC} which satisfies
\ben
& &\|\hat u-\pi_p(\hat u)\|_{H^t(\hat T)}\le Cp^{-(s-t)}|\hat u|_{H^s(\hat T)}\ \ \forall p\ge s-1,\\
& &\|\hat u-\pi_p(\hat u)\|_{L^\infty(\hat T)}\le Cp^{-(s-1)}|\hat u|_{H^s(\hat T)}\ \ \forall s>1.
\een
Then we define $\gamma_K^{hp}(v)=\pi_p(\hat u)|_{\hat K}\circ F_K^{-1}$. The estimates \eqref{ba1}-\eqref{ba2} for $\gamma_K^{hp}$ follow from the standard scaling argument.
\end{proof}

\begin{lem}\label{lem:2.3new}
For any $u\in H^{s}(\Omega)$, $s>3/2$, let $\Pi_{hp}(u)=\pi_h({u_{hp}})$, where $u_{hp}|_{K}=\pi_K^{hp}(u)\in Q_p(K)\ \ \forall K\in\cM$
and $\pi_h$ is the local smoothing operator in Lemma \ref{lem:local_smoothing}. Then we have
\ben
& &\|u-\Pi_{hp}(u)\|_{L^2(K)}+\frac{h_K}{p^{3/2}}\|\na(u-\Pi_{hp}(u))\|_{L^2(K)}\le C\frac{h_K^{\nu}}{p^{s}}\|u\|_{H^s(\tilde\omega(K))} ,
\een
where $\tilde\omega(K)$ is a set of elements including $K$ such that ${\rm diam}(\tilde\omega(K))\le Ch_K$ and the constant $C$ is independent of $h_K,p$, but may depend on $s$.
\end{lem}

\begin{proof} We only prove the estimate for $\na(u-\Pi_{hp}(u))$. The other estimate can be proved similarly. First we notice that since $\lj u\rj|_e=0\ \forall e\in\cE^{\rm side}$, in the notation of Lemma \ref{lem:local_smoothing}, $\sigma(K)=\{e\in\cE^{\rm side}:e\subset\tilde\omega(K)\}$, by \eqref{ba3},
\ben
\|\lj u_{hp}\rj\|_{L^2(\sigma(K))}&\le& \sum_{K'\in \tilde{\omega}(K)}\|\lj u-\pi_{K'}^{hp}(u) \rj \|_{L^2(\partial K')}\nn\\
&\le& C\frac{h_K^{\nu-1/2}}{p^{s-1/2}}\|u\|_{H^s(\tilde\omega(K))}.
\een
It then implies by Lemma \ref{lem:local_smoothing} and \eqref{ba1} that
\ben
\|\nabla(u-\Pi_{hp}(u))\|_{L^2(K)}&\le&\|\na(u-u_{hp})\|_{L^2(K)}+C\|ph^{-1/2}\lj u_{hp}\rj\|_{L^2(\sigma(K))}\nn\\
&\le& C \frac{h_K^{\nu-1}}{p^{s-3/2}}\|u\|_{H^s(\tilde{\omega}(K))}.
\een
This completes the proof.
\end{proof}

The following lemma on the $hp$ vertex and edge lifting operators is from \cite[Theorem B.1]{Melenk2010MC}.

\begin{lem}\label{Melenk}
Let $T\subset\R^2$ be a reference triangle with vertices $P_1,P_2,P_3$. Then \\
$1^\circ$ There exists a vertex lifting operator $L_{\rm V}:C(\bar{T})\to P_p(T)$ such that for any $v\in C(\bar{T})$, $L_{\rm V}(v)(P_j)=v(P_j)$, $j=1,2,3$, and
\ben
p\|L_{\rm V}(v)\|_{L^2(T)}+\|\na L_{\rm V}(v)\|_{L^2(T)}\le C\max_{j=1,2,3}|v(P_j)|.
\een
$2^\circ$ For each edge $\hat e$ of $T$, there exists an edge lifting operator $L_{\hat e}:H^{1/2}_{00}(\hat e)\cap P_p(\hat e)\to P_p(T)$ such that for any $v\in P_p(\hat e)\cap H^{1/2}_{00}(\hat e)$, we have $L_{\hat e}(v)=v$ on $\hat e$, $L_{\hat e}(v)=0$ on $\pa T\backslash\bar{\hat e}$, and
\ben
& &p\|L_{\hat e}(v)\|_{L^2(T)}+\|\na L_{\hat e}(v)\|_{L^2(T)}\le C(\|v\|_{H^{1/2}_{00}(\hat e)}+p^{1/2}\|v\|_{L^2(\hat e)}).
\een
Here the constant $C$ is independent of $p$.
\end{lem}

We also recall that for any $v\in P_p(-1,1)$, $v(-1)=v(1)=0$, by the inverse estimate in Babu\v{s}ka et al \cite[Lemma 6.5]{Babu91}
\be\label{BB}
\|v\|_{H^{1/2}_{00}(-1,1)}\le C(1+\log p)\|v\|_{H^{1/2}(-1,1)}.
\ee

\begin{lem}\label{lem:2.4new}
Let $u\in H^s(\Om)$, $s>3/2$. Then there exists a finite element function $\tilde\Pi_{hp}(u)$ such that for any $K\in\cM^\Ga$, $\tilde\Pi_{hp}(u)|_K \in X_p(K)$, $\tilde\Pi_{hp}(u)=\Pi_{hp}(u)$ on $\pa K\backslash\bar\Ga_K$, and
\ben
& &\|u-\tilde\Pi_{hp}(u)\|_{L^2(K)}+\Big(\frac{h_K}p\Big)\|\na(u-\tilde\Pi_{hp}(u))\|_{L^2(K)}\\
&\le&C\Theta_K^{1/2}(1+\log p)^2\frac{h_K^\nu}{p^{s-1/2}}\|u\|_{H^s(\tilde\omega(K))},
\een
where $\tilde\omega(K)$ is a set of elements including $K$ such that ${\rm diam}(\tilde\omega(K))\le Ch_K$. The constant $C$ is independent of $h_K,p$, but may depend on $s$.
\end{lem}

\begin{proof} For the sake of definiteness, we prove the lemma for the case when $\Ga$ intersects $K$ at two neighboring sides, see Fig.\ref{fig:2.1}(left). In this case, $J_1^K=1$, $J_2^K=3$. The other case can be proved similarly.

For any $K\in \cam^{\Gamma}$, we know from Lemma \ref{Babuska} that there exists $\gamma_K^{hp}(u)\in P_p(K)$ such that for $t=0,1$,
\be\label{vh0}
\|u-\gamma_K^{hp}(u)\|_{H^t(K)}\le C\frac{h_K^{\nu-t}}{p^{s-t}}\|u\|_{H^s(K)}.
\ee
Define $v=\Pi_{hp}(u)-\gamma_K^{hp}(u)\in Q_p(K)$. By \eqref{vh0} and Lemma \ref{lem:2.3new},
\be\label{xz0}
\|v\|_{L^2(K)}+\Big(\frac{h_K}p\Big)\|\na v\|_{L^2(K)}\le C\frac{h_K^\nu}{p^s}\|u\|_{H^s(\tilde\omega(K))}.
\ee
Moreover, by Lemma \ref{Babuska}, the $hp$ inverse estimate, and Lemma \ref{lem:local_smoothing}, we have
\ben
\|v\|_{L^\infty(K)}&\le&\|u-u_{hp}\|_{L^\infty(K)}+\|u-\gamma_K^{hp}(u)\|_{L^\infty(K)}+\|u_{hp}-\pi_h(u_{hp})\|_{L^\infty(K)}\nn\\
&\le&C\frac{h_K^{\nu-1}}{p^{s-1}}\|u\|_{H^s(K)}+Cp^2h_K^{-1}\|p^{-1}h^{1/2}\lj u_{hp}\rj\|_{L^2(\sigma(K))}\nn\\
&\le&C\frac{h_K^{\nu-1}}{p^{s-3/2}}\|u\|_{H^s(\tilde\omega(K))}.
\een

We are going to lift the vertex and boundary values of $v$ on $\partial K \cap \partial K_{ij}^h$ into the element $K_{ij}^h$, $i=1,2, j=1,\cdots,J_i^K$. Let $F_{K_{ij}^h}:T\to K_{ij}^h$ be the one-to-one affine mapping from the reference element $T$ to $K_{ij}^h$. We define $v^1=\hat v^1\circ F_{K_{ij}^h}^{-1}$, where $\hat v^1=L_{\rm V}(\hat v)$ is the vertex lifting of $\hat v=v\circ F_{K_{ij}^h}$ in Lemma \ref{Melenk}. Then $v^1|_{K_{ij}^h}\in P_p(K_{ij}^h)$, $v^1(P)=v(P)$ for all vertices of $K_{ij}^h$, and satisfies by the standard scaling argument that
\be\label{xz1}
\ \ \ \|v^1\|_{L^2(K_{ij}^h)}+\Big(\frac{h_K}p\Big)\|\na v^1\|_{L^2(K_{ij}^h)}&\le&C\frac{h_K}p\left(p\|\hat v^1\|_{L^2(T)}+\|\hat\na\hat v^1\|_{L^2(T)}\right)\\
&\le&C\frac{h_K}p\|v\|_{L^\infty(K)}\nn\\
&\le&C\frac{h_K^\nu}{p^{s-1/2}}\|u\|_{H^s(\tilde\omega(K))}.\nn
\ee
Now $v-v^1$ vanishes at the vertices of $K_{ij}^h$. For $i=1$ and $i=2, j=2,3$, we know that $\pa K_{ij}^h\cap\pa K$ consists of two edges $e_{ij}, e_{ij}'$. Denote by $\hat e_{ij}=F_{K_{ij}^h}^{-1}(e_{ij})$, $\hat e_{ij}'=F_{K_{ij}^h}^{-1}(e_{ij}')$. We let $v^2|_{K_{ij}^h}=\hat v^2\circ F_{K_{ij}^h}^{-1}\in P_p(K_{ij}^h)$, where $\hat v^2=L_{\hat e_{ij}}(\hat v-\hat v^1)+L_{\hat e_{ij}'}(\hat v-\hat v^1)$. Then
$v^2=v-v^1$ on $\pa K_{ij}^h\cap\pa K$ and by Lemma \ref{Melenk} and \eqref{BB}
\ben
& &\|v^2\|_{L^2(K_{ij}^h)}+\Big(\frac{h_K}p\Big)\|\na v^2\|_{L^2(K_{ij}^h)}\\
&\le&C\frac{h_K}p\,(\|\hat v-\hat v^1\|_{H^{1/2}_{00}(\hat e_{ij}\cup\hat e_{ij}')}+p^{1/2}\|\hat v-\hat v^1\|_{L^2(\hat e_{ij}\cup\hat e_{ij}')})\\
&\le&C\frac{h_K}p(1+\log p)(\|\hat v-\hat v^1\|_{H^{1/2}(\hat e_{ij}\cup \hat e_{ij}')}+p^{1/2}\|\hat v-\hat v^1\|_{L^2(\pa T)})\\
&\le&C\frac{h_K}p(1+\log p)(\|\hat v-\hat v^1\|_{H^1(T)}+p^{1/2}\|\hat v-\hat v^1\|_{L^2(T)}^{1/2}\|\hat v-\hat v^1\|_{H^1(T)}^{1/2}),
\een
where we have used $\|\varphi\|_{H^{1/2}(\pa T)}\le C\|\varphi\|_{H^1(T)}\ \forall\varphi\in H^1(T)$ and the multiplicative trace inequality \eqref{trace}. Now by the scaling argument we obtain
\be
& &\|v^2\|_{L^2(K_{ij}^h)}+\Big(\frac{h_K}p\Big)\|\na v^2\|_{L^2(K_{ij}^h)}\label{xz4}\\
&\le&C(1+\log p)\Big(\|v-v^1\|_{L^2(K_{ij}^h)}+\frac{h_K}p\|\na(v-v^1)\|_{L^2(K_{ij}^h)}\Big)\nn\\
&\le&C(1+\log p)\frac{h_K^\nu}{p^{s-1/2}}\|u\|_{H^s(\tilde\omega(K))},\nn
\ee
where we have used \eqref{xz0}-\eqref{xz1}. In a similar way,
for $i=2, j=1$, we let $v^2|_{K_{21}^h}\in P_p(K_{21}^h)$ be the lifting of the edge values of $(v^1+v^2)|_{K_{2j}^h}-v^1|_{K_{21}^h}$ on $\pa K_{21}^h\cap\pa K_{2j}^h$, $j=2,3$, which satisfies
\be
& &\|v^2\|_{L^2(K_{21}^h)}+\Big(\frac{h_K}p\Big)\|\na v^2\|_{L^2(K_{21}^h)}\label{xz2}\\
&\le&C(1+\log p)\sum^3_{j=1}\Big(\|v^1\|_{L^2(K_{2j}^h)}+\frac{h_K}p\|\na v^1\|_{L^2(K_{2j}^h)}\Big)\nn\\
& &\,+\,C(1+\log p)\sum^3_{j=2}\Big(\|v^2\|_{L^2(K_{2j}^h)}+\frac{h_K}p\|\na v^2\|_{L^2(K_{2j}^h)}\Big)\nn\\
&\le&C(1+\log p)^2\frac{h_K^\nu}{p^{s-1/2}}\|u\|_{H^s(\tilde\omega(K))}.\nn
\ee
Let $w=\gamma_K^{hp}(u)+v^1+v^2\in U_p(K)$. Since $v^2=v-v^1$ on $\pa K_{ij}^h\cap\pa K$, we know that $v^1+v^2+\gamma_K^{hp}(u)=\Pi_{hp}(u)$ on $\pa K$. Moreover, since $v^2|_{K_{21}^h}=(v^1+v^2)|_{K_{2j}^h}-v^1|_{K_{21}^h}$ on $\pa K_{21}^h\cap\pa K_{2j}^h$, $j=2,3$, we know that $v^1+v^2$ is continuous across $\pa K_{21}^h\cap\pa K_{2j}^h$, $j=2,3$. Thus $w\in H^1(K_1^h\cup K_2^h)$. Now we define $\tilde\Pi_{hp}(u)=\Lambda_K(w)$. Then
$\tilde\Pi_{hp}(u)\in X_p(K)$ and $\tilde\Pi_{hp}(u)-\gamma_K^{hp}(u)=\Lambda_K(v^1+v^2)$. Now by \eqref{z1}, Lemma \ref{Babuska}, and above estimates,
\ben
& &\|u-\tilde\Pi_{hp}(u)\|_{L^2(K)}+\Big(\frac{h_K}p\Big)\|\na(u-\tilde\Pi_{hp}(u))\|_{L^2(K)}\\
&\le&\|u-\gamma_K^{hp}(u)\|_{L^2(K)}+\Big(\frac{h_K}p\Big)\|\na(u-\gamma_K^{hp}(u))\|_{L^2(K)}\\
& &+C\Theta_K^{1/2}\Big(\|v^1+v^2\|_{L^2(K)}+\frac{h_K}p\|\na_h(v^1+v^2)\|_{L^2(K)}\Big)\\
&\le&C\Theta_K^{1/2}(1+\log p)^2\frac{h_K^\nu}{p^{s-1/2}}\|u\|_{H^s(\tilde\omega(K))}.
\een
This completes the proof.
\end{proof}

The following theorem is the main result of this subsection.

\begin{Thm}\label{thm:2.0}
Let $s>3/2$. There exists an interpolation operator $I_{hp}:H^s(\Om_1\cup\Om_2)\to X_p(\cM)$ such that for any $u\in H^s(\Om_1\cup\Om_2)$,
\be\label{y1}
& &\|u-I_{hp}(u)\|_\cM+\Big(\frac{h}p\Big)\|\na_h(u-I_{hp}(u))\|_\cM\\
&\le&C\Theta^{1/2}(1+\log p)^2\frac{h^\nu}{p^{s-1/2}}\|u\|_{H^s(\Om_1\cup\Om_2)}.\nn
\ee
Moreover, for $u=u_1\chi_{\Om_1}+u_2\chi_{\Om_2}$, $I_{hp}(u)=[I_{hp}(u)]_1\chi_{\Om_1}+[I_{hp}(u)]_2\chi_{\Om_2}$, we have
\be\label{y2}
& &\|u_i-[I_{hp}(u)]_i\|_{\cE^\Ga}+\Big(\frac{h}{p^{3/2}}\Big)\|\na(u_i-[I_{hp}(u)]_i)\|_{\cE^\Ga}\\
&\le&C\Theta^{1/2}(1+\log p)^2\frac{h^{\nu-1/2}}{p^{s-1}}\|u\|_{H^s(\Om_i)},\ \ i=1,2.\nn
\ee
Here $h=\max_{K\in\cM}h_K$, $\Theta=\max_{K\in\cM}\Theta_K$ and the constant $C$ is independent of $p$, $h_K$ for all $K\in\cM$, and $\eta_K$ for all $K\in\cM^\Ga$. Moreover, $I_{hp}(u)\in X_p^0(\cM)$ if in addition, $u\in H^1_0(\Om_1\cup\Om_2)$.
\end{Thm}

\begin{proof}
For $i=1,2$, let $\tilde u_i\in H^s(\Om)$ be the Stein extension of $u_i:=u|_{\Om_i}\in H^s(\Om_i)$ such that $\|\tilde u_i\|_{H^s(\Om)}\le C\|u_i\|_{H^s(\Om_i)}$. We define
\ben
I_{hp}(u)|_K=\left\{\begin{array}{ll}
\Pi_{hp}(\tilde u_i) & \mbox{if }K\in\cM, K\subset\Om_i,i=1,2,\\
\tilde\Pi_{hp}(\tilde u_1)\chi_{K_1}+\tilde\Pi_{hp}(\tilde u_2)\chi_{K_2} & \mbox{if } K\in\cM^\Ga.\end{array}\right.
\een
The estimate \eqref{y1} follows from Lemma \ref{lem:2.3new} and Lemma \ref{lem:2.4new}.

Now we prove the estimate for $\|\na(u_i-[I_{hp}(u)]_i)\|_{\cE^\Ga}$. The estimate for $\|u_i-[I_{hp}(u)]_i\|_{\cE^\Ga}$ can be proved similarly. Notice that by definition, for any $K\in\cM^\Ga$, $u_i-[I_{hp}(u)]_i=\tilde u_i-\tilde\Pi_{hp}(\tilde u_i)$ on $\Ga_K$. By the trace inequality on curved domains in Xiao et al \cite[Lemma 3.1]{Xiao},
\ben
\|\na(\tilde u_i-\gamma_K^{hp}(\tilde u_i))\|_{L^2(\Ga_K)}
&\le&C\|\na(\tilde u_i-\gamma_K^{hp}(\tilde u_i))\|_{L^2(K_i)}^{1/2}\|\na(\tilde u_i-\gamma_K^{hp}(\tilde u_i))\|_{H^1(K_i)}^{1/2}\\
& &\,+\,C\|\na(\tilde u_i-\gamma_K^{hp}(\tilde u_i))\|_{L^2(\pa K_i\backslash\bar\Ga_K)}.
\een
Since $\pa K_i\backslash\bar\Ga_K\subset\pa K$, by Lemma \ref{Babuska} we obtain
\be\label{xz3}
\|\na(\tilde u_i-\gamma_K^{hp}(\tilde u_i))\|_{L^2(\Ga_K)}
\le C\frac{h_K^{\nu-3/2}}{p^{s-3/2}}\|\tilde u_i\|_{H^s(\tilde\omega(K))}.
\ee
Notice that by the construction of $\tilde\Pi_{hp}$ in Lemma \ref{lem:2.4new}, we have $\gamma^{hp}_K(\tilde u_i)-\tilde\Pi_{hp}(\tilde u_i)=\Lambda_K(v_i^1+v_i^2)$, $\na(\gamma_K^{hp}(\tilde u_i)-\tilde\Pi_{hp}(\tilde u_i))=\Lambda_K(\na(v_i^1+v_i^2))$, where $v_i^1$, $v_i^2$ are the lifting functions in the proof of Lemma \ref{lem:2.4new} with respect to $\tilde u_i$. By using the trace inequality on curved domains again,
\ben
& &\|\na(\gamma_K^{hp}(\tilde u_i)-\tilde\Pi_{hp}(\tilde u_i))\|_{L^2(\Ga_K)}\\
&\le&C\|\Lambda_K(\na(v_i^1+v_i^2))\|_{L^2(K_i)}^{1/2}\|\Lambda_K(\na(v_i^1+v_i^2))\|_{H^1(K_i)}^{1/2}\\
& &\,+\,C\|\Lambda_K(\na(v_i^1+v_i^2))\|_{L^2(\pa K_i\backslash\bar\Ga_K)}.
\een
Since for any $\varphi\in U_p(K)$, $\Lambda_K(\varphi)=\varphi$ on $\pa K$ and $\pa K_i\backslash\bar \Ga_K=\pa K_i^h\cap\pa K$, by \eqref{z1} and the $hp$ inverse estimates,
\ben
& &\|\na(\gamma_K^{hp}(\tilde u_i)-\tilde\Pi_{hp}(\tilde u_i))\|_{L^2(\Ga_K)}\\
&\le&C\Theta_K^{1/2}\|\na(v^1_i+v_i^2)\|_{L^2(K_i^h)}^{1/2}\|\na(v_i^1+v_i^2)\|_{H^1(K_i^h)}^{1/2}+C\|\na(v_i^1+v_i^2)\|_{L^2(\pa K_i^h\cap\pa K)}\\
&\le&C\Theta_K^{1/2}ph_K^{-1/2}\|\na(v^1_i+v_i^2)\|_{L^2(K_i^h)}\\
&\le&C\Theta_K^{1/2}(1+\log p)^2\frac{h_K^{\nu-3/2}}{p^{s-5/2}}\|u\|_{H^s(\tilde\omega(K))},
\een
where we have used \eqref{xz1}-\eqref{xz2} in the last inequality.
This completes the proof by the triangle inequality and \eqref{xz3}.
\end{proof}

\subsection{The unfitted finite element method}

In this subsection we introduce our semi-discrete unfitted finite element method for the wave equation \eqref{modelproblem}. Let $P_h:L^2(\Omega)\to X_p^0(\cam)$ and $\bm{P}_h:[L^2(\Omega)]^2\to\fespaceq$ be the standard $L^2$ projection operators
\be
& &(P_h v,\varphi_h)_\cam=(v,\varphi_h)_\cam\quad\forall v\in L^2(\Omega),\ \forall\varphi_h\in\fespace,\label{projection1}\\
& &(\bm{P}_h\bms,\bmp_h)_{\cam}=(\bms,\bmp_h)_{\cam}\quad\forall \bms\in[L^2(\Omega)]^2,\ \forall\bmp_h\in \fespaceq.
\label{projection2}
\ee
The semi-discrete unfitted finite element method for solving \eqref{modelproblem} is then to find $(u_h,\bfq_h)\in \fespace\times \fespaceq$ such that
for any $(\varphi_h,\bmp_h)\in \fespace\times\fespaceq$,
\begin{align}
&\Big(\frac{1}{\rho c^2}\pa_tu_h,\varphi_h\Big)_{\cam}=\mathcal{H}^-(\bfq_h,\varphi_h)+(f,\varphi_h)_{\cam},\label{ldgsch1}\\
&(\rho\pa_t\bfq_h,\bmp_h)_{\cam}=\mathcal{H}^+(u_h,\bmp_h),\label{ldgscH3}\\
&u_h(\mathbf{x},0)=(P_h u_0)(\mathbf{x}),\ \ \bfq_h(\mathbf{x},0)=(\bm{P}_h\bfq_0)(\mathbf{x})\ \ \mbox{in }\Omega,\label{ldgscH4}
\end{align}
where
\be
& &\mathcal{H}^-(\bfq_h,\varphi_h)=-(\bfq_h,\nabla_h \varphi_h)_{\cam}+\langle\bfq^-_h\cdot \mathbf{n}, [\![\varphi_h]\!]\rangle_{\E^{\Gamma}},\label{Dneg}\\
& &\mathcal{H}^+(u_h,\bmp_h)=-(u_h, \text{div}_h\bmp_h)_{\cam}+\langle u^+_h, \lj \bmp_h \rj\cdot \mathbf{n}\rangle_{\E}.\label{Dpos}
\ee
Here for any $\bmp_h\in\fespaceq$,
\ben
& &\text{div}_h\bmp_h|_K=({\rm div}\bmp_h)|_K, \ \ \mbox{if }K\in\cM\backslash\cM^\Ga,\\
& &\text{div}_h\bmp_h|_{\widetilde K_{ij}^h}=({\rm div}\bmp_h)|_{\widetilde K_{ij}^h}, i=1,2,j=1,\cdots,J_i^K, \ \ \mbox{if }K\in\cM^\Ga.
\een
Since $\lj\varphi_h\rj=0$ on $\cE^{\rm side}$ for any $\varphi_h\in X_p^0(\cM)\subset H^1(\Om_1\cup\Om_2)$, it is easy to see by integration by parts that for $(\varphi_h,\bmp_h)\in \fespace\times\fespaceq$,
\ben
\mathcal{H}^+(\varphi_h,\bmp_h)=(\nabla_h \varphi_h, \bmp_h)_{\cam}-\langle [\![{\varphi_h}]\!], \bmp^-_h \cdot \mathbf{n}\rangle_{\E^{\Gamma}}.
\een
Thus
\be\label{ee}
\mathcal{H}^-(\bmp_h,\varphi_h)+\mathcal{H}^+(\varphi_h,\bmp_h)=0\ \ \ \forall\varphi_h\in\fespace,\ \bmp_h\in\fespaceq.
\ee
This identity reflects the duality of the gradient and divergence operators in the discrete setting.

We remark that our method \eqref{ldgsch1}-\eqref{ldgscH3} reduces to the mixed formulation of the spectral element method in \cite[\S13.4]{Cohen2002higher} when the interface is absent and the mesh are conforming. It is shown in \cite{Cohen2002higher} that the mixed formulation of the spectral element method is equivalent to the standard spectral element method for solving the wave equation in the second order form but has favorable properties in terms of the storage and the CPU time.



The following proposition shows that the proposed semi-discrete method conserves energy when the source is absent.
\begin{Pro}
Let $f=0$ in $\Om\times (0,T)$. Then the (continuous) energy
\begin{align*}
E_h(t)=\int_{\Omega}\Big(\frac{1}{\rho c^2} |u_h|^2+\rho |\bfq_h|^2\Big)\, d\mathbf{x}
\end{align*}
is conserved by the semi-discrete method \eqref{ldgsch1}-\eqref{ldgscH3} for all time.
\end{Pro}

\begin{proof}
By taking the test functions $\varphi_h=u_h$ and $\bmp_h=\bfq_h$ in \eqref{ldgsch1}-\eqref{ldgscH3}, adding the equations, and using \eqref{ee}, we obtain
\ben
\Big(\frac{1}{\rho c^2}\pa_tu_h,u_h\Big)_{\cam}+(\rho\pa_t\bfq_h,\bfq_h)_{\cam}=0.
\een
Therefore, the quantity $E_h(t)$ is invariant in time.
\end{proof}

\begin{lem}\label{lem:2.2} Let $s>3/2$. The $L^2$ projection operators in \eqref{projection1}-\eqref{projection2} satisfy the following error estimates
\ben
& &\hskip-0.5cm\|v-P_h v\|_{L^2(\Om)}\leq C\Theta^{1/2}(1+\log p)^2 \frac{h^{\min(p+1,s)}}{p^{s-1/2}}\|v\|_{H^s(\Omega_1\cup\Omega_2)}\ \ \forall v\in H^s(\Om_1\cup\Om_2),\\
& &\hskip-0.5cm\|\bms-\bm{P}_h\bms\|_{L^2(\Om)}\le C \frac{h^{\min(p,s)}}{p^{s}}\|\bms\|_{H^s(\Om_1\cup\Om_2)}\quad\forall\bms\in [H^s(\Om_1\cup\Om_2)]^2.
\een
Moreover, for $\bms=\bms_1\chi_{\Om_1}+\bms_2\chi_{\Om_2}$, $\bm{P}_h\bms=(\bm{P}_h\bms)_1\chi_{\Om_1}+(\bm{P}_h\bms)_2\chi_{\Om_2}$, we have
\ben
& &\|\bms_i-(\bm{P}_h\bms)_i\|_{\cE^\Ga}
\leq C\Theta^{1/2}\frac{h^{\min(p+1,s)-1/2}}{p^{s-1}}\|\bms\|_{H^s(\Om_i)}\quad\forall\bms\in [H^s(\Om_i)]^2.
\een
Here the constant $C$ is independent of $p$, $h_K$ for all $K\in\cM$, and $\eta_K$ for all $K\in\cM^\Ga$.
\end{lem}

We remark that the first estimate is slightly sub-optimal in the power of $p$ compared with \eqref{ba3}. This is due to the presence of the curved interface. Optimal $hp$ error estimates for the $L^2$ projection operator on Cartesian meshes are known, see Houston et al \cite[Lemma 3.6]{Houston}.

\begin{proof} The first estimate follows directly from Theorem \ref{thm:2.0} since $\|v-P_hv\|_{L^2(\Om)}\le ||v-I_{hp}(v)\|_{L^2(\Om)}$. Now we show the the second and the third estimates.
For any $\bms\in [H^s(\Om_1\cup\Om_2)]^2$, we denote $\tilde{\bms}_i\in[H^s(\Om)]^2$ the Stein extension of $\bms_i:=\bms|_{\Om_i}\in[H^s(\Om_i)]^2$. Since $[Q_{p-1}(\cM)]^2\subset\fespaceq$, we define the interpolation operator $\hat\Pi_{hp}:[H^s(\Om_1\cup\Om_2)]^2\to\fespaceq$ by
\begin{align*}
\hat{\Pi}_{hp}(\bms)|_K=\left\{\begin{array}{ll}
\pi_{K}^{h,p-1}(\tilde{\bms}_i), & \text{if } K\in \cam, K\subset\Om_i, i=1,2,\\
 \gamma_K^{hp}(\tilde{\bms}_1)\chi_{K_1}+\gamma_K^{hp}(\tilde{\bms}_2)\chi_{K_2}, & \text{if }K\in\cM^\Ga.
 \end{array}\right.
\end{align*}
The second estimate follows from Lemma \ref{Babuska} since $\|\bms-\bm{P}_h\bms\|_{L^2(\Om)}\le \|\bms-\hat\Pi_{hp}(\bms)\|_{L^2(\Om)}$. The third estimate can be proved by the same argument as that at the end of the proof of Theorem \ref{thm:2.0}. Here we omit the details.
\end{proof}

For any Lipschitz domain $D\subset\R^2$, we use the standard notation $H^k(\div;D)=\{\bmp\in [H^k(D)]^2:\div\bmp\in H^k(D)\}$ and $\curl\varphi=(\pa\varphi/\pa x_2,-\pa\varphi/\pa x_1)^T\ \forall\varphi\in H^1(D)$. For any integer $k\ge 1$, denote
\ben
Z_k:=H^{k+1}(\Om_1\cup\Om_2)\times H^k(\div;\Om_1\cup\Om_2)
\een
which is a Hilbert space under the graph norm.

Now we define an elliptic projection operator $\Upsilon_h:Z_0\to X_p^0(\cM)\times\fespaceq$ which plays an important role in studying the convergence of our unfitted finite element method in this paper. For any $(v,\bms)\in Z_0$, let $(v_I,\bms_I)=\Upsilon_h(v,\bms)\in\fespace\times\fespaceq$ satisfy
\be
& &\Big(\frac 1{\rho c^2}(v_I-v),\varphi_h\Big)_\cM=\mathcal{H}^{-}(\bms_I-\bms,\varphi_h)\ \ \forall\varphi_h\in X_p^0(\cM),\label{f1}\\
& &(\rho(\bms_I-\bms),\bmp_h)_\cM=\mathcal{H}^+(v_I-v,\bmp_h)\ \ \forall\bmp_h\in\fespaceq.\label{f2}
\ee
Notice that for any $(\varphi,\bmp)\in Z_0$, $\bmp^\pm\cdot{\bf n}\in H^{-1/2}(\Ga)$. Thus $\mathcal{H}^-(\bmp,\varphi), \mathcal{H}^+(\varphi,\bmp)$ are well defined for $(\varphi,\bmp)\in Z_0$ as follows:
\ben
& &\mathcal{H}^-(\bmp,\varphi)=-(\bmp,\na_h\varphi)_\cM+\la\bmp^-\cdot{\bf n},\lj\varphi\rj\ra_\Ga,\\
& &\mathcal{H}^+(\varphi,\bmp)=-(\varphi,\div_h\bmp)_\cM+\la\varphi^+,\lj\bmp\rj\cdot{\rm n}\ra_\Ga,
\een
where $\la\cdot,\cdot\ra_\Ga$ is the duality pairing between $H^{-1/2}(\Ga)$ and $H^{1/2}(\Ga)$. When $(\varphi,\bmp)\in\fespace\times\fespaceq$, the above definition agrees with the definition in \eqref{Dneg}-\eqref{Dpos}. Again, integration by parts, we have
\ben
\mathcal{H}^+(\varphi,\bmp)=(\na_h\varphi,\bmp)_\cM-\la\lj\varphi\rj,\bmp^-\cdot{\bf n}\ra_{\Ga}.
\een

It is easy to see that \eqref{f1}-\eqref{f2} has a unique solution $(v_I,\bmp_I)\in\fespace\times\fespaceq$ for given $(v,\bmp)\in Z_0$. The key observation is that for any $K\in\cM$, $\bmp_2\in [W_p(K)]^2$, if we take $\bmp_h=\bmp_2\chi_{K_2}\in\fespaceq$ in \eqref{f2}, then
\ben
(\rho(\bms_I-\bms),\bmp_2)_{K_2}=(\na(v_I-v),\bmp_2)_{K_2}\ \ \forall\bmp_2\in [W_p(K)]^2,
\een
where $(\cdot,\cdot)_{K_2}$ is the inner product of $L^2(K_2)$. This implies $\rho(\bms_I-\bm{P}_h\bms)=\na v_I-\bm{P}_h(\na v)$ in $K_2$. Thus we have
\be\label{keyp}
\rho^+(\bms_I-\bm{P}_h\bms)^+=(\na v_I-\bm{P}_h(\na v))^+\ \ \mbox{on }\Ga.
\ee
We remark that the same argument applying to \eqref{ldgscH3} yields $\rho^+\pa_t\bfq_h^+=\na_h u_h^+$ on $\Ga$. This property seems to be new in the literature, which is crucial for us to derive optimal energy error estimates without using the penalty terms.

\begin{lem}\label{lem:key}
Let $(v,\bms)\in Z_k$, $k\ge 1$, and $(v_I,\bms_I)=\Upsilon_h(v,\bms)$ be the solution of \eqref{f1}-\eqref{f2}. Then  we have
\ben
\|v-v_I\|_{L^2(\Om)}+\|\bms-\bms_I\|_{L^2(\Om)}\le C\Theta(1+\log p)^2\frac{h^{\min(p,k)}}{p^{k-3/2}}\|(v,\bms)\|_{Z_k},
\een
where the constant $C$ is independent of $p$, $h_K$ for all $K\in\cM$, and $\eta_K$ for all $K\in\cM^\Ga$.
\end{lem}

\begin{proof} We first introduce an interpolation function for $\bms\in H^k(\div;\Om_1\cup\Om_2)$. By the lifted regular decomposition theorem for Lipschitz domains in Hiptmair et al \cite[Theorem 5.2]{Hiptmair2012}, there exist $\bm{u}\in [H^{k+1}(\Om_1\cup\Om_2)]^2, w\in H^{k+1}(\Om_1\cup\Om_2)$ such that $\bms=\bm{u}+\curl w$, and
\be\label{yy1}
\|\bm{u}\|_{H^{k+1}(\Om_1\cup\Om_2)}+\|w\|_{H^{k+1}(\Om_1\cup\Om_2)}\le C\|\bms\|_{H^k(\div;\Om_1\cup\Om_2)}.
\ee
Define
\ben
\hat{\bm{P}}_h\bms=\bm{P}_h(\bm{u}+\curl I_{hp}(w))\in\fespaceq.
\een
By the triangle inequality, Lemma \ref{lem:2.2}, and Theorem \ref{thm:2.0}, we have
\be
& &\|\bms-\hat\bmP_h\bms\|_\cM\label{y3}\\
&\le&\|\bmu-\bmP_h\bmu\|_\cM+\|\curl w-\bmP_h\curl w\|_\cM+\|\curl(w-I_{hp}(w))\|_\cM\nn\\
&\le&C\Theta^{1/2}(1+\log p)^2\frac{h^{\min(p,k)}}{p^{k-3/2}}(\|\bmu\|_{H^{k+1}(\Om_1\cup\Om_2)}+\|w\|_{H^{k+1}(\Om_1\cup\Om_2)}).\nn
\ee
For any $K\in\cM^\Ga$, we know that $\curl I_{hp}(w)\in [W_p(K)]^2$ by definition. Thus
\be\label{y4}
\hat{\bmP}_h\bms=\bmP_h\bmu+\curl I_{hp}(w) \ \ \mbox{in }K\in\cM^\Ga.
\ee
Now for $\bms=\bms_1\chi_{\Om_1}+\bms_2\chi_{\Om_2}$, $\bmu=\bmu_1\chi_{\Om_1}+\bmu_2\chi_{\Om_2}$, $w=w_1\chi_{\Om_1}+w_2\chi_{\Om_2}$, we have $\bms^+=\bmu_2+\curl w_2$, $(\hat\bmP_h\bms)^+=(\bm{P}_h\bmu)_2+[\curl I_{hp}(w)]_2$ on $\Ga_K$.

Since $(I_{hp}(v),\hat\bmP_h\bms)\in\fespace\times\fespaceq$, we obtain by \eqref{f1}-\eqref{f2} that
\be
\quad\Big(\frac 1{\rho c^2}(I_{hp}(v)-v_I),\varphi_h\Big)&=&\mathcal{H}^-(\hat\bmP_h\bms-\bms_I,\varphi_h)+\Big(\frac 1{\rho c^2}(I_{hp}(v)-v),\varphi_h\Big)\label{f3}\\
& &\,+\,\mathcal{H}^{-}(\bms-\hat\bmP_h\bms,\varphi_h),\nn\\
(\rho(\hat\bmP_h\bms-\bms_I),\bmp_h)&=&\mathcal{H}^+(I_{hp}(v)-v_I,\bmp_h)+(\rho(\hat\bmP_h\bms-\bms),\bmp_h)\label{f4}\\
& &\,+\,\mathcal{H}^+(v-I_{hp}(v),\bmp_h).\nn
\ee
By taking $\varphi_h=I_{hp}(v)-v_I$ in \eqref{f3} and $\bmp_h=\hat\bmP_h\bms-\bms_I$ in \eqref{f4}, adding the two equations, and using \eqref{ee}, we have
\be\label{f9}
& &\|(\rho c^2)^{-1/2}(I_{hp}(v)-v_I)\|_\cM^2+\|\rho^{1/2}(\hat\bmP_h\bms-\bms_I)\|_\cM^2\\\
&=&\Big(\frac 1{\rho c^2}(I_{hp}(v)-v),\varphi_h\Big)+\mathcal{H}^-(\bms-\hat\bmP_h\bms,\varphi_h)\nn\\
& &\,+\,(\rho(\hat\bmP_h\bms-\bms),\bmp_h)+\mathcal{H}^+(v-I_{hp}(v),\bmp_h)\nn\\
&=:&{\rm I}+\cdots+{\rm IV}.\nn
\ee
By Theorem \ref{thm:2.0}, \eqref{y3}, and \eqref{yy1} we have
\be\label{f5}
\hskip0.5cm& & |{\rm I}+{\rm III}|\le C\Theta^{1/2}(1+\log p)^2\frac{h^{\min(p,k)}}{p^{k-3/2}}\|(v,\bms)\|_{Z_k}(\|\varphi_h\|_{L^2(\Om)}+\|\bmp_h\|_{L^2(\Om)}).
\ee
Since $\nabla_h\fespace\subset\fespaceq$, by \eqref{y4} and \eqref{projection2} we know that
\be
{\rm II}&=&-(\bm{u}-\bm{P}_h\bm{u},\nabla_h\varphi_h)_{\cam}+\langle (\bm{u}-\bm{P}_h\bm{u})^{-}\cdot\mathbf{n}, \lj \varphi_h \rj \rangle_{\E^{\Gamma}}\label{f6}\\
& &-({\bf curl\,}w-\bm{P}_h{\bf curl\,}I_{hp}(w)),\nabla_h\varphi_h)_{\cam}\nn\\
& &+\langle ( {\bf curl\,}w-\bm{P}_h{\bf curl\,}I_{hp}(w))^{-}\cdot\mathbf{n}, \lj \varphi_h \rj \rangle_{\E^{\Gamma}}\nn\\
&=&\langle (\bm{u}-\bm{P}_h\bm{u})^{-}\cdot\mathbf{n}, \lj \varphi_h \rj \rangle_{\E^{\Gamma}}
-({\bf curl\,}(w-I_{hp}(w)),\nabla_h\varphi_h)_{\cam}\nn\\
& &+\langle ({\bf curl\,}(w-I_{hp}(w)))^{-}\cdot\mathbf{n}, \lj \varphi_h \rj \rangle_{\E^{\Gamma}}\nn\\
&=:&{\rm II}_1 +{\rm II}_2+{\rm II}_3.\nn
\ee
By Lemma \ref{lem:2.2} and Lemma \ref{lem:2.1},
\be\label{yy2}
|{\rm II}_1|\leq C\Theta \frac{h^{\min(p,k)}}{p^{k-1}}\|\bm{u}\|_{H^{k+1}(\Omega_1\cup\Omega_2)}\|\varphi_h\|_{L^2(\Omega)}.
\ee
By integration by parts, we have
\ben
{\rm II}_2+{\rm II}_3=-\la\lj\curl(w-I_{hp}(w))\rj\cdot{\bf n},\varphi_h^+\ra_{\cE^\Ga}
=-\la\na_\Ga\lj w-I_{hp}(w)\rj,\varphi_h^+\ra_{\cE^\Ga},
\een
where for any $\varphi\in H^1(\Om_1\cup\Om_2)$, $\na_\Ga\varphi=\na\varphi\cdot\bm{\tau}$ with $\bm{\tau}=(-n_2,n_1)^T$ the unit tangential vector along $\Ga$. Since $X_p(\cM)\subset H^1(\Om_1\cup\Om_2)$ so that $\lj w-I_{hp}(w)\rj,\varphi_h^+\in H^{1/2}(\Ga)$, by integration by parts on the interface, we obtain
\ben
{\rm II}_2+{\rm II}_3=\la\lj w-I_{hp}(w)\rj,\na_\Ga\varphi^+_h\ra_{\cE^\Ga}.
\een

Since $\varphi_h=I_{hp}(v)-v_I$ and $\bmp_h=\hat\bmP_h\bms-\bms_I$, by \eqref{keyp} we have
\ben
{\rm II}_2+{\rm II}_3&=&\la\lj w-I_{hp}(w)\rj,\na_\Ga(I_{hp}(v)-v)^+\ra_{\cE^\Ga}\\
& &+\la\lj w-I_{hp}(w)\rj,(\na v-\bmP_h(\na v))^+\cdot\bm{\tau}\ra_{\cE^\Ga}+\la\lj w-I_{hp}(w)\rj,\rho\bmp_h^+\cdot\bm{\tau}\ra_{\cE^\Ga}\\
&=:&{\rm V}_1+{\rm V}_2+{\rm V}_3.
\een
By Theorem \ref{thm:2.0},
\ben
|{\rm V}_1|\le C\Theta(1+\log p)^4\frac{h^{2\min(p,k)}}{p^{2k-3/2}}\|w\|_{H^{k+1}(\Om_2)}\|v\|_{H^{k+1}(\Om_2)}.
\een
By Theorem \ref{thm:2.0} and Lemma \ref{lem:2.2},
\ben
|{\rm V}_2|\le C\Theta(1+\log p)^2\frac{h^{2\min(p,k)}}{p^{2k-1}}\|w\|_{H^{k+1}(\Om_2)}\|v\|_{H^{k+1}(\Om_2)}.
\een
By Theorem \ref{thm:2.0} and Lemma \ref{lem:2.1}
\ben
|{\rm V}_3|\le C\Theta(1+\log p)^2\frac{h^{\min(p,k)}}{p^{k-1}}\|w\|_{H^{k+1}(\Om_2)}\|\bmp_h\|_{L^2(\Om_2)}.
\een
Inserting above estimates to \eqref{f6} and using \eqref{yy1}, we obtain
\be\label{f7}
\qquad|{\rm II}|&\le&C\Theta(1+\log p)^4\frac{h^{2\min(p,k)}}{p^{2k-3/2}}\|(v,\bms)\|_{Z_k}^2\\
& &+\,C\Theta(1+\log p)^2\frac{h^{\min(p,k)}}{p^{k-1}}\|(v,\bms)\|_{Z_k}(\|\varphi_h\|_{L^2(\Om)}+\|\bmp_h\|_{L^2(\Om)}).\nn
\ee
Similarly, by Theorem \ref{thm:2.0} and Lemma \ref{lem:2.1},
\be
|{\rm IV}|\le C\Theta(1+\log p)^2\frac{h^{\min(p,k)}}{p^{k-3/2}}\|v\|_{H^{k+1}(\Om_1\cup\Om_2)}\|\bmp_h\|_{L^2(\Om)}.\label{f8}
\ee
Combing \eqref{f9}, \eqref{f5}, \eqref{f7}-\eqref{f8}, we obtain easily
\ben
\|I_{hp}(v)-v_I\|_{L^2(\Om)}+\|\hat\bmP_h\bms-\bms_I\|_{L^2(\Om)}\le C\Theta(1+\log p)^2\frac{h^{\min(p,k)}}{p^{k-3/2}}\|(v,\bms)\|_{Z_k}.
\een
This completes the proof by Theorem \ref{thm:2.0} and \eqref{y3}.
\end{proof}

The following theorem is the main result of this section.

\begin{Thm}\label{thm:2.1}
Assume that $(u_0,\bfq_0)\in Z_k$ for some integer $k\ge 1$. Let $(u,\bfq)\in H^1(0,T;Z_k)$ be the solution of the wave equations \eqref{modelproblem}, and $(u_h,\bfq_h)\in\fespace\times\fespaceq$ be the solution of \eqref{ldgsch1}-\eqref{ldgscH4}. Then there exists a constant $C$ independent of $p$, $h_K$ for all $K\in\cM$, and $\eta_K$ for all $K\in\cM^\Ga$ such that
\ben
& &\max_{0\le t\le T}(\|u-u_h\|_{L^2(\Om)}+\|\bfq-\bfq_h\|_{L^2(\Om)})\\
&\leq&C\Theta(1+\log p)^2\frac{h^{\min(p,k)}}{p^{k-1}}\|(u_0,\bfq_0)\|_{Z_k}\\
& &+\,C\Theta(1+\log p)^2(1+T^{1/2})\frac{h^{\min(p,k)}}{p^{k-3/2}}\|(u,\bfq)\|_{H^1(0,T;Z_k)}.
\een
\end{Thm}

We remark that the regularity assumption $u\in H^1(0,T;H^{k+1}(\Omega_1\cup \Omega_2))$ implies $\bfq\in H^2(0,T;[H^k(\Om_1\cup\Om_2)]^2)$ by the second equation of \eqref{modelproblem}. The assumption $\div\bfq\in H^1(0,T;H^k(\Om_1\cup\Om_2))$ follows from the first equation of \eqref{modelproblem} if $\pa_t u,\pa_t f\in H^1(0,T;H^k(\Om_1\cup\Om_2))$.

\begin{proof}
It follows from the equations in \eqref{modelproblem} that
\begin{align*}
&\Big(\frac{1}{\rho c^2}\pa_tu,\varphi_h\Big)_{\cam}=\mathcal{H}^-(\bfq,\varphi_h)+(f,\varphi_h)_{\cam}\quad\forall\varphi_h\in\fespace,\\
&(\rho\pa_t\bfq,\bmp_h)_{\cam}=\mathcal{H}^+(u,\bmp_h)\quad\forall\bmp_h\in\fespaceq.
\end{align*}
Denote $(u_I,\bfq_I)=\Upsilon_h(u,\bfq)\in\fespace\times\fespaceq$ defined in \eqref{f1}-\eqref{f2}. Then for $(\varphi_h,\bmp_h)\in\fespace\times\fespaceq$,
\be
 \Big(\frac{1}{\rho c^2}\pa_t u_I,\varphi_h\Big)_{\cam}&=&\mathcal{H}^-(\bfq_I,\varphi_h)+(f,\varphi_h)_{\cam}+(R_u,\varphi_h)_{\cam},\label{equuhstar}\\
(\rho\pa_t\bfq_I,\bmp_h)_{\cam}&=&\mathcal{H}^+(u_I,\bmp_h)+(\mathbf{R}_{\bfq},\bmp_h)_{\cam},\label{equqhstar}
\ee
where $(R_u,\mathbf{R}_{\bm q})\in\fespace\times\fespaceq$ are defined by
\be
&(R_u,\varphi_h)_{\cam} &=\Big(\frac{1}{\rho c^2}[\pa_t(u_I-u)+(u-u_I)],\varphi_h\Big)_{\cam}
\ \ \forall\varphi_h\in\fespace,\label{d3}\\
&(\mathbf{R}_{\bfq},\bmp_h)_{\cam}&=(\rho[\pa_t(\bfq_I-\bfq)+(\bfq-\bfq_I)],\bmp_h)_{\cam}\ \ \forall \bmp_h\in\fespaceq.\label{d4}
\ee
Thus subtract \eqref{equuhstar}-\eqref{equqhstar} from \eqref{ldgsch1}-\eqref{ldgscH3}, take $\varphi_h=u_I-u_h$, $\bmp_h=\bfq_I-\bfq_h$, add two equations, and use \eqref{ee} and Lemma \ref{lem:key}, we have
\ben
& &\frac 12\frac{d}{dt}\Big(\|(\rho c^2)^{-1/2}(u_I-u_h)\|_\cM^2+\|\rho^{1/2}(\bfq_I-\bfq_h)\|_\cM^2\Big)\\
&\le&C\Theta(1+\log p)^2\frac{h^{\min(p,k)}}{p^{k-3/2}}(\|(u,\bfq)\|_{Z_k}+\|\pa_t(u,\bfq)\|_{Z_k})(\|\varphi_h\|_{L^2(\Om)}+\|\bmp_h\|_{L^2(\Om)}).
\een
Let $t^*={\rm argmax}_{t\in (0,T]}(\|u_I-u_h\|_{L^2(\Om)}+\|\bfq_I-\bfq_h\|_{L^2(\Om)})$, integrate the above estimate over $(0,t^*)$, use Lemma \ref{lem:key} and Lemma \ref{lem:2.2} for the initial approximations, we obtain by standard argument that
\ben
& &\max_{0\le t\le T}(\|u-u_h\|_{L^2(\Om)}+\|\bfq-\bfq_h\|_{L^2(\Om)})\\
&\le&C\Theta(1+\log p)^2\frac{h^{\min(p,k)}}{p^{k-3/2}}\|(u_0,\bfq_0)\|_{Z_k}\\
& &+C\Theta(1+\log p)^2\frac{h^{\min(p,k)}}{p^{k-3/2}}(1+T^{1/2})\|(u,\bfq)\|_{H^1(0,T;Z_k)}.
\een
This completes the proof.
\end{proof}

\begin{rem}\label{rem1}
We remark that the energy error estimates in Theorem \ref{thm:2.1} are optimal in $h$ and slightly suboptimal in $p$ under the regularity assumption of the exact solution and the approximation orders of the finite element spaces used. In Chou et al \cite{csxjcp2014}, by using the same approximation order finite element spaces for the pressure and the velocity and assuming higher regularity of the exact solution, optimal error estimates of local DG methods without using penalty for a straight interface are obtained on Cartesian meshes. However, for unstructured meshes, the same optimal error estimates can only be obtained for numerical fluxes with a penalty, see Sun and Xing \cite[Theorem 3.1]{SX2021mc}.
\end{rem}

To conclude this section, we introduce the equivalent ODE system of \eqref{ldgsch1}-\eqref{ldgscH4}.
For any $(v_h, \bm{\sigma}_h)\in\fespace\times\fespaceq$,  we define $D^-\bm{\sigma}_h\in \fespace$ and $D^{+}v_h \in \fespaceq$ such that
\begin{align}
&(D^-\bm{\sigma}_h,\varphi_h)_{\cam}=\mathcal{H}^-(\bm{\sigma}_h,\varphi_h)\quad \forall \varphi_h \in \fespace,\label{c1}\\
&(D^{+}v_h,\bmp_h)_{\cam}=\mathcal{H}^+(v_h,\bmp_h)\quad \forall \bmp_h \in \fespaceq. \label{c2}
\end{align}
By \eqref{ee}, we have
\begin{align}\label{semineg}
(D^-\bms_h,v_h)_{\cam}+(D^{+}v_h,\bms_h)_{\cam}=0 \quad \forall (v_h,\bms_h) \in \fespace \times \fespaceq.
\end{align}
Let $\{\phi_i\}^{M_1}_{i=1}$ be the basis of $\fespace$ and $\{\bm{\psi}_i\}^{M_2}_{i=1}$ be the basis of the $\fespaceq$. We denote
$\bM_{\rho,c}\in\R^{M\times M}$, where $M=M_1+M_2$, the $2\times 2$ block diagonal matrix whose off-diagonal blocks are zero and the elements in the diagonal blocks are
\ben
& &(\mathbb{M}_{\rho,c})_{ij}=\Big(\frac 1{\rho c^2}\phi_i,\phi_j\Big)_\cM,\ \ i,j=1,\cdots,M_1, \\
& &{\rev{(\mathbb{M}_{\rho,c})_{i+M_1,j+M_1}=(\rho\bm{\psi}_i,\bm{\psi}_j)_\cM,\ \ i,j=1,\cdots,M_2.}}
\een
Similarly, $\bM\in\R^{M\times M}$ is the $2\times 2$ block diagonal matrices whose off-diagonal blocks are zero and the elements in the diagonal blocks are
\ben
& &(\mathbb{M})_{ij}=(\phi_i,\phi_j)_\cM,\ \ i,j=1,\cdots,M_1, \\
& &{\rev{(\mathbb{M})_{i+M_1,j+M_1}=(\bm{\psi}_i,\bm{\psi}_j)_\cM,\ \ i,j=1,\cdots,M_2.}}
\een
The mass matrices $\mathbb{M}_{\rho,c}$ and $
\mathbb{M}$ are sparse, symmetric, and positive definite.

The stiffness matrix is
\ben
\mathbb{A}=\begin{pmatrix}
0 & \mathbb{D}^- \\
\mathbb{D}^+ & 0
\end{pmatrix},
\een
where $\mathbb{D}^-\in\R^{M_1\times M_2}, \mathbb{D}^+\in\R^{M_2\times M_1}$ are the matrices whose elements are
$$(\mathbb{D}^-)_{ij}=\mathcal{H}^-(\bm{\psi}_j,\phi_i),\ \ (\mathbb{D}^+)_{ji}=\mathcal{H}^+(\phi_i,\bm{\psi}_j),\ \ i=1,\cdots, M_1, \ j=1,\cdots, M_2.$$
By \eqref{semineg},
\begin{align}\label{ddd}
\mathbb{A}+\mathbb{A}^T=\mathbb{O}.
\end{align}

The spatial discretization \eqref{ldgsch1}-\eqref{ldgscH3} can now be rewritten as the ODE system
\begin{align}\label{ODEs}
\mathbb{M}_{\rho,c}\frac{d}{dt}\begin{pmatrix}
\mathbf{U} \\
\mathbf{Q}
\end{pmatrix}
=\mathbb{A}\begin{pmatrix}
\mathbf{U} \\
\mathbf{Q}
\end{pmatrix}
+\mathbb{M}
\begin{pmatrix}
\mathbf{F} \\
\mathbf{0}
\end{pmatrix},
\end{align}
where $\mathbf{U}=(u_1,\cdots,u_{M_1})^T$ such that $u_h(\cdot,t)=\sum^{M_1}_{i=1}u_i(t)\phi_i(\cdot)$, $\mathbf{Q}=(q_1,\cdots,{q}_{M_2})^T$ such that $\bfq_h(\cdot,t)=\sum^{M_2}_{i=1}{q}_i(t)\bm{\psi}_i(\cdot)$, and $\mathbf{F}=(F_1,\cdots,F_{M_1})^T$ such that $P_hf(\cdot,t)=\sum^{M_1}_{i=1}F_i(t)\phi_i(\cdot)$.

In the following, we use $\Phi:\fespace\times\fespaceq\to\R^{M_1}\times\R^{M_2}$ to denote the correspondence between the finite element functions and their coefficient vectors. For any $(v_h,\bms_h)\in\fespace\times\fespaceq$, $v_h=\sum^{M_1}_{i=1}v_i\phi_i$, $\bms_h=\sum^{M_2}_{i=1}\sigma_i\bm{\psi}_i$, we denote $\mathbf{V}=(v_1,\cdots,v_{M_1})^T, \bm{\Sigma}=(\sigma_1,\cdots,\sigma_{M_2})^T$ and write
\be\label{d2}
(\mathbf{V}^T,\bm{\Sigma}^T)^T=\Phi(v_h,\bms_h).
\ee
In this notation, we have $(\mathbf{U}^T,\mathbf{Q}^T)^T=\Phi(u_h,\bfq_h)$ and $(\mathbf{F}^T,\bm{0}^T)^T=\Phi(P_hf,0)$.

By multiplying \eqref{ODEs} by $\mathbb{M}_{\rho,c}^{-\frac 12}$, we obtain
\begin{align}\label{ODEsystem1}
\frac{d}{dt}\mathbf{Y}(t)=\mathbb{D}\mathbf{Y}(t)+\mathbf{R}(t),
\end{align}
with $\mathbf{Y}=\bfm^{\frac 12}(\bfu^T,\mathbf{Q}^T)^T$, $\bfd=\bfm^{-\frac 12}\mathbb{A}\bfm^{-\frac 12}$, and $\bfr=\bfm^{-\frac 12}\mathbb{M}(\bff^T, \mathbf{0}^T)^T$. Since the matrix $\bfm^{-\frac 12}$ is symmetric, by \eqref{ddd}, we still have
\begin{align}\label{ddt0}
\bfd+\bfd^T=\mathbb{O}.
\end{align}
In next section we will propose an explicit time discretization method for solving \eqref{ODEsystem1} under the condition \eqref{ddt0}.

\begin{rem}\label{rem2}
Since we use the conforming Galerkin finite element space to approximate $u$ in each subdomain and discontinuous finite element space to approximate $\bfq$, our mass matrix $\bfm$ can become diagonal in the region away from the interface when the mass lumping techniques are used \cite{Cohen2002higher}.
\end{rem}

\setcounter{equation}{0}
\section{The explicit time discretization for the ODE system}
\label{sectimedisc}

In this section, we propose a strongly stable high order explicit time discretization method for the ODE system
\begin{align}\label{ODEsystem}
\mathbf{Y}'(t)=\mathbb{D}\mathbf{Y}(t)+\mathbf{R}(t)\ \ \mbox{in }(0,T),\ \ \ \ \mathbf{Y}(0)=\bfy^0,
\end{align}
where for $M\ge 1$, $\mathbf{Y}\in\R^M, \mathbb{D}\in\R^{M\times M}$ satisfies \eqref{ddt0}, and $\mathbf{R},\bfy^0\in\R^M$.

Let $0=t_0 < t_1 <\cdots < t_N=T$ be a partition of the interval $(0,T)$ with the time step sizes $\tau_n=t_{n+1}-t_n$. We set $\displaystyle \tau=\max_{0\le n\le N-1}\tau_n$. For any integer $m\ge 1$, we define the finite element space
\ben
\qquad\mathbf{V}_m(0,T)=\{\mathbf{v}\in (C[0,T])^M: \mathbf{v}|_{(t_n,t_{n+1})}\in [P^m(t_n,t_{n+1})]^M, 0\le n\le N-1\}.
\een

From \eqref{ODEsystem} we know that for $j\ge 1$,
\ben
\bfy^{(j)}(t)=\mathbb{D}^j\bfy(t)+\sum^{j-1}_{\ell=0}\mathbb{D}^{j-1-\ell}\mathbf{R}^{(\ell)}(t).
\een
Thus by Taylor expansion, we know that for $t\in (t_n,t_{n+1})$,
\be\label{cc1}
\bfy(t)&\approx&\sum^r_{j=0}\frac 1{j!}\bfy^{(j)}(t_n)(t-t_n)^j\\
&=&\sum^r_{j=0}\frac 1{j!}\left[\mathbb{D}^j\bfy(t_n)+\sum^{j-1}_{\ell=0}\mathbb{D}^{j-1-\ell}\mathbf{R}^{(\ell)}(t_n)\right](t-t_n)^j.\nn
\ee
For $r\ge 1$, the $r$ stage $r$ order explicit Runge-Kutta scheme is equivalent to compute $\bfy_r(t_{n+1})$, which is the approximate value of $\bfy(t_{n+1})$, from the approximate value $\bfy^n$ of $\bfy_r(t_n)$, where
\be\label{cc2}
\bfy_r(t)=\sum^r_{j=0}\frac 1{j!}\left[\mathbb{D}^j\bfy^n+\sum^{j-1}_{\ell=0}\mathbb{D}^{j-1-\ell}\mathbf{R}^{(\ell)}(t_n)\right](t-t_n)^j.
\ee
This scheme is, unfortunately, not strongly stable when $r=1,2\,(\mbox{mod }4)$, see Sun and Shu \cite[Theorem 4.2]{Sun2019ssp}. To overcome the difficulty, we propose to add a stabilization term in \eqref{cc2} and introduce the following scheme when $r=1,2\,(\mbox{mod }4)$
\be\label{cc3}
\hskip1cm\widetilde\bfy_r(t)=\bfy_r(t)+\frac 1{r!(r+\gamma)}\left[\mathbb{D}^{r+1}\bfy^n+\sum^r_{\ell=0}\mathbb{D}^{r-\ell}\mathbf{R}^{(\ell)}(t_n)\right](t-t_n)^{r+1}.
\ee
Notice that $\widetilde{\bfy}_r(t)=\bfy_{r+1}(t)$ if $\gamma=1$. We will show that \eqref{cc3} leads to a strongly stable explicit scheme when chosing $\gamma\in (0,1)$ for $r=1,2\,(\mbox{mod }4)$.

Our explicit time discretization method for solving \eqref{ODEsystem} is the following algorithm which iteratively computes $\bfy_r,\widetilde\bfy_r$ in \eqref{cc2}, \eqref{cc3}. Given $\mathbf{R}(t) \in (C^r[0,T])^M$, we find $\widetilde{\bfy}_r\in \mathbf{V}_{r+1}(0,T)$ such that $\widetilde{\bfy}_r(0)=\bfy^0$, and in each time interval $(t_n,t_{n+1})$, $0\le n\le N-1$, $\widetilde{\bfy}_r$ is computed by the following algorithm with $\bfy^n=\widetilde{\bfy}_r(t_n)$.

\begin{alg}\label{timedisc} Given $\gamma\in (0,1)$, $\bfy^n\in\R^M$, and $\mathbf{R}\in\R^M$. \\
$1^\circ$ Set $\bfy_0=\bfy^n$ and find $\widetilde{\bfy}_0\in [P^1(t_n,t_{n+1})]^M$ such that
\ben
\widetilde{\bfy}_0'=\frac{1}{\gamma}\,\bfd\bfy^n,\ \ \widetilde{\bfy}_0(t_n)=\bfy^n.
\een
$2^\circ$ For $1\le m\le r$, compute $\bfy_m\in [P^{m}(t_n,t_{n+1})]^M$ such that
\ben
\bfy_m'=\bfd\bfy_{m-1}+I_{m-1}\bfr,\ \ \bfy_m(t_n)=\bfy^n.
\een
$3^\circ$ For $1\le m\le r$, compute $\widetilde{\bfy}_m\in [P^{m+1}(t_n,t_{n+1})]^M$ such that
\ben
& &\widetilde{\bfy}_m'=\bfd[\gamma_m^r\widetilde{\bfy}_{m-1}+(1-\gamma_m^r)\bfy_{m-1}]+\gamma_m^r\tilde{I}_{m-1}\bfr+(1-\gamma_m^r)I_{m-1}\bfr,\\
& &\widetilde{\bfy}_m(t_n)=\bfy^n,
\een
where for $1\le m\le r$, the parameter
\ben
\gamma_{m}^r=\left\{\begin{array}{cc}
0,& \text{ if } r=0,3\,(\mbox{\rm mod}\, 4),\\
\frac{(m+1)(m-1+\gamma)}{m(m+\gamma)}, & \text{ if } r=1,2\,(\mbox{\rm mod}\, 4),
\end{array}\right.
\een
and
\ben
& I_{m-1}\bfr&=\sum_{j=0}^{m-1}\frac{1}{j!}\bfr^{(j)}(t_n)(t-t_n)^j,\\
&\tilde{I}_{m-1}\bfr& =I_{m-1}\bfr+\frac{1}{(m-1)!(m-1+\gamma)}\bfr^{(m)}(t_n)(t-t_n)^{m}.
\een
\end{alg}

The parameter $\gamma^r_m$ is so chosen that one can easily show by mathematical induction that the outcome of Algorithm \ref{timedisc} is $\widetilde\bfy_r=\bfy_r$ in \eqref{cc2} if $r=0,3\,(\mbox{\rm mod}\,4)$ and $\widetilde\bfy_r$ in \eqref{cc3} if $r=1,2\,(\mbox{\rm mod}\,4)$. We also remark that Algorithm \ref{timedisc} is fully explicit and we will construct an efficient implementation of Algorithm \ref{timedisc} in subsection 3.3 based on recursion formulas.

\subsection{Strong stability }
\label{subsecssp}
In this subsection, we show the strong stability of the proposed time discretizaiton method. For simplicity, we first consider the strong stability without source terms. We denote $(\cdot,\cdot)$ the usual $\ell_2$ inner product and $\|\cdot\|_{\ell_2}$ the $\ell_2$ norm of $\R^M$.

\begin{Thm}\label{thmtimestab}
Let $\bfr=0$ in Algorithm \ref{timedisc}. Then we have
\begin{align*}
\max_{0\le t\le T}\|\widetilde{\bfy}_r\|_{\ell_2}\leq\|\bfy^0\|_{\ell_2}
\end{align*}
under the CFL condition $\lambda=\tau\|\mathbb{D}\|_{\ell_2}\le\lambda(r,\gamma)$, where
\be
& &\lambda(r,\gamma)=\sqrt{6}\ \ \mbox{if } r= 0\,(\mbox{\rm mod}\,4),\label{cfl3}\\
& &\lambda(1,\gamma)= \sqrt{\frac{1-\gamma^2}{2}},\ \ \lambda(r,\gamma)= \sqrt{\frac{2(1-\gamma)}{3-\gamma}}\ \ \mbox{if }r=1\,(\mbox{\rm mod}\,4),r\ge 5, \label{cfl1}\\
& &\lambda(2,\gamma)=\sqrt{\frac{2(4-\gamma^2)}{3}},\ \ \lambda(r,\gamma)=
\sqrt{\frac{6(2-\gamma)}{4-\gamma}}\ \ \mbox{if } r= 2\, (\mbox{\rm mod}\, 4), r\ge 6,\label{cfl2}\\
& &\lambda(r,\gamma)= \sqrt{2}\ \ \mbox{if }r= 3 \,(\mbox{\rm mod}\, 4).\label{cfl4}
\ee
\end{Thm}

\begin{proof}
If $r\equiv 0,3\,(\mbox{\rm mod}\, 4)$, by using $\bfd+\bfd^T=\mathbb{O}$ in \eqref{ddt0} we have $(\bfd\bfy_r,\bfy_r)=0$. By the second step of Algorithm \ref{timedisc} and \eqref{cc2}
\begin{align*}
(\bfy_r'(t),\bfy_r(t))&=(\bfd\bfy_{r-1}(t),\bfy_r(t))\nonumber\\
&=(\bfd(\bfy_{r-1}(t)-\bfy_r(t)),\bfy_r(t))\nonumber\\
&=\Big(-\frac{1}{r!}\bfd^{r+1}\bfy^n(t-t_n)^r,\bfy_r(t)\Big)\nonumber\\
&=\sum_{k=0}^{r}\beta_{r+k}(t-t_n)^{r+k},
\end{align*}
where
\begin{align*}
\beta_{r+k}=(-1)^{r}\frac{1}{r!k!}(\bfy^n,\bfd^{r+k+1}\bfy^n).
\end{align*}
Thus $\beta_{r+k}=0$ if $r+k+1$ is odd.

If $r= 0\,(\mbox{\rm mod}\, 4)$, we assume $r=4s$, $s\ge 0$. Then $\beta_{r+k}\not=0$ only when $k$ is odd, that is, $k=2q+1$ for some $q\ge 0$. Hence
\begin{align}
(\bfy_r'(t),\bfy_r(t))&=\sum_{q=0}^{2s-1}(-1)^{q+1}\frac{1}{r!(2q+1)!}\|\bfd^{2s+q+1}\bfy^n\|_{\ell_2}^2(t-t_n)^{r+2q+1}\nonumber\\
&\leq\sum_{j=0}^{s-1}\Big(-\frac{1}{r!(4j+1)!}+\frac{\lambda^2}{r!(4j+3)!}\Big)\|\bfd^{2s+4j+1}\bfy^n\|_{\ell_2}^2(t-t_n)^{r+4j+1}.\nonumber
\end{align}
This implies $(\bfy_r'(t),\bfy_r(t))\le 0$ if $\lambda^2\leq(4j+3)(4j+2)\ \ \forall j=0,\ldots s-1$. Hence $\displaystyle \max_{t_n\le t\le t_{n+1}}\|\bfy_r\|_{\ell_2}\le\|\bfy^n\|_{\ell_2}$ when $\lambda\le\sqrt 6$. This proves the theorem when \eqref{cfl3} holds.

If $r=3\ (\mbox{\rm mod}\, 4)$, we assume $r=4s+3$, $s\ge 0$. Then $\beta_{r+k}\not=0$ only when $k$ is even, that is, $k=2q$ for some $q\ge 0$. This yields
\begin{align*}
(\bfy_r'(t),\bfy_r(t))&=\sum_{q=0}^{2s+1}(-1)^{q+1}\frac{1}{r!(2q)!}\|\bfd^{2s+q+2}\bfy^n\|_{\ell_2}^2(t-t_n)^{r+2q}\nonumber\\
&\leq\sum_{j=0}^{s}\Big(-\frac{1}{r!(4j)!}+\frac{\lambda^2}{r!(4j+2)!}\Big)\|\bfd^{2s+4j+2}\bfy^n\|_{\ell_2}^2(t-t_n)^{r+4j}.\nonumber
\end{align*}
This implies $(\bfy_r'(t),\bfy_r(t))\le 0$ if $\lambda^2\leq(4j+2)(4j+1)\ \ \forall j=0,\ldots s$. This proves the theorem when \eqref{cfl4} holds.

If $r= 1, 2\ (\mbox{\rm mod}\, 4)$, by using \eqref{ddt0}, we have
\begin{align*}
(\widetilde{\bfy}_r'(t),\widetilde{\bfy}_r(t))=&(\bfd(\gamma_r^r\widetilde{\bfy}_{r-1}(t)+(1-\gamma_r^r)\bfy_{r-1}(t)),\widetilde{\bfy}_r(t))\nonumber \\
=&(\bfd(\gamma_r^r\widetilde{\bfy}_{r-1}(t)+(1-\gamma_r^r)\bfy_{r-1}(t)-\widetilde{\bfy}_r(t)),\widetilde{\bfy}_r(t))\nonumber\\
=&\left(\frac{1-\gamma}{r!(r+\gamma)}\bfd^{r+1}\bfy^n(t-t_n)^r
-\frac{1}{r!(r+\gamma)}\bfd^{r+2}\bfy^n(t-t_n)^{r+1},\widetilde{\bfy}_r(t)\right)\nonumber \\
=&\sum_{k=0}^{r+2}\delta_{r+k}(t-t_n)^{r+k},
\end{align*}
where
\ben
& &\delta_{r+k}=(-1)^{r+1}\Big(\frac{(k+1-\gamma)}{r!k!(r+\gamma)}\bfy^n,\bfd^{r+k+1}\bfy^n\Big),\quad k=0,\ldots, r,\\
& &\delta_{2r+1}=\frac{r+1}{(r+\gamma)^2(r!)^2}\|\bfd^{r+1}\bfy^n\|_{\ell_2}^2,\quad \delta_{2r+2}=0.
\een
It is clear that for $k=1,\cdots,r$, $\delta_{r+k}=0$ if $r+k+1$ is odd since $\bfd+\bfd^T=\mathbb{O}$.

Hence, if $r= 1\ (\mbox{\rm mod}\, 4)$, we assume $r=4s+1$, $s\geq 0$, then $\delta_{r+k}\not=0$ when $k=2q$, $q\ge 0$. We have
\begin{align*}
&(\widetilde{\bfy}_r'(t),\widetilde{\bfy}_r(t))\\
=&\sum_{q=0}^{2s}(-1)^{q+1}\frac{2q+1-\gamma}{(r+\gamma)r!(2q)!}\|D^{2s+q+1}\bfy^n\|_{\ell_2}^2(t-t_n)^{r+2q}\nonumber\\
&+\frac{r+1}{(r+\gamma)^2(r!)^2}\|\bfd^{r+1}\bfy^n\|_{\ell_2}^2(t-t_n)^{2r+1}\nonumber \\
\leq &\sum_{j=0}^{s-1}\left[-\frac{4j+1-\gamma}{(r+\gamma)r!(4j)!}+\frac{4j+3-\gamma}{(r+\gamma)r!(4j+2)!}\lambda^2\right]\|\bfd^{2s+2j+1}\bfy^n\|_{\ell_2}^2(t-t_n)^{r+4j}\nonumber\\
&+\left[-\frac{r(r-\gamma)}{(r+\gamma)(r!)^2}+\frac{r+1}{(r+\gamma)^2(r!)^2}\lambda^2\right]\|\bfd^{r}\bfy^n\|_{\ell_2}^2(t-t_n)^{2r-1}.
\end{align*}
This implies $(\widetilde{\bfy}_r'(t),\widetilde{\bfy}_r(t))\le 0$ if
\begin{align*}
\lambda^2\leq\frac{(4j+1-\gamma)(4j+1)(4j+2)}{4j+3-\gamma}, \ \  j=0,\ldots, s-1, \text{ and } \lambda^2\leq\frac{r(r^2-\gamma^2)}{r+1}.
\end{align*}
This proves the theorem when \eqref{cfl1} holds.

If $r= 2\ (\mbox{\rm mod}\, 4)$, we assume $r=4s+2$, $s\ge 0$, then $\delta_{r+k}\not=0$ when $k=2q+1$ for some $q\ge 0$. We have
\begin{align*}
&(\widetilde{\bfy}_r'(t),\widetilde{\bfy}_r(t))\\
=&\sum_{q=0}^{2s}(-1)^{q+1}\frac{(2q+2-\gamma)}{(r+\gamma)r!(2q+1)!}\|\bfd^{2s+q+2}\bfy^n\|_{\ell_2}^2(t-t_n)^{r+2q+1}\nonumber\\
&+\frac{(r+1)}{(r+\gamma)^2(r!)^2}\|\bfd^{r+1}\bfy^n\|_{\ell_2}^2(t-t_n)^{2r+1}\nonumber\\
\leq&\sum_{j=0}^{s-1}\left[-\frac{4j+2-\gamma}{(r+\gamma)r!(4j+1)!}+\frac{4j+4-\gamma}{(r+\gamma)r!(4j+3)!}\lambda^2\right]\|\bfd^{2s+2j+2}\bfy^n\|_{\ell_2}^2(t-t_n)^{r+4j+1}\nonumber\\
&+\left[-\frac{r(r-\gamma)}{(r+\gamma)(r!)^2}+\frac{r+1}{(r+\gamma)^2(r!)^2}\lambda^2\right]\|\bfd^{r}\bfy^n\|_{\ell_2}^2(t-t_n)^{2r-1}.
\end{align*}
This implies $(\widetilde{\bfy}_r'(t),\widetilde{\bfy}_r(t))\le 0$ if
\begin{align*}
\lambda^2\leq\frac{(4j+2-\gamma)(4j+2)(4j+3)}{4j+4-\gamma}, \ \ j=0,\ldots, s-1, \text{ and } \lambda^2\leq\frac{r(r^2-\gamma^2)}{r+1}.
\end{align*}
This proves the theorem when \eqref{cfl2} holds.
\end{proof}

\begin{rem}
The strong stability in Theorem \ref{thmtimestab} is obtained thanks to the important relationship of the spatial discretization operators $D^-$ and $D^+$ in \eqref{semineg}. The energy conserving mixed finite element methods for solving the Hodge wave equation in Wu and Bai \cite{Wu2021SINUM} all satisfy this relation and thus the explicit time stepping method proposed in this section can be used to solve the ODE systems resulting from their mixed finite element methods.
\end{rem}

\begin{rem}
A careful check of the proof of Theorem \ref{thmtimestab} shows that when $r=2\,({\rm mod}\,4)$, the strong stability holds for any $\gamma\in (0,2)$. In particular, if $\gamma=1$, then $\widetilde{\bfy}_r=\bfy_{r+1}$ by \eqref{cc3}. However, since $\lambda(r,\gamma)$ increases as $\gamma$ decreases, the choice $\gamma\in (0,1)$ is more favorable than the choice $\gamma\in [1,2)$. This is the reason that we also impose $\gamma\in (0,1)$ in Algorithm \ref{timedisc} when $r=2\,({\rm mod}\,4)$, which is the same as the case when $r=1\,({\rm mod}\,4)$. In practical computations, one can choose $\gamma=0.1$.
\end{rem}

\begin{rem}
For the cases of $r= 0,3\,({\rm mod}\,4)$, our time discretization without source terms at nodes is equivalent to the standard $r$ stage $r$ order explicit RK method whose strong stability is proved in \cite[Theorems 4.2 and 4.5]{Sun2019ssp}. Our Theorem \ref{thmtimestab} improves the results in \cite{Sun2019ssp} in that we prove the strong stability in the whole time interval $(0,T)$ instead of only at the times $t=t_n, 1\le n\le N$. Moreover, we also provide explicit upper bounds for the CFL conditions.
\end{rem}

As a direct corollary, we obtain the following strong stability results with the source term.
\begin{coro}\label{coro31}
Let $\bfr\in C^r([0,T];\ell_2)$. Under the CFL condition $\lambda=\tau\|\mathbb{D}\|_{\ell_2}\le\lambda(r,\gamma)$, the solution $\widetilde\bfy_r$ of Algorithm \ref{timedisc} satisfies
\begin{align*}
\max_{0\le t\le T}\|\widetilde{\bfy}_r\|_{\ell_2}\leq\|\bfy^0\|_{\ell_2}+ CT\|\bfr\|_{C^r([0,T];\ell_2)},
\end{align*}
where the constant $C$ is independent of $\tau$ and $r$.
\end{coro}

\begin{proof}
We only prove the case when $r=1,2\,(\mbox{\rm mod}\, 4)$. By \eqref{cc3} and Theorem \ref{thmtimestab}, we have
for $t\in [t_n,t_{n+1}]$, $0\le n\le N-1$,
\begin{align*}
\|\widetilde{\bfy}_r\|_{\ell_2}
\le&\,\left\|\sum_{j=0}^r\frac{1}{j!}\,\bfd^j\bfy^n(t-t_n)^j+\frac{1}{r!(r+\gamma)}\bfd^{r+1}\bfy^n(t-t_n)^{r+1}\right\|_{\ell_2}+{\rm VI}\\
\le&\|\bfy^n\|_{\ell_2}+{\rm VI},
\end{align*}
where
\begin{align*}
{\rm VI}=&\left\|\sum_{j=0}^{r}\sum_{\ell=0}^{j-1}\frac{1}{j!}\bfd^{j-\ell-1}\bfr^{(\ell)}(t_n)(t-t_n)^j+\sum_{\ell=0}^{r}\frac{1}{r!(r+\gamma)}\bfd^{r-\ell}\bfr^{(\ell)}(t_n)(t-t_n)^{r+1}\right\|_{\ell_2}\\
\le&\|\bfr\|_{C^r([t_n,t_{n+1}];\ell_2)}\left(\sum^r_{j=0}\frac{\tau^j}{j!}\sum^{j-1}_{\ell=0}\|\bfd\|^{j-\ell-1}_{\ell_2}+\sum^r_{j=0}\frac{\tau^{r+1}}{r!(r+\gamma)}\|\bfd\|_{\ell_2}^{r-\ell}\right)\\
\le&\frac{r+1}{r+\gamma}\tau\|\bfr\|_{C^r([t_n,t_{n+1}];\ell_2)}\sum^{r+1}_{j=0}\frac{1}{j!}\sum^{j-1}_{\ell=0}\lambda(r,\gamma)^{j-\ell-1}\\
\le &C\tau\|\bfr\|_{C^r([t_n,t_{n+1}];\ell_2)},
\end{align*}
where the constant $C$ depends only on $\lambda(r,\gamma)$ whose upper bound is independent of $r$ by Theorem \ref{thmtimestab}. This completes the proof.
\end{proof}

\subsection{Error estimates}
In this subsection, we give the error estimates of our time discretization method.

\begin{Thm}
\label{errtimedisc}
Assume that the CFL condition $\tau\|\mathbb{D}\|_{\ell_2}\le\lambda(r,\gamma)$ is satisfied. Let $\bfy(t)\in C^{r+1}([0,T];\ell_2)$ be the exact solution of the ODE systems \eqref{ODEsystem}, then we have
\begin{align*}
\max_{0\le t\le T}\|\bfy-\widetilde{\bfy}_r\|_{\ell_2}\leq CT\frac{\tau^{r}}{(r+1)!}\|\bfy^{(r+1)}\|_{C([0,T];\ell_2)},
\end{align*}
where the constant $C$ is independent of $\tau$ and $r$.
\end{Thm}

\begin{proof}
We only prove the case when $r=1,2\, (\rm{mod}\, 4)$. We first recall the following well-known formula for the Taylor expansion which can be easily proved by integration by parts
\be\label{cc4}
& &\bfy(t)=\sum^r_{j=0}\frac 1{j!}\bfy^{(j)}(t_n)(t-t_n)^j+\frac 1{r!}\int^t_{t_n}(t-s)^r\bfy^{(r+1)}(s)ds\ \ \ \ \forall t\in (0,T).
\ee
We define
\begin{align*}
\Pi_r\bfy(t)=\sum_{j=0}^{r}\frac{1}{j!}\bfy^{(j)}(t_n)(t-t_n)^j+\frac{1}{r!(r+\gamma)}\bfy^{(r+1)}(t_n)(t-t_n)^{r+1},
\end{align*}
then by \eqref{cc4}
\begin{align}
\max_{t_n\le t\le t_{n+1}}\|\Pi_r\bfy-\bfy\|_{\ell_2}\leq (2+r^{-1})\frac{\tau^{r+1}}{(r+1)!}\|\bfy^{(r+1)}\|_{C([t_n,t_{n+1}];\ell_2)}.\label{cc5}
\end{align}
By \eqref{cc3} and \eqref{cc1}, we have
\be\label{tayloraprox}
\widetilde{\bfy}_r-\Pi_r\bfy&=&\sum_{j=0}^r\frac{1}{j!}\,\bfd^j(\bfy^n-\bfy(t_n))(t-t_n)^j\\
& &+\frac{1}{r!(r+\gamma)}\bfd^{r+1}(\bfy^n-\bfy(t_n))(t-t_n)^{r+1}.\nn
\ee
From the proof of Theorem \ref{thmtimestab}, we have
\be\label{ineq1}
\max_{t_n\le t \le t_{n+1}}\|\Pi_r\bfy-\widetilde{\bfy}_r\|_{\ell_2}\leq \|\bfy^n-\bfy(t_n)\|_{\ell_2}= \|\widetilde{\bfy}_r(t_n)-\bfy(t_n)\|_{\ell_2}.
\ee
By the triangle inequality and \eqref{cc5}
\begin{align*}
\|\widetilde{\bfy}_r(t_{n+1})-\bfy(t_{n+1})\|_{\ell_2}&\leq \|\widetilde{\bfy}_r(t_{n+1})-\Pi_r\bfy(t_{n+1})\|_{\ell_2}+\|\bfy(t_{n+1})-\Pi_r\bfy(t_{n+1})\|_{\ell_2}\\
&\leq \|\widetilde{\bfy}_r(t_n)-\bfy(t_n)\|_{\ell_2}+\frac{3\tau^{r+1}}{(r+1)!}\|\bfy^{(r+1)}\|_{C([0,T];\ell_2)}.
\end{align*}
This yields, since $\widetilde{\bfy}_r(t_0)=\bfy(t_0)$,
\begin{align*}
\max_{1\le n\le N}\|\widetilde{\bfy}_r(t_{n})-\bfy(t_{n})\|_{\ell_2}&\leq 3n\frac{\tau^{r+1}}{(r+1)!}\|\bfy^{(r+1)}\|_{C([0,T];\ell_2)}.
\end{align*}
This togethers with \eqref{cc5} and \eqref{ineq1} completes the proof.
\end{proof}

To conclude this section, we remark that if $\bfy_h\in C^{r+1}([0,T];\ell_2)$ satisfies
\begin{align}\label{ODEsys2}
\bfy_h'=\bfd\bfy_h+\bfr+\bfr_h\ \ \mbox{in }(0,T),\ \ \ \ \bfy_h(0)=\bfy^0,
\end{align}
where $\bfr_h$ is the error due to some spatial discretization. Let $\widetilde{\bfy}^h_r\in {\bf V}_{r+1}(0,T)$ such that $\widetilde\bfy^h_r(0)=\bfy^0$, and in each time interval $(t_n,t_{n+1})$, $0\le n\le N-1$, $\widetilde\bfy^h_r$ is computed by Algorithm \ref{timedisc} with $\bfy^n=\widetilde\bfy^h_r(t_n)$ and the source $\bfr+\bfr_h\in\R^M$. Then by Theorem \ref{errtimedisc} and Corollary \ref{coro31} we have
\be\label{errtime2}
& &\max_{0\le t\le T}\|\bfy_h-\widetilde{\bfy}_r\|_{\ell_2}\\
&\le& \max_{0\le t\le T}(\|\bfy_h-\widetilde\bfy^h_r\|_{\ell_2}+\|\widetilde\bfy^h_r-\widetilde\bfy_r\|_{\ell_2})\nn\\
&\le& CT\frac{\tau^r}{(r+1)!}\|\bfy^{(r+1)}_h\|_{C([0,T];\ell_2)}+CT\|\bfr_h\|_{C^r([0,T];\ell_2)}.\nn
\ee
This estimate will be used in section \ref{secfull}.

\subsection{Implementation of the time discretization method}
\label{implement}

In this subsection, we propose a recursive implementation of Algorithm \ref{timedisc} in each time interval $(t_n,t_{n+1})$,
$0\le n\le N-1$. Let $\{L_j\}^\infty_{j=0}$ be the standard Legendre polynomials
on the interval $(-1,1)$. Denote $\psi:(-1,1)\to(t_n,t_{n+1})$ the affine transform $\psi(\xi)=\frac{t_n+t_{n+1}}{2}+\frac{t_{n+1}-t_n}{2}\xi\ \ \forall\xi\in (-1,1)$,
then $\{\widetilde L_j\}^\infty_{j=0}$, $\widetilde L_j=L_j\circ\psi^{-1}$, defines a complete orthonormal basis of $L^2(t_n,t_{n+1})$. It follows from the standard identity for Legendre polynomials, cf., e.g., \cite{Bernardi},
\ben
\int^1_{-1}L_k'(\xi)L'_j(\xi)(1-\xi^2)d\xi=\frac{k(k+1)}{k+\frac 12}\delta_{k,j}
\een
that
\begin{align}\label{b1}
\int_{t_n}^{t_{n+1}}\widetilde{L}'_k(t)\widetilde{L}'_j(t)(t-t_n)(t_{n+1}-t)\, dt=\frac{\tau_n}{2}\frac{k(k+1)}{k+\frac 12} \delta_{k,j},
\end{align}
where $\delta_{k,j}$ is the Kronecker delta function. We will also use the recursion relation $(2k+1)L_k=L_{k+1}'-L_{k-1}'$, which implies
\be\label{b2}
(2k+1)\widetilde L_k=\frac{\tau_n} 2(\widetilde L_{k+1}'-\widetilde L_{k-1}').
\ee
We assume
\begin{align}\label{b0}
&\bfy_m(t)=\sum_{j=0}^{r_m}\ba_j^m\widetilde{L}_j(t),\quad \widetilde{\bfy}_m(t)=\sum_{j=0}^{\tilde{r}_m}\tilde{\ba}_{j}^m\widetilde{L}_j(t),\\
& I_{m-1}\bfr(t)=\sum_{j=0}^{m-1}\bfr_j^m\widetilde{L}_j(t),\quad \tilde{I}_{m-1}\bfr(t)=\sum_{j=0}^{m}\tilde{\bfr}_j^m\widetilde{L}_j(t),\label{b01}
\end{align}
where $0\le m\le r$, $r_m=m+r-1$, and $\tilde{r}_m=\max(m+r-1,m+1)$. For simplicity, we set
\be\label{bb}
&\ \ \ \ \ba^m_j=\mathbf{0}\ \ \mbox{for } j>r_m,\ \ \tilde \ba^m_j=\mathbf{0}\ \ \mbox{for }j>\tilde r_m,\\
& \ \ \bfr_j^m=\mathbf{0}\ \ \mbox{for }j>m-1, \ \ \tilde{\bfr}_j^m=\mathbf{0}\ \ \mbox{for }j>m.
\ee

\begin{Thm}
The coefficients of the functions $\bfy_m,\widetilde{\bfy}_m$ in \eqref{b0} can be computed recursively as
\be
& &\ba_0^0=\bfy^n,\ \ \ba_j^0=\mathbf{0},\ \ 1\le j\le r-1,\label{b3}\\
& &\tilde{\ba}_0^0=\bfy^n+\frac{\tau_n}{2\gamma}\,\bfd\bfy^n,\ \ \tilde{\ba}_1^0=\frac{\tau_n}{2\gamma}\,\bfd\bfy^n,\ \ \tilde{\ba}_j^0=\mathbf{0}, \ \ 2\le j\le\tilde r_0,\label{b4}
\ee
and for $1\le m\le r$,
\be
 \ba_{0}^m=\bfy^n-\sum_{j=1}^{r_m}(-1)^j\ba_{j}^m,\ \ \tilde{\ba}_{0}^m=\bfy^n-\sum_{j=1}^{\tilde{r}_m}(-1)^j\tilde{\ba}_{j}^m,\label{a0m}
 \ee
 \begin{align}
 \ba_k^m=\frac{\tau_n}{2}\,\bfd\left(\frac{\ba_{k-1}^{m-1}}{2k-1}-\frac{\ba_{k+1}^{m-1}}{2k+3}\right)+\frac{\tau_n}{2}\left(\frac{\bfr_{k-1}^m}{2k-1}-\frac{\bfr_{k+1}^m}{2k+3}\right), \quad 1\leq k\leq r_m,\label{akm}
 \end{align}
 \be
 &\tilde{\ba}_k^m=\frac{\tau_n}{2}\,\bfd\left[\gamma_m^r\left(\frac{\tilde{\ba}_{k-1}^{m-1}}{2k-1}-\frac{\tilde{\ba}_{k+1}^{m-1}}{2k+3}\right)+(1-\gamma_m^r)\left(\frac{\ba_{k-1}^{m-1}}{2k-1}-\frac{\ba_{k+1}^{m-1}}{2k+3}\right)\right]\label{atildekm}\\
 &\hskip1cm\,+\frac{\tau_n}{2}\left[\gamma_m^r\left(\frac{\tilde{\bfr}_{k-1}^m}{2k-1}-\frac{\tilde{\bfr}_{k+1}^m}{2k+3}\right)+(1-\gamma_m^r)\left(\frac{\bfr_{k-1}^m}{2k-1}-\frac{\bfr_{k+1}^m}{2k+3}\right)\right], \quad 1\leq k\leq \tilde{r}_m.\nn
\ee
\end{Thm}

\begin{proof}
\eqref{b3}-\eqref{b4} follows easily from the definition of $\bfy_0,\widetilde\bfy_0$ in the first step of Algorithm \ref{timedisc} since $\widetilde L_0=1,\widetilde L_1=\psi^{-1}(t)$.
For $1\le m\le r$, by the second {\rev{step}} in Algorithm \ref{timedisc}, we have
\begin{align*}
\sum_{j=0}^{r_m}\ba_j^m\widetilde{L}'_j(t)=\bfd\sum_{j=0}^{r_{m-1}}\ba_j^{m-1}\widetilde{L}_j(t)+\sum_{j=0}^{m-1}\bfr_j^m\widetilde{L}_j(t).
\end{align*}
For any $k\ge 1$, multiply the equation by $(t-t_n)(t_{n+1}-t)\widetilde{L}'_k(t)$ and integrate over $(t_n,t_{n+1})$, we obtain by \eqref{b1} that
\begin{align}
\ba_k^m\frac{\tau_n}{2}\frac{k(k+1)}{k+\frac 12} = &\bfd \sum_{j=0}^{r_{m-1}}\int_{t_n}^{t_{n+1}}\ba_j^{m-1}\widetilde{L}_j(t)(t-t_n)(t_{n+1}-t)\widetilde{L}'_k(t)\, dt\label{aa1}\\
&+\sum_{j=0}^{m-1}\bfr_j^m\int_{t_n}^{t_{n+1}}\widetilde{L}_j(t)(t-t_n)(t_{n+1}-t)\widetilde{L}'_k(t)\, dt.\nn
\end{align}
By using \eqref{b2} we have
\ben
&&\int_{t_n}^{t_{n+1}}\widetilde{L}_j(t)(t-t_n)(t_{n+1}-t)\widetilde{L}'_k(t)\, dt\\
&=&\frac{1}{2j+1}\int_{t_n}^{t_{n+1}}\frac{\tau_n}{2}(\widetilde{L}'_{j+1}(t)-\widetilde{L}'_{j-1}(t))\widetilde{L}'_k(t)(t-t_n)(t_{n+1}-t)\, dt\nonumber\\
&=&\frac{\tau_n^2}{4}\frac{k(k+1)}{k+\frac 12}\frac{1}{2j+1}(\delta_{j+1,k}-\delta_{j-1,k})\nonumber\\
&=&\frac{\tau_n^2}{4}\frac{k(k+1)}{k+\frac 12}\left(\frac{1}{2k-1}\delta_{j+1,k}-\frac{1}{2k+3}\delta_{j-1,k}\right).
\een
Substitute above identity into \eqref{aa1} we obtain \eqref{akm} by using the convention \eqref{bb}.
Finally, since $\bfy_m(t_n)=\bfy^n$, we have by $\widetilde{L}_j(t_n)=L_j(-1)=(-1)^j$ that
\begin{align*}
\ba_{0}^m=\bfy^n-\sum_{j=1}^{r_m}(-1)^j\ba_{j}^m.
\end{align*}
This is the first formula in \eqref{a0m}. The other relations for $\widetilde{\bfy}_m$ can be proved similarly. Here we omit the details. \end{proof}

\setcounter{equation}{0}
\section{The full discretization scheme}
\label{secfull}

We will obtain the fully discrete scheme for solving \eqref{modelproblem} by applying the explicit discrete method developed in last section
to the equivalent ODE system \eqref{ODEs} of the semi-discrete method \eqref{ldgsch1}-\eqref{ldgscH4}.

For any integer $m\ge 1$ and interval $(a,b)\subset (0,T)$, we define the space
\ben
& &V^m(a,b;\fespace\times\fespaceq)=\{(v_h,\bms_h)\in C([a,b];\fespace\times\fespaceq): \\
& &\qquad\quad\qquad (v_h,\bms_h)(\bfx,\cdot)|_{(t_n,t_{n+1})}\in P^m(t_n,t_{n+1})\ \ \mbox{a.e. } \mathbf{x}\in\Om, \ 0\le n\le N-1\}.
\een

The fully discrete scheme for solving \eqref{modelproblem} is to find
\be\label{fulldisc1}
(\tilde{u}_h^r,\tilde{\bfq}_h^r)\in V^{r+1}(0,T;\fespace\times\fespaceq)
\ee
such that $(\tilde{u}_h^r,\tilde{\bfq}_h^r)|_{t=0}=(P_hu_0,\bm{P}_h\bfq_0)$, and in the time interval $(t_n,t_{n+1})$, $0\le n\le N-1$, $(\tilde u^r_h,\tilde{\bfq}^r_h)$ is computed by the following Algorithm \ref{aa2} with $(U^n_h,\mathbf{Q}^n_h)=(\tilde u^r_h(t_n),\tilde{\bfq}^r_h(t_n))$.

\begin{alg}\label{aa2} Given $\gamma\in (0,1)$ and $(U^n_h,\mathbf{Q}^n_h)\in\fespace\times\fespaceq$. \\
$1^{\circ}$ For $m=0$, set $({u}_h^m,{\bfq}^m_h)=(U^n_h,\mathbf{Q}^n_h)$ and find $(\tilde u_h^m,\tilde{\bfq}_h^m)\in V^1(t_n,t_{n+1};\fespace\times\fespaceq)$ such that $(\tilde u_h^m,\tilde{\bfq}^m_h)|_{t=t_n}=(U^n_h,\mathbf{Q}^n_h)$ and
\begin{align*}
&\Big(\frac{1}{\rho c^2}\pa_t\tilde u^m_h,\varphi_h\Big)_{\cam}
=\frac{1}{\gamma}\,\mathcal{H}^-(\mathbf{Q}^n_h,\varphi_h)\ \ \forall \varphi_h\in\fespace, \\
&(\rho\pa_t\tilde{\bfq}^m_h,\bmp_h)_{\cam}=\frac{1}{\gamma}\,\mathcal{H}^+(U^n_h,\bmp_h)\ \ \forall \bmp_h\in\fespaceq.
\end{align*}
$2^\circ$ For $1\le m\le r$, find $(u^m_h,\bfq^m_h)\in V^{m}(t_n,t_{n+1};\fespace\times\fespaceq)$ such that $(u^m_h,\bfq^m_h)|_{t=t_n}=(U^n_h,\mathbf{Q}^n_h)$ and
\begin{align*}
&\Big(\frac{1}{\rho c^2}\partial_t{u}^m_h,\varphi_h\Big)_{\cam}=\mathcal{H}^-(\bfq^{m-1}_h,\varphi_h)+(I_{m-1}f,\varphi_h)_{\cam}\ \ \forall \varphi_h\in\fespace,\\
&(\rho\partial_t{\bfq}^m_h,\bmp_h)_{\cam}=\mathcal{H}^+(u^{m-1}_h,\bmp_h)\ \ \forall \bmp_h\in\fespaceq.
\end{align*}
$3^\circ$ For $1\le m\le r$, find $(\tilde u^m_h,\tilde{\bfq}^m_h)\in V^{m+1}(t_n,t_{n+1};\fespace\times\fespaceq)$ such that $(\tilde u^m_h,\tilde{\bfq}^m_h)|_{t=t_n}=(U^n_h,\mathbf{Q}^n_h)$ and
\begin{align*}
&\Big(\frac{1}{\rho c^2}\partial_t\tilde{u}^m_h,\varphi_h\Big)_{\cam}=\gamma_m^r\mathcal{H}^-(\tilde{\bfq}^{m-1}_h,\varphi_h)
+(1-\gamma_m^r)\mathcal{H}^-({\bfq}^{m-1}_h,\varphi_h)\\
&\hskip 3cm+(\gamma_m^r\tilde{I}_{m-1}f+(1-\gamma_m^r){I}_{m-1}f,\varphi_h)_{\cam}\ \ \forall \varphi_h\in\fespace,\\
&(\rho\partial_t\tilde{\bfq}^m_h,\bmp_h)_{\cam}=\gamma_m^r\mathcal{H}^+(\tilde{u}^{m-1}_h,\bmp_h)
+(1-\gamma_m^r)\mathcal{H}^+({u}^{m-1}_h,\bmp_h)\ \ \forall \bmp_h\in\fespaceq.
\end{align*}
\end{alg}

We remark that $(I_{m-1}f,\varphi_h)_\cM=(I_{m-1}P_hf,\varphi_h)_\cM\ \ \forall\varphi_h\in\fespace$. The source function used in Algorithm \ref{aa2} is in fact $P_hf$. 

Let $(\widetilde{\mathbf{U}}_r^T,\widetilde{\mathbf{Q}}_r^T)^T=\Phi(\tilde u_h^r,\tilde{\bfq}^r_h)$ be the coefficient vector of $(\tilde u^r_h,\tilde{\bfq}^r_h)\in\fespace\times\fespaceq$ defined in \eqref{d2}. Then
$\widetilde{\mathbf{Y}}_r=\mathbb{M}_{\rho,c}^{\frac 12}(\widetilde{\mathbf{U}}_r^T,\widetilde{\mathbf{Q}}_r^T)^T$ is the output of Algorithm \ref{timedisc} at each time interval $(t_n,t_{n+1})$, $0\le n\le N-1$, with
\ben
\bfy^0=\mathbb{M}_{\rho,c}^{\frac 12}(\mathbf{U}_0^T,\mathbf{Q}_0^T)^T,\ \
\mathbf{R}=\mathbb{M}_{\rho,c}^{-\frac 12}\mathbb{M}(\mathbf{F}^T,\bm{0}^T)^T,
\een
where $(\mathbf{U}_0^T, \mathbf{Q}_0^T)^T=\Phi(P_hu_0,\bm{P}_h\bfq_0)$. Obviously,
\be
& &\|\widetilde{\mathbf{Y}}_r\|_{\ell_2}^2=\|(\rho c^2)^{-1/2}\tilde u_h^r\|_{L^2(\Om)}^2+\|\rho^{1/2}\tilde{\bfq}_h^r\|_{L^2(\Om)}^2,\label{d61}\\
& &\|{\mathbf{Y}}^0\|_{\ell_2}^2=\|(\rho c^2)^{-1/2}P_hu_0\|_{L^2(\Om)}^2+\|\rho^{1/2}\bm{P}_h\bfq_0\|_{L^2(\Om)}^2.\label{d62}
\ee

\begin{Thm}\label{thm:4.1}
There exists a constant $C^*$ independent of $p$, $h_K$ for all $K\in\cM$, and $\eta_K$ for all $K\in\cM^\Ga$ such that under the CFL condition $\tau p^2h_{\rm min}^{-1}\Theta\le\lambda(r,\gamma)/C^*$, where $\displaystyle h_{\rm min}=\min_{K\in\cM} h_K$ and $\lambda(r,\gamma)$ is defined in \eqref{cfl3}-\eqref{cfl4}, we have
\ben
& &\max_{0\le t\le T}(\|\tilde u^r_h\|_{L^2(\Om)}^2+\|\tilde{\bfq}^r_h\|_{L^2(\Om)}^2)^{1/2}\\
&\le&(\|P_h u_0\|_{L^2(\Om)}^2+\|\bm{P}_h\bfq_0\|_{L^2(\Om)}^2)^{1/2}+CT\|f\|_{C^r([0,T];L^2(\Om))}.
\een
\end{Thm}

\begin{proof}
Our aim is to use Theorem \ref{thmtimestab} for which we first estimate $\|\mathbb{D}\|_{\ell_2}$. For any $(v_n,\bms_h),(\tilde v_h,\tilde{\bms}_h)\in\fespace\times\fespaceq$, we denote $(\mathbf{V}^T,\bm{\Sigma}^T)^T=\Phi(v_h,\bms_h)$, $(\widetilde{\mathbf{V}}^T,\widetilde{\bm{\Sigma}}^T)^T=\Phi(\tilde v_h,\tilde{\bms}_h)$ the coefficient vectors defined according to \eqref{d2}.
By definition, $\mathbb{D}=\mathbb{M}_{\rho,c}^{-\frac 12}\mathbb{A}\mathbb{M}_{\rho,c}^{-\frac 12}$, we have
\ben
\|\mathbb{D}\|_{\ell_2}&=&\sup_{(\mathbf{V}^T,\bm{\Sigma}^T)^T, (\widetilde{\mathbf{V}}^T,\widetilde{\bm{\Sigma}}^T)^T\in\ell_2}
\frac{(\mathbb{D}(\mathbf{V}^T,\bm{\Sigma}^T)^T, (\widetilde{\mathbf{V}}^T,\widetilde{\bm{\Sigma}}^T)^T)}
{\|(\mathbf{V}^T,\bm{\Sigma}^T)^T\|_{\ell_2}\|(\widetilde{\mathbf{V}}^T,\widetilde{\bm{\Sigma}}^T)^T\|_{\ell_2}}\\
&=&\sup_{(\mathbf{V}^T,\bm{\Sigma}^T)^T, (\widetilde{\mathbf{V}}^T,\widetilde{\bm{\Sigma}}^T)^T\in\ell_2}
\frac{(\mathbb{A}(\mathbf{V}^T,\bm{\Sigma}^T)^T, (\widetilde{\mathbf{V}}^T,\widetilde{\bm{\Sigma}}^T)^T)}
{\|\mathbb{M}_{\rho,c}^{\frac 12}(\mathbf{V}^T,\bm{\Sigma}^T)^T\|_{\ell_2}\|\mathbb{M}_{\rho,c}^{\frac 12}(\widetilde{\mathbf{V}}^T,\widetilde{\bm{\Sigma}}^T)^T\|_{\ell_2}}.
\een
By the inverse inequalities in Lemma \ref{lem:2.1} we obtain
\ben
& &(\mathbb{A}(\mathbf{V}^T,\bm{\Sigma}^T)^T, (\widetilde{\mathbf{V}}^T,\widetilde{\bm{\Sigma}}^T)^T)\\
&=&(\mathbb{D}^-\bm{\Sigma},\widetilde{\mathbf{V}})+(\mathbb{D}^+\mathbf{V},\widetilde{\bm{\Sigma}})\\
&=&\mathcal{H}^-(\bms_h,\tilde v_h)+\mathcal{H}^+(v_h,\tilde{\bms}_h)\\
&\le&Cp^2h^{-1}_{\rm min}\Theta(\|\bms_h\|_{L^2(\Om)}\|{\rev{\tilde v_h}}\|_{L^2(\Om)}+\| {\rev{v_h}}\|_{L^2(\Om)}\|\tilde{\bms}_h\|_{L^2(\Om)}).
\een
On the other hand,
\ben
& &\|\mathbb{M}_{\rho,c}^{\frac 12}(\mathbf{V}^T,\bm{\Sigma}^T)^T\|_{\ell_2}
\ge C(\|v_h\|_{L^2(\Om)}^2+\|\bms_h\|_{L^2(\Om)}^2)^{\frac 12}\\
& &\|\mathbb{M}_{\rho,c}^{\frac 12}(\widetilde{\mathbf{V}}^T,\widetilde{\bm{\Sigma}}^T)^T\|_{\ell_2}
\ge C(\|\tilde v_h\|_{L^2(\Om)}^2+\|\tilde{\bms}_h^2\|_{L^2(\Om)^2})^{\frac 12}.
\een
Thus, by the inequality {\rev{$\frac{ab+cd}{\sqrt{(a^2+c^2)(b^2+d^2)}}\le 1$}}, there exists a constant $C^*$ such that $\|\mathbb{D}\|_{\ell_2}\le C^*p^2h^{-1}_{\rm min}\Theta$. By a similar argument, one can prove
\be\label{d1}
\|\mathbb{M}_{\rho,c}^{-\frac 12}\mathbb{M}^{\frac 12}\|_{\ell_2}\le C.
\ee
Now by Theorem \ref{thmtimestab}, under the CFL condition $\tau p^2h^{-1}_{\rm min}\Theta\le\lambda(r,\gamma)/C^*$, we have
\ben
\|\widetilde{\bfy}_r\|_{\ell_2}\le\|\bfy^0\|_{\ell_2}+CT\|\mathbf{R}\|_{C^r([0,T];\ell_2)}.
\een
Since $\|\mathbb{M}^{\frac 12}(\mathbf{F}^T,\mathbf{0}^T)^T\|_{\ell_2}=\|P_h f\|_{L^2(\Om)}\le \|f\|_{L^2(\Om)}$, by \eqref{d1},
we have that
\begin{align*}
\|\mathbf{R}\|_{C^r([0,T];\ell_2)}=\|\mathbb{M}_{\rho,c}^{-\frac 12}\mathbb{M}(\mathbf{F}^T,\mathbf{0}^T)^T\|_{C^r([0,T];\ell_2)}\le \|f\|_{C^r([0,T];L^2(\Om))}.
\end{align*}
This completes the proof by \eqref{d61}-\eqref{d62}.
\end{proof}

\begin{lem}\label{lem:4.1}
Let $(v,\bms)\in Z_0$ such that $\lj v\rj=0$, $\lj\bms\rj\cdot{\bf n}=0$ on $\Ga$ and $v=0$ on $\pa\Om$. Let $(v_I,\bms_I)=\Upsilon_h(v,\bms)\in\fespace\times\fespaceq$ defined in \eqref{f1}-\eqref{f2}. Then we have
\ben
\|v_I\|_{L^2(\Om)}+\|\bms_I\|_{L^2(\Om)}\le C\|(v,\bms)\|_{Z_0},
\een
where the constant $C$ depends only on the coefficients $\rho,c$.
\end{lem}

\begin{proof}
Notice that $\mathcal{H}^-(\bms,\varphi_h)=(\div\bms,\varphi_h)_\cM$, $\mathcal{H}^+(v,\bmp_h)=(\na v,\bmp_h)_\cM$ since $\lj v\rj=0,$, $\lj\bms\rj\cdot{\bf n}=0$ on $\Ga$. By taking $\varphi_h=v_I,\bmp_h=\bms_I$ in \eqref{f1}-\eqref{f2}, adding the two equations, and using \eqref{ee}, we obtain
\ben
& &\|(\rho c^2)^{-1/2} v_I\|_{L^2(\Om)}^2+\|\rho^{1/2}\bms_I\|_{L^2(\Om)}^2\\
&=&\Big(\frac 1{\rho c^2} v,v_I\Big)_\cM+(\rho\bms,\bms_I)_\cM-(\div\bms,v_I)_\cM-(\na v,\bms_I)_\cM.
\een
The lemma now follows easily.
\end{proof}

The following theorem provides the $hp$ error estimates both in space and time for the fully discrete solution $(\tilde u^r_h,\tilde{\bfq}^r_h)$.

\begin{Thm}\label{Thmerror}
Let $(u,\bfq)\in C^{r+1}([0,T];Z_k)$, $k\ge 1$, be the solution of the problem \eqref{modelproblem} . Then there exists a constant $C^*$ independent $p$, $h_K$ for all $k\in\cM$, and $\eta_K$ for all $K\in\cM^\Ga$ such that under the CFL condition $\tau p^2h_{\rm min}^{-1}\Theta^{\frac 12}\le\lambda(r,\gamma)/C^*$, we have
\ben
& &\max_{0\le t\le T}(\|u-\tilde u^r_h\|_{L^2(\Om)}+\|\bfq-\tilde{\bfq}_h^r\|_{L^2(\Om)})\\
&\le&C(1+T)\left[\frac{\tau^r}{(r+1)!}+\Theta(1+\log p)^2\frac{h^{\min(p,k)}}{p^{k-3/2}}\right]\|(u,\bfq)\|_{C^{r+1}([0,T];Z_k)}.
\een
\end{Thm}

\begin{proof}
Let $(\hat{\bfu}^T, \hat{\bfQ}^T)^T=\Phi(u_I,\bfq_I)$ be the coefficient vector of $(u_I,\bfq_I)=\Upsilon_h(u,\bfq)\in\fespace\times\fespaceq$ defined in \eqref{f1}-\eqref{f2}.
Then $\bfy_h=\mathbb{M}_{\rho,c}^{\frac 12}(\hat{\bfu}^T,\hat{\bfQ}^T)^T$ satisfies
\ben
\bfy'_h=\mathbb{D}\bfy_h+\mathbf{R}+\mathbf{R}_h,
\een
where $\mathbf{R}_h=\mathbb{M}_{\rho,c}^{-\frac 12}\mathbb{M}(\mathbf{G}_u^T,\mathbf{G}_{\bfq}^T)^T$ with $(\mathbf{G}_u^T,\mathbf{G}_{\bfq}^T)^T=\Phi(R_u,\mathbf{R}_{\bm q})$ the coefficient vector of $(R_u,\mathbf{R}_{\bfq})\in\fespace\times\fespaceq$ in \eqref{d3}-\eqref{d4}. By \eqref{errtime2} we have
\begin{align}\label{d5}
\max_{0\le t\le T}\|\bfy_h-\widetilde{\bfy}_r\|_{\ell_2}\leq &CT\frac{\tau^{r}}{(r+1)!}\|\bfy_h^{(r+1)}\|_{C([0,T];\ell_2)}+CT\|\bfr_h\|_{C^r([0,T];\ell_2)}.
\end{align}
Similar to \eqref{d61}, it is easy to see  that $\|\bfy_h-\widetilde{\bfy}_r\|_{\ell_2}^2=\|(\rho c^2)^{-1/2}(u_I-\tilde u^r_h)\|_{L^2(\Om)}^2+\|\rho^{1/2}(\bfq_I-\tilde{\bfq}^r_h)\|_{L^2(\Om)}^2$. By Lemma \ref{lem:4.1},
\be\label{d6}
\|\bfy_h^{(r+1)}\|_{C([0,T];\ell_2)}
&\le&C(\|u_I\|_{C^{r+1}([0,T];L^2(\Om))}+\|\bfq_I\|_{C^{r+1}([0,T];L^2(\Om))})\\
&\le&C\|(u,\bfq)\|_{C^{r+1}([0,T];Z_0)}.\nn
\ee
Again by \eqref{d1}, \eqref{d3}, \eqref{d4} and using Lemma \ref{lem:key}, for $0\le\ell\le r$, we have
\ben
\|\bfr_h^{(\ell)}\|_{\ell_2}&\le&C\|\mathbb{M}^{1/2}((\mathbf{G}^{(\ell)}_u)^T,(\mathbf{G}^{(\ell)}_{\bfq})^T)^T\|_{\ell_2}\\
&\le&C\left(\Big\|\frac{\pa^\ell R_u}{\pa t^\ell}\Big\|_{L^2(\Om)}+\|\frac{\pa^\ell\bfr_\bfq}{\pa t^\ell}\Big\|_{L^2(\Om)}\right)\\
&\le&C\sum_{j=\ell,\ell+1}\left(\Big\|\frac{\pa^j}{\pa t^j}(u-u_I)\Big\|_{L^2(\Om)}+
\Big\|\frac{\pa^j}{\pa t^j}(\bfq-\bfq_I)\Big\|_{L^2(\Om)}\right)\\
&\le&C\Theta(1+\log p)^2\frac{h^{\min(p,k)}}{p^{k-3/2}}\sum_{j=\ell,\ell+1}\Big\|\frac{\pa^j}{\pa t^j}(u,\bfq)\Big\|_{Z_k}.
\een
This, together with \eqref{d5}-\eqref{d6},implies
\ben
& &\max_{0\le t\le T}(\|u_I-\tilde u_h^r\|_{L^2(\Om)}+\|\bfq_I-\tilde\bfq_h^r\|_{L^2(\Om)})\\
&\le&CT\left[\frac{\tau^r}{(r+1)!}+\Theta(1+\log p)^2\frac{h^{\min(p,k)}}{p^{k-3/2}}\right]\|(u,\bfq)\|_{C^{r+1}([0,T];Z_k)}.
\een
The theorem follows by using Lemma \ref{lem:key}.
\end{proof}

\setcounter{equation}{0}
\section{Numerical examples}
\label{secnum}

In this section, we provide some numerical examples to verify our theoretical results. The computations are carried out using MATLAB on a workstation with Intel(R) Core(TM) i9-10885H CPU 2.40GHz and 64GB memory.

The shape functions of $X_p(\cM)$ are constructed as following. For the quadrilateral element $K\in\cM\backslash\cM^\Ga$, we use the Lagrangian interpolation polynomials at the Gauss-Lobatto-Legendre quadrature points as the local basis functions, which are the standard quadrilateral spectral element. For the interface element $K\in\cM^\Ga$, the shape functions in each possibly curved triangle $\widetilde K_{ij}^h$, $i=1,2,j=1,\cdots, J_i^K$, are formed from the shape functions in $K_{ij}^h$ by the mapping $\Lambda_K:U_p(K)\to W_p(K)$. On the triangle $K_{ij}^h$ we use the Lobatto interpolation grid on the triangle in Blyth and Pozrikidis \cite{Blyth} to construct the Lagragian interpolation functions whose nodes along the boundary of the triangle are the one-dimensional Gauss-Lobatto points which conform with the standard quadrilateral spectral elements. The approach in Stolfo et al \cite{Stolfo2016FEAD} is used to treat the hanging nodes. The shape functions in $W_p(\cM)$ are constructed similarly without enforcing the conformity along the edges in $\cE^{\rm side}$.
For elements with curved edges, we use Stokes formula to convert volume integrals to line integrals when computing the local stiffness matrix.

The CFL constant in Theorem \ref{thmtimestab} is taken as $\lambda(p,0.1)$, then the time step is taken as $\tau=0.1\frac{\lambda(p,0.1)h_{\rm min}}{p^2\Theta^{1/2}}$. The numerical errors are measured in the energy norm at the terminal time, that is,
\begin{align*}
E_{en}:=(\|(u- \tilde{u}^{p}_h)(\cdot,T)\|_{L^2(\Om)}^2+{\rev{\|(\bfq-\tilde{\bfq}^{p}_h)(\cdot,T)\|_{L^2(\Om)}^2}})^{1/2}.
\end{align*}

\begin{exmp}\label{example2}
$($Traveling wave$)$ We consider the wave equation \eqref{modelproblem} with $\rho_1=\rho_2=1$, $c_1=c_2=1$. The computational domain is $\Omega=(-2,2)\times(-2,2)$. The source term $f(x,y,t)$ is chosen such that the exact solution of \eqref{modelproblem} is
\begin{align*}
u(x,y,t)=\sin(\sqrt{2}\pi t+2\pi x)\sin(4\pi x)\sin(4\pi y),
\end{align*}
and $\bfq(x,y,t)$ is computed by \eqref{modelproblem} with the initial condition
\begin{align*}
\bfq_0=(-\frac{1}{\sqrt{2}}(\cos(2\pi x)+3\cos(6\pi x))\sin(4\pi y),-2\sqrt{2}\cos(2\pi x)\sin(4\pi x)\cos(4 \pi y))^T.
\end{align*}
\end{exmp}

There is no interface in this example. We use this example to show that our explicit time finite element method can also be applied when standard conforming spatial discretization methods for discretizing the pressure are used to solve the wave equations.

We tested polynomial finite element spaces $p=1,2,3,4,5$ on $N\times N$ uniform meshes and the terminal time $T=1.0$. The orders of the energy error are shown in Table \ref{tab2}. The numerical results verify our theoretical findings. From Table \ref{tab2_2}, we can clearly observe that high-order schemes are more efficient than the low-order schemes in terms of the number of degrees of freedom (\#DoFs).

\begin{table}[!ht]\centering
	\caption{Example \ref{example2}: numerical errors and orders on uniform meshes.}\label{tab2}
\resizebox{\textwidth}{1.8cm}{
\begin{tabular}{|c|cc|cc|cc|cc|cc|}
  \hline.
 &\multicolumn{2}{|c|}{$p=1$}&\multicolumn{2}{|c|}{$p=2$}&\multicolumn{2}{|c|}{$p=3$} &\multicolumn{2}{|c|}{$p=4$} &\multicolumn{2}{|c|}{$p=5$}\\\hline
  $h$ & $E_{en}$ & order & $E_{en}$ & order & $E_{en}$ & order  & $E_{en}$ & order & $E_{en}$ & order \\ \hline
  $2/5 $  &2.26E+00&        --  & 1.77E+00&--       &	1.41E+00&--         &	5.23E-01&--         &	2.36E-01&--      \\
  $1/5 $  &1.50E+00&     0.59	& 6.86E-01&     1.37&	2.12E-01&     2.73	&	4.73E-02&     3.47  &	1.07E-02&     4.46\\
  $1/10$  &6.21E-01&     1.27	& 1.86E-01&     1.89&	2.96E-02&     2.84	&	3.68E-03&     3.68  &	3.32E-04&     5.01\\
  $1/20$  &1.66E-01&     1.90	& 5.25E-02&     1.82&	4.16E-03&     2.83	&	2.32E-04&     3.99  &	1.08E-05&     4.94\\
  $1/40$  &4.74E-02&     1.81	& 1.34E-02&     1.98&	5.02E-04&     3.05  &	1.49E-05&     3.96  &	3.48E-07&     4.95\\ \hline
\end{tabular}
}
\end{table}

\begin{table}[!ht]\centering
\caption{Example \ref{example2}: numerical {\rev{errors}} in terms of \#DoFs.}\label{tab2_2}
\begin{tabular}{|c|c|c|c|c|c|}
  \hline
   & $p=1$  & $p=2$  & $p=3$ & $p=4$ & $p=5$\\\hline
  \#DoFs & 204800  & 201684 & 201640 & 196625 & 194400 \\\hline
  $E_{en}$ & 4.74E-02 & 3.46E-02  &   5.52E-03& 1.06E-03 &1.98E-04\\
  \hline
\end{tabular}
\end{table}

\begin{exmp}\label{example3}
We consider the interface $\Gamma$ is a circle of radius $r_0=1.1$. We take $\Omega=(-2,2)\times(-2,2)$, $\Omega_1=\{(x,y)\in \Omega: \sqrt{x^2+y^2}<r_0\}$, and $\Omega_2=\Omega \setminus \bar{\Omega}_1$. We consider the wave equation \eqref{modelproblem} with $\rho_1=1/10,\, \rho_2=1$, $c_1=c_2= 1$, and the source $f(x,y,t)$ is chosen such that the exact solution is
\begin{align*}
&u(x,y,t)=\left\{\begin{array}{cc}
\cos(t)(-1+{\rm exp}(-r_0+r))\sin(\pi x)^2\sin(\pi y)^2 & \text{ in } \Omega_1,\\
\\
10\cos(t)(-r_0+r)\sin(\pi x)^2\sin(\pi y)^2 & \text{ in } \Omega_2,
\end{array}\right.
\end{align*}
where $r=\sqrt{x^2+y^2}$. $\bfq(x,y,t)$ is computed by \eqref{modelproblem} with the initial condtion $\bfq_0=0$.
\end{exmp}
In this example, we use Algorithm 7 in \cite{ChenLiu2022} to generate the induced mesh satisfying Assumptions (H1) and (H3) and the interface deviation $\eta_K\leq\eta_0=0.05$ for all $K\in \mathcal{M}^\Gamma$ starting from an initial uniform mesh $\cT_0$ of size $h$. By Theorem \ref{thm:A1}, the Assumption (H2) is also satisfied.

An illustration of the mesh with $h=1/2$ and $\eta_0=0.05$ is demonstrated in Fig \ref{fig_mesh_exmp_one_circle}. We tested finite element spaces with $p=3,4,5$ and the terminal time $T=1.0$. Table \ref{tab3} shows clearly the optimal convergence rates of the method, which confirm our theoretical results.

\begin{figure}[!ht]
\begin{minipage}[c]{0.45\textwidth}
\includegraphics[width=0.9\textwidth, height = 0.9\textwidth]{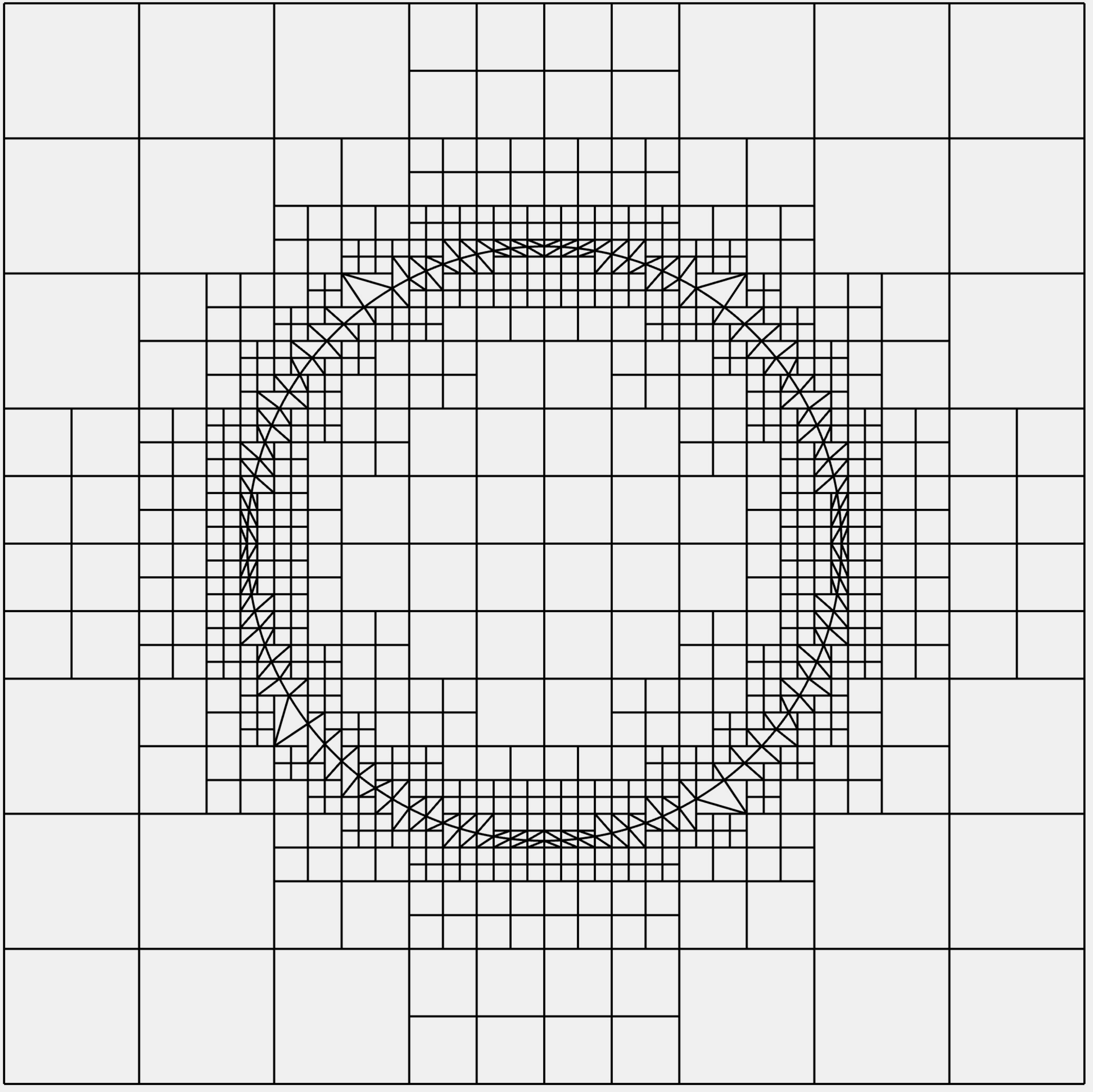}
\end{minipage}
\begin{minipage}[c]{0.45\textwidth}
\includegraphics[width=0.9\textwidth, height = 0.9\textwidth]{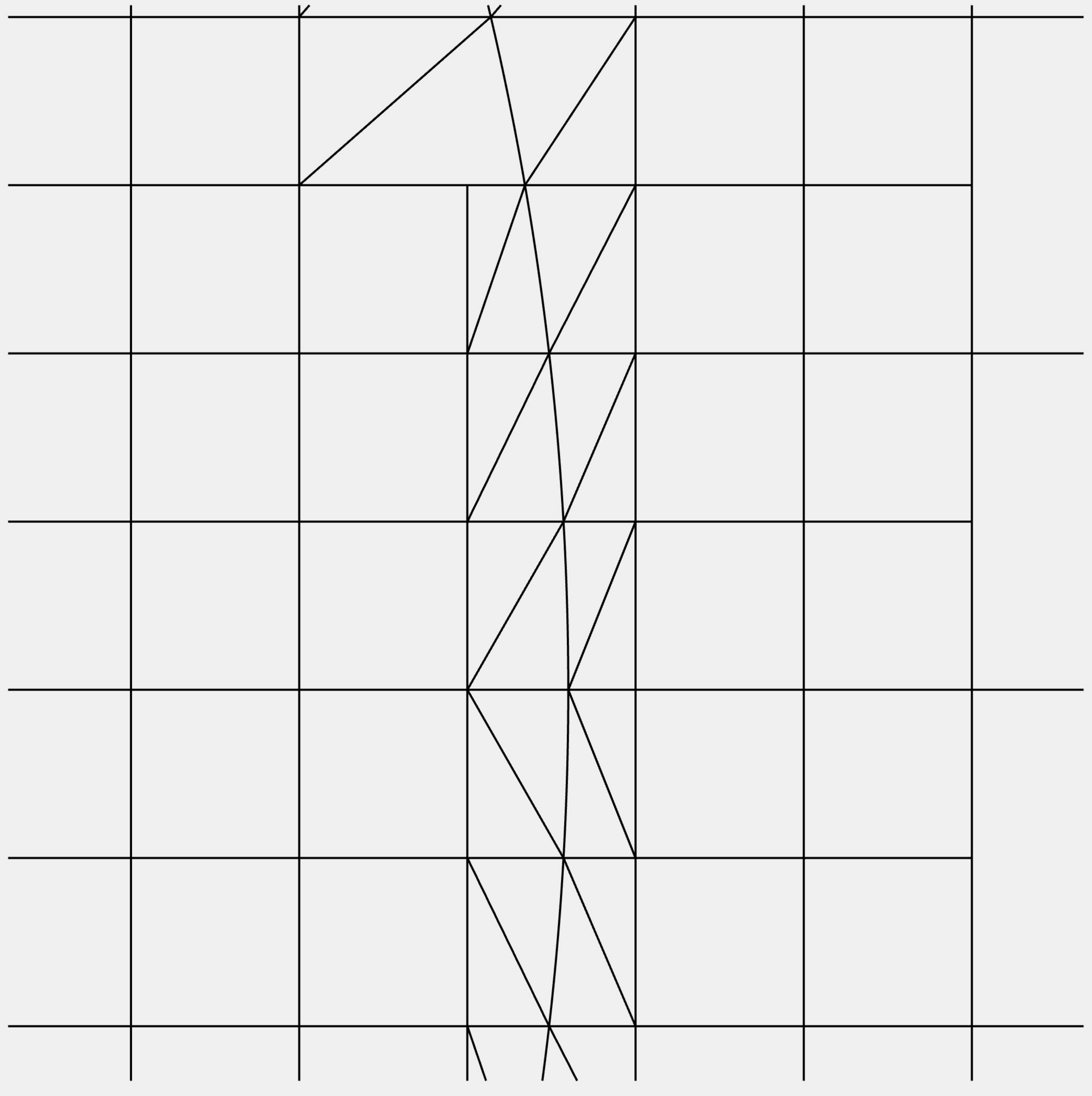}
\end{minipage}
\caption{\label{fig_mesh_exmp_one_circle}  Illustration of the computational domain and the mesh (left) and the corresponding zoomed local mesh (right) with $h=1/2$ and $\eta_0=0.05$ in Example \ref{example3}. }
\end{figure}

\begin{table}[!ht]\centering
	\caption{Example \ref{example3}: numerical errors and orders, $\eta_0= 0.05$.}\label{tab3}
\begin{tabular}{|c|cc|cc|cc|}
  \hline
 &\multicolumn{2}{|c|}{$p=3$}&\multicolumn{2}{|c|}{$p=4$}&\multicolumn{2}{|c|}{$p=5$}\\\hline
  $h$    & $E_{en}$ & order& $E_{en}$ & order & $E_{en}$ & order  \\ \hline
 $1/2$  & 6.25E-02 & -- & 1.85E-02 & --  &2.99E-03  &  --  \\
 $1/4$ & 1.01E-02 & 2.63	& 1.27E-03 & 3.86&1.06E-04  &4.82   \\
 $1/8$  & 1.38E-03 & 2.98	& 7.82E-05 & 4.02& 3.33E-06 &4.99 \\
 $1/16$  &1.74E-04  & 2.99	& 4.86E-06 & 4.01&1.04E-07  &5.00\\
  $1/32$  & 2.17E-05 & 3.00	& 3.03E-07 & 4.00&3.25E-09 &5.00  \\  \hline
\end{tabular}
\end{table}

\begin{exmp}\label{example4}
We assume the interface $\Gamma$ is the union of two closely located circles of radius $r_0=0.51$. We take $\Omega=(-2,2)\times(-2,2)$, $\Omega_1=\{(x,y)\in\Omega:\sqrt{(x-x_1)^2+y^2}< r_0 \text{ or }\sqrt{(x-x_2)^2+y^2}< r_0\}$, which is the union of two disks, and $\Omega_2=\Omega\backslash \bar{\Omega}_1$. Here $x_1=-0.52$ and $x_2=0.52$. The distance between two circles is $d=0.02$. We consider the wave equation \eqref{modelproblem} with $\rho_1=1/2,\, \rho_2=1$, $c_1=c_2= 1$, and the source $f(x,y,t)$ is chosen such that the exact solutions is
\begin{align*}
&u(x,y,t)=\left\{\begin{array}{cc}
\cos(3t)\sin(r_1^2-r^2)\sin(r_2^2-r^2)\sin(3\pi x)^3\sin(3\pi y)^3 & \text{ in } \Omega_1,\\
\\
2\cos(3t)\sin(r_1^2-r^2)\sin(r_2^2-r^2)\sin(3\pi x)^3\sin(3\pi y)^3 & \text{ in } \Omega_2,
\end{array}\right.
\end{align*}
where $r_1=\sqrt{(x-x_1)^2+y^2},\quad r_2=\sqrt{(x-x_2)^2+y^2}$.
$\bfq(x,y,t)$ is computed by \eqref{modelproblem} with the initial condition $\bfq_0=0$.
\end{exmp}
Note that these two circles are close but not tangent. We again use Algorithm 7 in \cite{ChenLiu2022} to generate the induced mesh satisfying Assumptions (H1) and (H3) and the interface deviation $\eta_K\leq\eta_0=0.05$ for all $K\in \mathcal{M}^\Gamma$ starting from an initial uniform mesh $\cT_0$ of size $h$. An illustration of the mesh with $h=1/4$ and $\eta_0=0.05$ is demonstrated in Fig \ref{fig_mesh_exmp_two_circles}.

We tested finite element spaces with $p=3,4,5$ and the terminal time $T=1.0$. Table \ref{tab4} shows clearly the optimal convergence rates of the method, which confirm our theoretical results.
We remark, however, since the minimum size of the mesh $h_{\rm min}$ is smaller for a well resolved mesh, the computation is more expensive as the result of smaller time step due to the CFL condition. One possible remedy, which deserves further investigation, is the methods of local time stepping for which we refer to the recent works Carle et al \cite{Carle2020}, Grote et al \cite{Grote2021} and the references therein. 

\begin{figure}[!ht]
\begin{minipage}[c]{0.45\textwidth}
\includegraphics[width=0.9\textwidth, height = 0.9\textwidth]{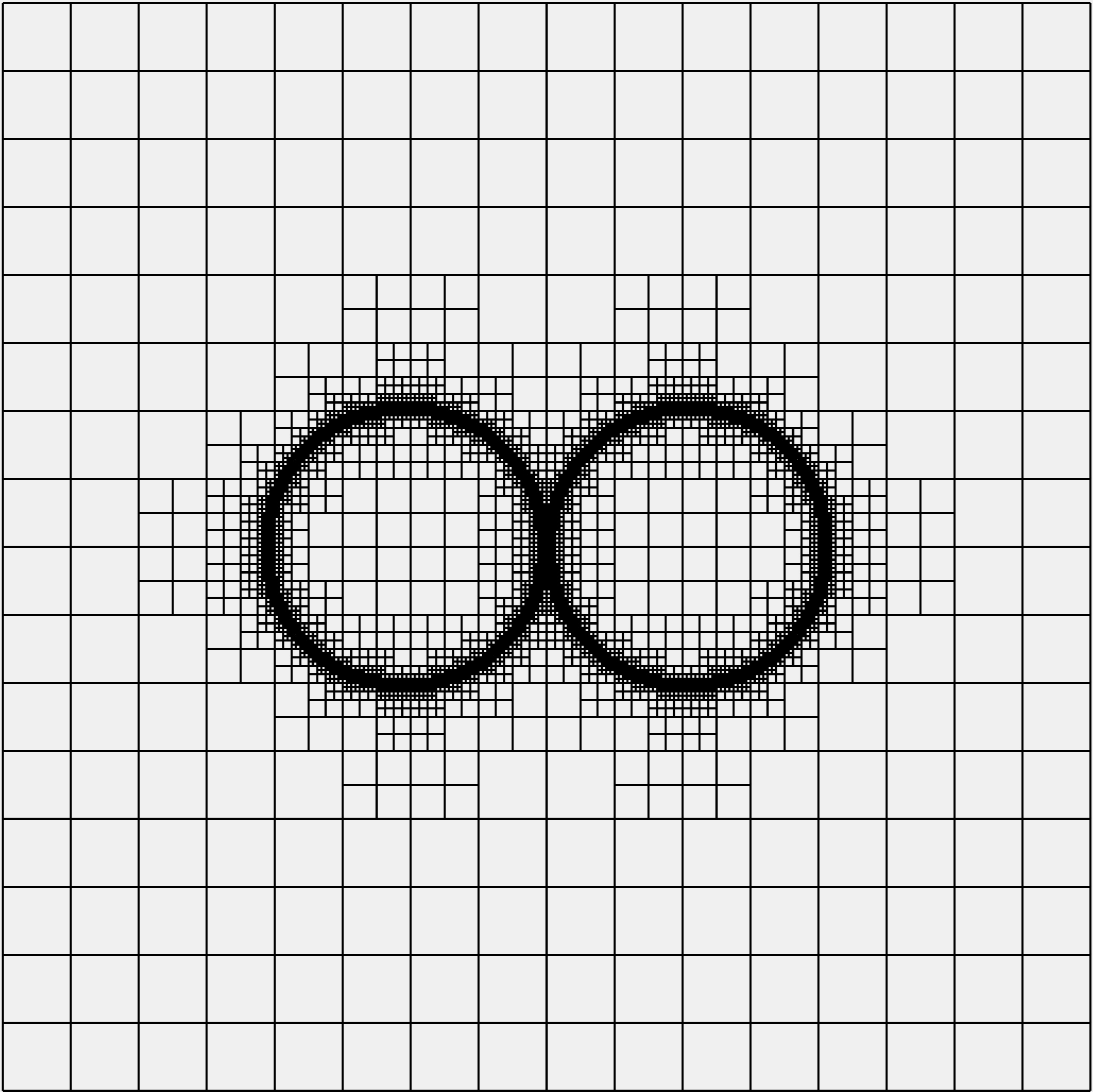}
\end{minipage}
\begin{minipage}[c]{0.45\textwidth}
\includegraphics[width=0.9\textwidth, height = 0.9\textwidth]{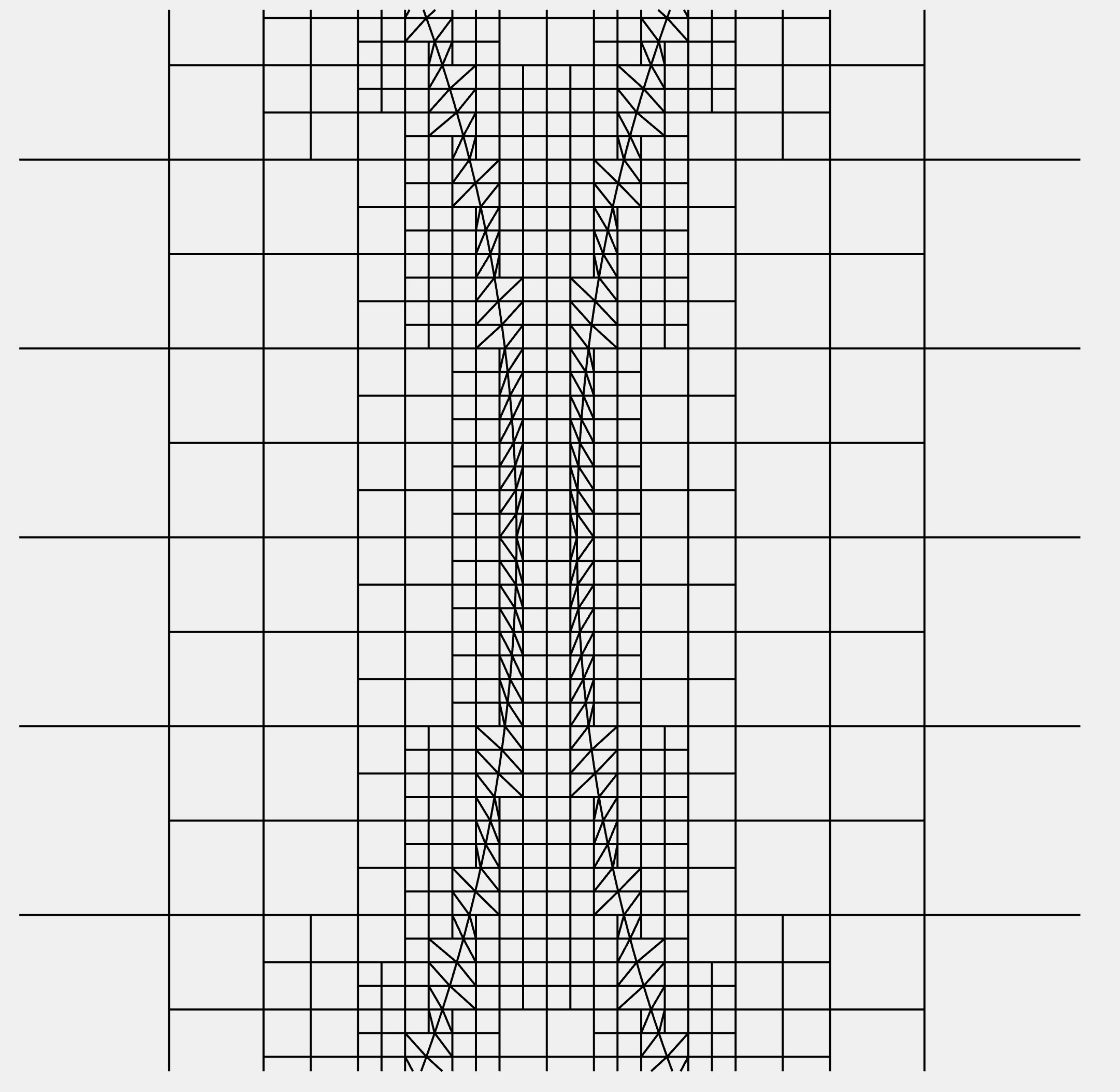}
\end{minipage}
\caption{\label{fig_mesh_exmp_two_circles}  Illustration of the computational domain and the mesh (left) and the corresponding zoomed local mesh (right) with $h=1/4$ and $\eta_0=0.05$ in Example \ref{example4}. }
\end{figure}

\begin{table}[!ht]\centering
	\caption{Example \ref{example4}: numerical errors and orders, $\eta_0= 0.05$.}\label{tab4}
\begin{tabular}{|c|cc|cc|cc|}
  \hline
 &\multicolumn{2}{|c|}{$p=3$}&\multicolumn{2}{|c|}{$p=4$}&\multicolumn{2}{|c|}{$p=5$}\\\hline
  $h$    & $E_{en}$ & order& $E_{en}$ & order & $E_{en}$ & order  \\ \hline
 $1/2$  & 5.87E-01 & -- & 4.52E-01 & --  & 3.50E-01 &  --  \\
 $1/4$ & 3.66E-01 & 0.68	& 1.75E-01 & 1.21& 7.60E-02 & 2.20  \\
 $1/8$  & 6.75E-02 & 2.44 & 1.71E-02 & 3.28& 3.40E-03 & 4.48\\
 $1/16$  & 9.37E-03 & 2.85& 1.15E-03 & 3.84& 1.19E-04 & 4.84  \\
  $1/32$  & 1.17E-03 & 2.99& 7.26E-05 & 4.00& 3.73E-06 & 4.99 \\ \hline
\end{tabular}
\end{table}

\section{Appendix: The Assumption (H2)}
In this section we show that the induced mesh obtained by the merging algorithm developed in \cite[Algorithm 6]{ChenLiu2022} satisfies the assumption (H2). The merging algorithm is based on the concept of the admissible chain of interface elements, the classification of patterns for merging elements, and appropriate ordering in generating macro-elements from the patterns so that the reliability of the algorithm can be proved. In the following, we first recall the concept of the admissible chain and five types of patterns of merging elements in \cite{ChenLiu2022} and then show that any algorithm generating macro-elements from the admissible chain of interface elements by using the patterns will output an induced mesh which satisfies the Assumption (H2).

Since the interface intersects the boundary of $K$ twice at different sides (including the end points), there are only four possible ways for the interface to intersects the element as shown in Fig.\ref{fig:3.1}. We denote $\cT_1$ the set of interface elements shown in Fig.\ref{fig:3.1}(a), $\cT_2$ the set of interface elements shown in Fig.\ref{fig:3.1}(b) and (c), and $\cT_3$ the set of interface elements shown in Fig.\ref{fig:3.1}(d). By Definition \ref{def:2.1}, each element in $\cT_3$ is a large element. Thus we only need to consider the merging of type $\cT_1$ and $\cT_2$ elements.

\begin{figure}
\centering
\includegraphics[width=0.8\textwidth]{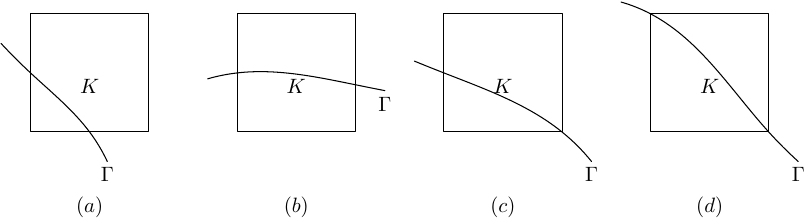}
\caption{Different types of interface elements. The type 2 elements include elements illustrated in (b) and (c). }\label{fig:3.1}
\end{figure}

A chain of interface elements $\mathfrak{C}=\{G_1\rightarrow G_2 \rightarrow\cdots \rightarrow G_n\}$ orderly consists of $n$ interface elements $G_i\in\cT^\Gamma$, $i=1,\cdots,n$, such that $\bar\Gamma_{G_i}\cup\bar\Gamma_{G_{i+1}}$ is a continuous curve, $1\le i\le n-1$. We call $n$ the length of $\mathfrak{C}$ and denote $\mathfrak{C}\{i\}=G_i$, $i=1,\cdots,n$.

For any element $K\in\cT$, we call $N(K)\in\cT$ a neighboring element of $K$ if $K$ and $N(K)$ share a common side, and $D(K)\in\cT$ a diagonal element of $K$ if $K$ and $D(K)$ only share one common vertex. Set $\cS(K)_0=\{K\}$, and for $j\ge 1$, denote $\cS(K)_j=\{K''\in\cT:\exists\,K'\in\cS(K)_{j-1}\ \mbox{such that }\bar K''\cap\bar K'\not=\emptyset\}$, that is, $\cS(K)_j$ is the set of all $k$-th layer
elements surrounding $K$, $0\le k\le j$. Obviously, $\cS(K)_0\subset\cS(K)_1\subset\cdots\subset\cS(K)_j$ for any $ j\ge 1$. The following concept is introduced in \cite[Definition 3.1]{ChenLiu2022}.

\begin{Def}\label{def:3.1}
A chain of interface elements $\mathfrak{C}$ is called admissible if the following rules are satisfied.
\begin{enumerate}
\item[$1.$] For any $K\in\mathfrak{C}$, all elements in $\cS(K)_2$ have the same size as
that of $K$.
\item[$2.$] If $K\in\mathfrak{C}$ has a side $e$ such that $\bar e\subset \Om_i$, then $e$ must be a side of some neighboring element $N(K)\subset\Om_i$, $i=1,2$.
\item[$3.$] Any elements $K\in\cT\backslash\cT^\Gamma$ can be neighboring at most two elements in $\mathfrak{C}$.
\item[$4.$] For any $K\subset\Om_i$, the interface elements in {$\cS(K)_j$, $j=1,2$,} must be connected in the sense that the interior of the closed
set $\cup\{\bar G: G\in\cS(K)_{j}\cap\cT^\Ga\}$ is a connected domain.
\end{enumerate}
\end{Def}

The four rules in the definition can be easily satisfied if the mesh is well refined near the interface. The guideline for introducing the rules is to have enough non-interface elements in the vicinity of each interface element so that the merging algorithm is successful. We refer to \cite{ChenLiu2022} for further details on the properties of the rules.

A pattern is a set of interface elements and their neighboring and diagonal elements whose union consists of a macro-element. We introduce five types of patterns according to the combination of different types of interface elements, see Fig.\ref{pattern1} and Fig.\ref{pattern3}. A macro-element $M$ is generated by the pattern of type 1 if it is a rectangle including the interface elements $N(K)_1,K,N(K)_2$ in Fig.\ref{pattern1} (left) so that $M$ is a large element with respect to both domains $\Om_1,\Om_2$. $M$ can consist of $N(K)_1,K,N(K)_2, D(K)$ if $|e_1|, |e_2|$ are large. It can consist of $N(K)_1,K,N(K)_2, D(K), G_1,G_2$ if $|e_1|$ is small but $|e_2|$ is large. It can consist of all $9$ elements if both $|e_1|, |e_2|$ are small. One can find the precise definition in \cite{ChenLiu2022}. The macro-elements generated by other types of patterns are defined similarly in \cite{ChenLiu2022}. We denote $\cP_j$ the collection of patterns of type $j$, $j=1,2,3,4,5$.

\begin{figure}
  \centering
\includegraphics[width=0.6\textwidth]{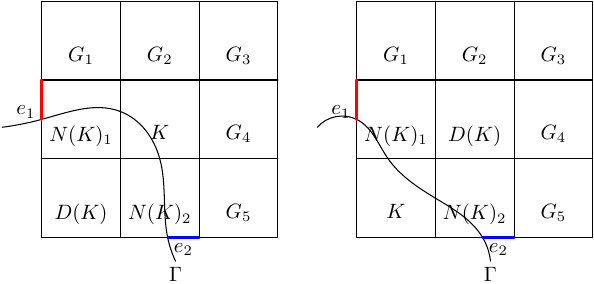}
  \caption{Illustration of type 1 (left) and type 2 (right) patterns.}\label{pattern1}
\end{figure}

\begin{figure}
  \centering
\includegraphics[width=\textwidth]{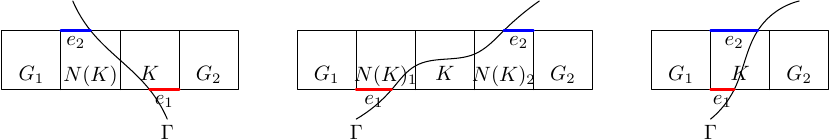}
  \caption{Illustration of type 3 (left), type 4 (middle) and type 5 (right) patterns.}\label{pattern3}
\end{figure}

The following theorem is the main result of this section.

\begin{Thm}\label{thm:A1}
Let $\cM$ be an induced mesh whose macro-elements are generated by five types of patterns illustrated in Fig.\ref{pattern1}-Fig.\ref{pattern3} from an admissible chain of interface elements $\mathfrak{C}$ of length $n\ge 2$ with $\mathfrak{C}(1)=\mathfrak{C}(n)$. Then for any $e=\pa M\cap\pa M'$, where $M,M'\in\cM$, let $f,f'$ be respectively the sides of $M,M'$ including $e$, we have either (1) $f\subset f'$ or $f'\subset f$; or (2) $e\cap\Ga\not=\emptyset$.
\end{Thm}

Let $\cS(\Ga)_0=\cT^\Ga$. For $j\ge 1$, let $\cS(\Ga)_j=\{\cS(K)_j:K\in\cT^\Ga\}\backslash\cS(\Ga)_{j-1}$, which is the $j$-th layer elements surrounding the interface $\Ga$. By the construction of patterns, only elements in $\cup^2_{j=0}\cS(\Ga)_j$ can be merged to generate the macro-elements. If $e=\pa M\cap M'$, where $M,M'$ are macro-elements, then there exists a pair of neighboring elements $D,D'\in\cT$ whose common side $\pa D\cap\pa D'\subset e$. We call $(D,D')$ is subordinate to $(M,M')$ with respect to $e$. The following lemma shows that $D,D'$ cannot be elements in $\cS(\Ga)_0,\cS(\Ga)_2$ if $e\cap\Ga=\emptyset$.

\begin{lem}\label{lem:A1}
Let $e=\pa M\cap\pa M'$, where $M,M'$ are macro-elements generated by patterns $\cP,\cP'\in\{\cP_1,\cdots,\cP_5\}$, and $(D,D')$ is subordinate to $(M,M')$ with respect to $e$. If $e\cap\Ga=\emptyset$, then $D,D'\not\in\cS(\Ga)_0\cup\cS(\Ga)_2$.
\end{lem}

\begin{proof}
We first show that $D,D'\not\in\cS(\Ga)_2$. Assume $D\in\cS(\Ga)_2$, then  by the construction of the patterns, $\cP$ is pattern of type 2 and $D=G_3$ in Fig.\ref{pattern1} (right). By the Rule 4 of the admissible chain, $\cS(G_3)_2$ cannot have any interface elements other than $N(K)_1,K,N(K)_2$. Thus the neighboring elements $D'\not\in\cS(\Ga)_0\cup\cS(\Ga)_1$. If $D'\in\cS(\Ga)_2$, then  $M'$ is also constructed as a pattern of type $2$, They are two possibilities, where $D'=G_3'$, see Fig.\ref{fig:A1}. In the left figure, $N(K')_1\in\cS(G_3)_2$ and in the right figure, $N(K')_2\in\cS(G_3)_2$. They contradict to the Rule 4 of the admissible chain. Thus $D\not\in\cS(\Ga)_2$. Similarly, $D'\not\in\cS(\Ga)_2$.

Now we show $D, D'\not\in\cS(\Ga)_0$. If $D\in\cS(\Ga)_0=\cT^\Ga$, by the Rule 2 of the admissible chain, $D'$ cannot be an interface element. Assume $D'\subset\Om_i$, $i=1,2$. There are two possibilities.

$1^\circ$ If $D\in\cT_1$, then $N(D)$ must be in $\cT_2$ by the Rule 2 of the admissible chain and $e\cap\Ga=\emptyset$, as shown in Fig.\ref{fig:A2} (left). By the Rule 4 of the admissible chain, $\cS(D')_1$ cannot have any interface elements other than $D,N(D)$. Thus $\cP'$ cannot have any interface elements neighboring or diagonal to $D'$. This yields $\cP'$ can only be a pattern of type 2 and $D'\in\cS(\Ga)_2$. This contradicts to the first part of the proof of the lemma.

\begin{figure}
  \centering
\includegraphics[width=\textwidth]{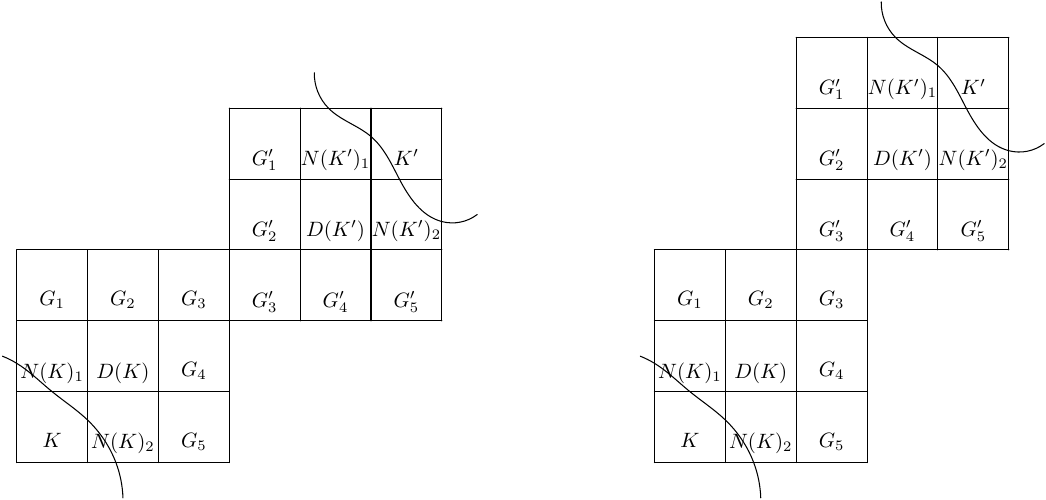}
  \caption{Illustration when $D,D'\in\cS(\Ga)_2$ in the proof of Lemma \ref{lem:A1}.}\label{fig:A1}
\end{figure}

$2^\circ$ Now let $D\in\cT_2$. By the first part of the proof, $D'\not\in\cS(\Ga)_2$. Thus there exists an interface element $K'\in\cS(D')_1$ merged with $D'$ to form the pattern $\cP'$. By the Rule 3 and 4, $K'$ cannot be in the position $X$ in Fig.\ref{fig:A2} (right). By the Rule 4, $\cS(D')_1\cap\cT^\Ga$ must be connected, thus $N(D)$ must be in $\cT_1$. By our construction of patterns, $N(D)$ can only be merged with its neighboring element(s) to form a pattern. Thus $N(D)\in\cP$. In this case, since $D\in\cT_2$, $\cP$ can only be a pattern of type 1 or 4. Neither is possible because $K'\in\cP'$. Thus $\cP'$ cannot have an element $K'\in\cT^\Ga$ neighboring to $D'$. Similarly, $\cP'$ cannot have an element $K'\in\cT^\Ga$ diagonal to $D'$. This shows that $D$ cannot be an interface element.

Similarly, $D'$ also cannot be an interface element. This completes the proof.
\end{proof}

\begin{figure}
  \centering
\includegraphics[width=0.7\textwidth]{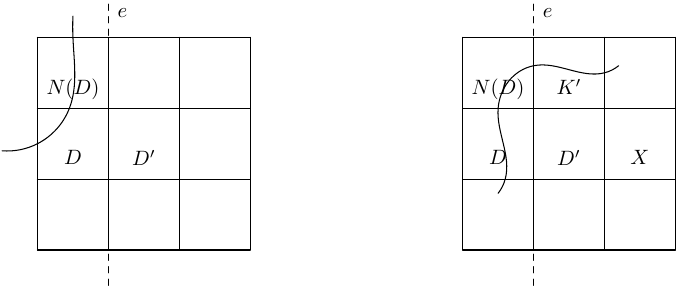}
  \caption{Illustration when $D,D'\in\cS(\Ga)_0$ in the proof of Lemma \ref{lem:A1}.}\label{fig:A2}
\end{figure}

By Lemma \ref{lem:A1}, if $(D,D')$ is subordinate to $(M,M')$ with respect to $e$ and $e\cap\Ga=\emptyset$, then $D,D'\not\in\cS(\Ga)_0\cup\cS(\Ga)_2$. By the construction, the $\cS(\Ga)_1$-elements in the patterns can (1) have no neighboring interface elements like $G_3$ in the pattern of type 1, $G_2,G_4$ in the pattern of type 2; or (2) have two neighboring interface elements like $D(K)$ in patterns of type 1 and 2, see Fig.\ref{pattern1}; or (3) have only one neighboring interface element. The following lemma rules out the first two possibilities.

\begin{lem}\label{lem:A2}
Let $e=\pa M\cap\pa M'$, where $M,M'$ are macro-elements generated by patterns $\cP,\cP'\in\{\cP_1,\cdots,\cP_5\}$, and $(D,D')$ is subordinate to $(M,M')$ with respect to $e$. If $e\cap\Ga=\emptyset$, then $D,D'\in\cS(\Ga)_1$ cannot have exactly $j$, $j=0,2$, neighboring interface elements in $\cP,\cP'$, respectively.
\end{lem}

\begin{proof} We first show the case when $j=2$. Let $D\in\cS(\Ga)_1\cap\cP$ has two neighboring interface elements. Then  $D=D(K)$ is in a pattern of type 1 or 2. First, let $D,K\in\cP_1$, see Fig.\ref{fig:A3} (left). Then by Lemma \ref{lem:A1}, $D'\not\in\cS(\Ga)_0\cup\cS(\Ga)_2$. Assume $D'\in\cS(\Ga)_1$ and $K'\in\cP'$ is an interface element. Since $e\cap\Ga=\emptyset$, $K'\not=D_1'$ in Fig.\ref{fig:A3} (left). If $K'=D_2'$, by the Rule 4, we know that $D_1'\in\cT_2, D_2'\in\cT_1$, and consequently, $D_1'$ must be merged with $D'$ by the construction of patterns. That is $D_1'\in\cP'$. This is a contradiction. Similalrly, one can show $K'$ cannot be $D_3',D_4',D_5'$. The case when $D,K\in\cP_2$ can be proved similarly. Thus the lemma is true for $j=2$.

Now we consider the case when $j=0$. Then $D=G_3$ in the pattern of type 1 or $D=G_2,G_4$ in the pattern of type 2 in Fig.\ref{pattern1}. Let $K\in\cT^\Ga\cap\cP$ be diagonal to $D$ and $K\in \cT_1$ in Fig.\ref{fig:A3} (middle and right). By the Rule 4, the interface elements in $\cS(D)_1$ are connected, $D_1',D_5'\not\in\cT^\Ga$. Thus $K'\in\cT^\Ga$ can only be one of the elements $D_2',D_3',D_4'$. If $D=G_3$ in the pattern of type 1, $S(D)_2$ cannot have interface elements other than interface elements in $\cP$. However, as $K'\in S(D)_2$ in Fig.\ref{fig:A3} (middle), this is a contradiction. If $D=G_2$, $G_4$ in the pattern of type 2, $S(D')_2$ cannot have interface elements other than interface elements in $\cP$. However, again as $K'\in S(D')_2$ in Fig.\ref{fig:A3} (right), this is a contradiction. This completes the proof.
\end{proof}

\begin{figure}
  \centering
\includegraphics[width=\textwidth]{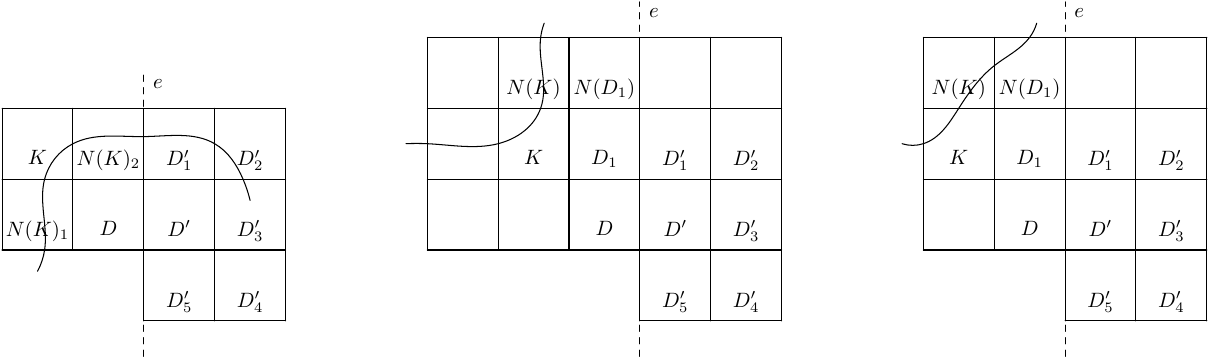}
  \caption{Illustration when $D\in\cS(\Ga)_1$ has two neighboring elements (left) and has no neighboring element (middle and right) in $\cP$ in the proof Lemma \ref{lem:A2}. }\label{fig:A3}
\end{figure}

Now we consider the case when the pair $(D,D')$ is subordinate to $(M,M')$ with respect to $e$ such that $D,D'\in\cS(\Ga)_1$ have only one neighboring interface element.

\begin{lem}\label{lem:A3}
Let $e=\pa M\cap\pa M'$, where $M,M'$ are macro-elements generated by patterns $\cP,\cP'\in\{\cP_1,\cdots,\cP_5\}$, and $(D,D')$ is subordinate to $(M,M')$ with respect to $e$. If $e\cap\Ga=\emptyset$ and $D,D'\in\cS(\Ga)_1$ such that $D=N(K), D'=N(K')$ for some $K\in\cT^\Ga\cap\cP, K'\in\cT^\Ga\cap\cP'$, then $K'\not\in\cS(D)_1$, $K\not\in\cS(D')_1$.
\end{lem}

\begin{proof} Assume that $K'\in\cS(D)_1$. Then $K,K'$ are connected in $\cS(D)_1$ by the Rule 4. There are only two possibilities as shown in Fir.\ref{fig:A4}. In the left figure, since $D_1\in\cT_1, D_2\in\cT_2$, $K$ must be in $\cT_2$ and $D_1,D_2,K$ form a pattern of type 1. This contradicts to that $e\cap\Ga=\emptyset$. In the right figure, $D'$ is neighboring to three interface elements which contradicts to the Rule 3 of the admissible chain. This completes the proof.
\end{proof}

\begin{figure}
  \centering
\includegraphics[width=0.7\textwidth]{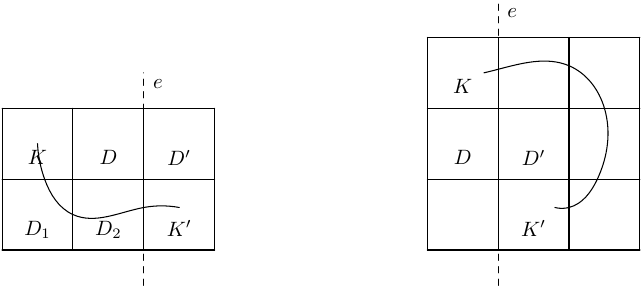}
  \caption{Illustration when $D,D'\in\cS(\Ga)_1$ has only one neighboring element in $\cP,\cP'$, respectively.}\label{fig:A4}
\end{figure}

Now we are in the position to prove Theorem \ref{thm:A1}.

\medskip

{\bf Proof of Theorem \ref{thm:A1}.} If one of $M,M'\in\cM$ is not a macro-element, since the elements in $\cT$ are obtained by locally quad refining the elements around the interface to form an admissible chain from an initial uniform rectangular mesh $\cT_0$, the theorem is obviously true. If both $M,M'$ are macro-elements, by Lemmas \ref{lem:A1}-\ref{lem:A3}, if $e\cap\Ga=\emptyset$, then $D,D'\in\cS(\Ga)_1$, $D=N(K), D'=N(K')$ for some $K\in\cT^\Ga\cap\cP$, $K'\in\cT^\Ga\cap\cP'$ such that $K'\not\in\cS(G)_1,K\not\in\cS(D')_1$. Thus $D,D',K,K'$ can only in the configuration illustrated in Fig.\ref{fig:A4} (left). This implies that if $\cP\in\{\cP_3,\cP_4,\cP_5\}$, then $e=\pa M\cap\pa M'$ can only be the common side of $D,D'$ and $f=f'=e$.

It remains the case when $\cP\in\{\cP_1,\cP_2\}$. Then $D$ can only be the elements $G_1,G_2,G_4,G_5$ in the pattern of type 1, or $G_1,G_5$ in the pattern of type 2 in Fig.\ref{pattern1}. In both cases, we denote $G=G_3$. Then $G$ is neighboring to $D$ or one element away from $D$, as shown in Fig.\ref{fig:A5} (middle and right). By the Rule 4, $\cS(G)_2$ cannot have interface elements other than the interface elements in $\cP$. However, as $K'\in\cS(G)_2$ in both figures in Fig.\ref{fig:A5} (middle and right), this is a contradiction. The theorem is now proved. $\Box$

\begin{figure}
  \centering
\includegraphics[width=\textwidth]{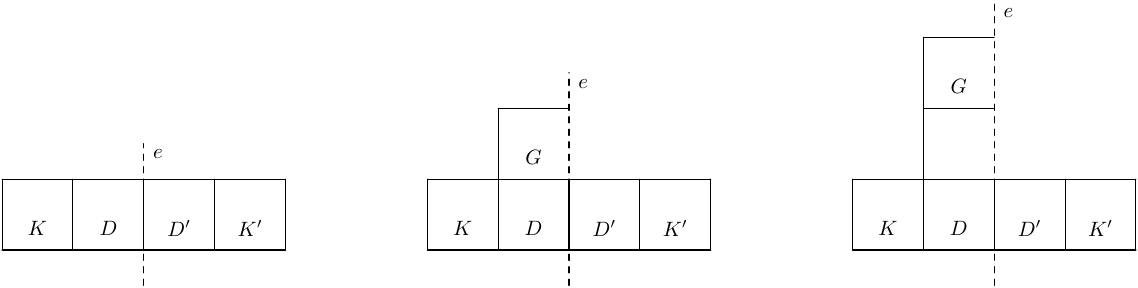}
  \caption{The only possible configuration when $D,D'\in\cS(\Ga)_1$ have only one neighboring element in $\cP,\cP'$, respectively (left). Illustration when $D$ is in a pattern of type 1 (middle) and in a pattern of type 2 (right) in the proof of Theorem \ref{thm:A1}.}\label{fig:A5}
\end{figure}

\end{document}